%% file: stef-arxiv.tex
\documentclass[11pt]{article}

\usepackage{graphicx}
\usepackage{nicefrac}  
\graphicspath{{gfx/}}
\usepackage[margin=1in,nohead]{geometry}

\usepackage[authoryear,round,compress,elide,comma]{natbib}

\input{stef-head.tex}
\usepackage{booktabs}

\begin{document}

\title{Exponential Family Trend
  Filtering on Lattices}
\author{Veeranjaneyulu Sadhanala\\
Google Research\\ New York, NY, USA \and
Robert Bassett\\
Naval Postgraduate School\\ Monterey, CA, USA \and
James Sharpnack\\
Amazon AWS\\ Santa Clara, CA, USA \and
Daniel J. McDonald\\
University of British Columbia\\
Vancouver, BC Canada}

\maketitle

\input{stef-abstract.tex}

\input{stef-body.tex}

\bibliographystyle{rss}
\bibliography{stef.bib,veeru}

\newpage
\appendix

\input{stef-appendix.tex}

\end{document}

%% file: stef-head.tex
\usepackage{graphicx}
\usepackage{amsfonts,amssymb,amsmath,amsthm}

\usepackage{pdfsync}
\usepackage{multirow}
\usepackage{enumitem}
\setlist{nosep}
\usepackage[pagebackref=true]{hyperref}
\hypersetup{colorlinks=true,citecolor=blue,linkcolor=blue,urlcolor=blue}
\usepackage{algorithm,algorithmic}

\makeatletter
\usepackage{aliascnt}
\theoremstyle{plain}

\newtheorem{theorem}{Theorem}
\newtheorem{lemma}{Lemma}
\newtheorem{proposition}{Proposition}
\newtheorem{corollary}{Corollary}[theorem]

\newtheorem{remark}{Remark}

\newcommand{\aref}[1]{\hyperref[#1]{Appendix~\ref*{#1}}}

\DeclareMathOperator*{\argmin}{argmin}

\renewcommand{\P}{\mathbb{P}}
\newcommand{\norm}[1]{\left\lVert #1 \right\rVert}
\newcommand{\snorm}[1]{\lVert #1 \rVert}
\newcommand{\R}{\mathbb{R}}
\newcommand{\E}{E}

\newcommand{\cH}{\mathcal{H}}

\newcommand{\cM}{\mathcal{M}}
\newcommand{\cS}{\mathcal{S}}
\newcommand{\cN}{\mathcal{N}}
\newcommand{\cU}{\mathcal{U}}
\newcommand{\one}{\mathbf{1}}

\newcommand{\Expect}[1]{\E\left[#1\right]}
\newcommand{\given}{\mbox{ }\vert\mbox{ }}
\renewcommand{\hat}{\widehat}
\renewcommand{\top}{\mathsf{T}}

\DeclareMathOperator*{\trace}{tr}
\DeclareMathOperator*{\Cov}{Cov}
\DeclareMathOperator*{\diag}{diag}
\renewcommand{\tilde}{\widetilde}
\newcommand{\half}{\nicefrac{1}{2}}

\usepackage{xspace}
\makeatletter
\newcommand*{\iid}{%
    \@ifnextchar{.}%
        {i.i.d.}%
        {i.i.d.\@\xspace}%
}
\makeatother

\newcommand{\Pnd}{P_{\cN(\breve{D})}}

\newcommand{\KL}[2]{\mathrm{KL}\left(#1\; \Vert\; #2\right)}
\newcommand{\KLbar}[2]{\overline{\mathrm{KL}}\left(#1\; \Vert\; #2\right)}
\newcommand{\lap}{\mathrm{Lap}}
\newcommand{\kl}{\mathrm{KL}}
\newcommand{\ktvset}{T}
\newcommand{\ktfmat}{D}

\usepackage{mathtools}
\mathtoolsset{showonlyrefs}

%% file: stef-abstract.tex
\begin{abstract}
  Trend filtering is a modern approach to nonparametric regression that
  is more adaptive to local smoothness than splines or similar basis
  procedures. Existing analyses of trend filtering focus on estimating
  a function corrupted by homoskedastic Gaussian noise, but our work extends this
  technique to general exponential family distributions. This extension
  is motivated by the need to study massive, gridded climate data
  derived from polar-orbiting satellites. We present algorithms tailored
  to large problems, theoretical results for general exponential family likelihoods, and
  principled methods for tuning parameter selection without excess
  computation.
\end{abstract}

%% file: stef-body.tex
\section{Introduction}
\label{sec:introduction}

Modeling data using exponential family distributions on the vertices of a graph
is a standard task in statistics and artificial intelligence. Examples include
satellite images or photographs, traffic or mobility patterns, communications
networks, spatiotemporal data, and many others. Suppose we observe $y_i \in \R$
for $i=1,\ldots,n$ on the nodes of a graph and assume that they independently
follow a natural exponential family with density of the form
\begin{equation}
  \label{eq:exp-fam}
  p(y_i \given \theta_i^*) = h(y_i)\exp\left\{y_i\theta_i^*
    - \psi(\theta_i^*)\right\},
\end{equation}
for functions $h: \R \rightarrow [0,\infty)$ and $\psi:
\Theta\rightarrow\R$ and natural parameter $\theta^*_i \in \Theta.$ The maximum
likelihood estimator for $\theta^*$ is easily shown to be $\psi^{\prime -1}(y)$
where we apply the function component wise. Unfortunately, this estimator fails
to respect the known graphical structure, and therefore has high estimation risk
(e.g., $\E\snorm{\psi^{\prime -1}(y) -\theta^*}^2_2 \propto n$ for the Gaussian 
family). In this paper, we imagine
that the natural parameter vector $\theta^* \in \Theta^n \subseteq \R^n$ is smooth on
the graph in a total variation sense described below.
We study methods to filter (estimate) the true parameter vector $\theta^*$, given
observations $y\in \R^n$ subject to this structure.

As an example, \autoref{fig:canada-temp} shows estimates for the instantaneous
variance (imagining $y_i$ is a member of the Gamma family) of the temperature
for New Year's Day 2010 over a grid for Canada using maximum likelihood and a few
configurations of the main family of estimators we investigate. The smoothness
imposed by the grid of neighbouring locations leads to predictable patterns in
the estimate that follow topographical features like mountain ranges and bodies
of water. We will revisit this example in more detail in
\autoref{sec:experiments}. Before describing our methodology more
carefully, we define notation.

\begin{figure}[t!]
  \centering
  \includegraphics[width=.9\textwidth]{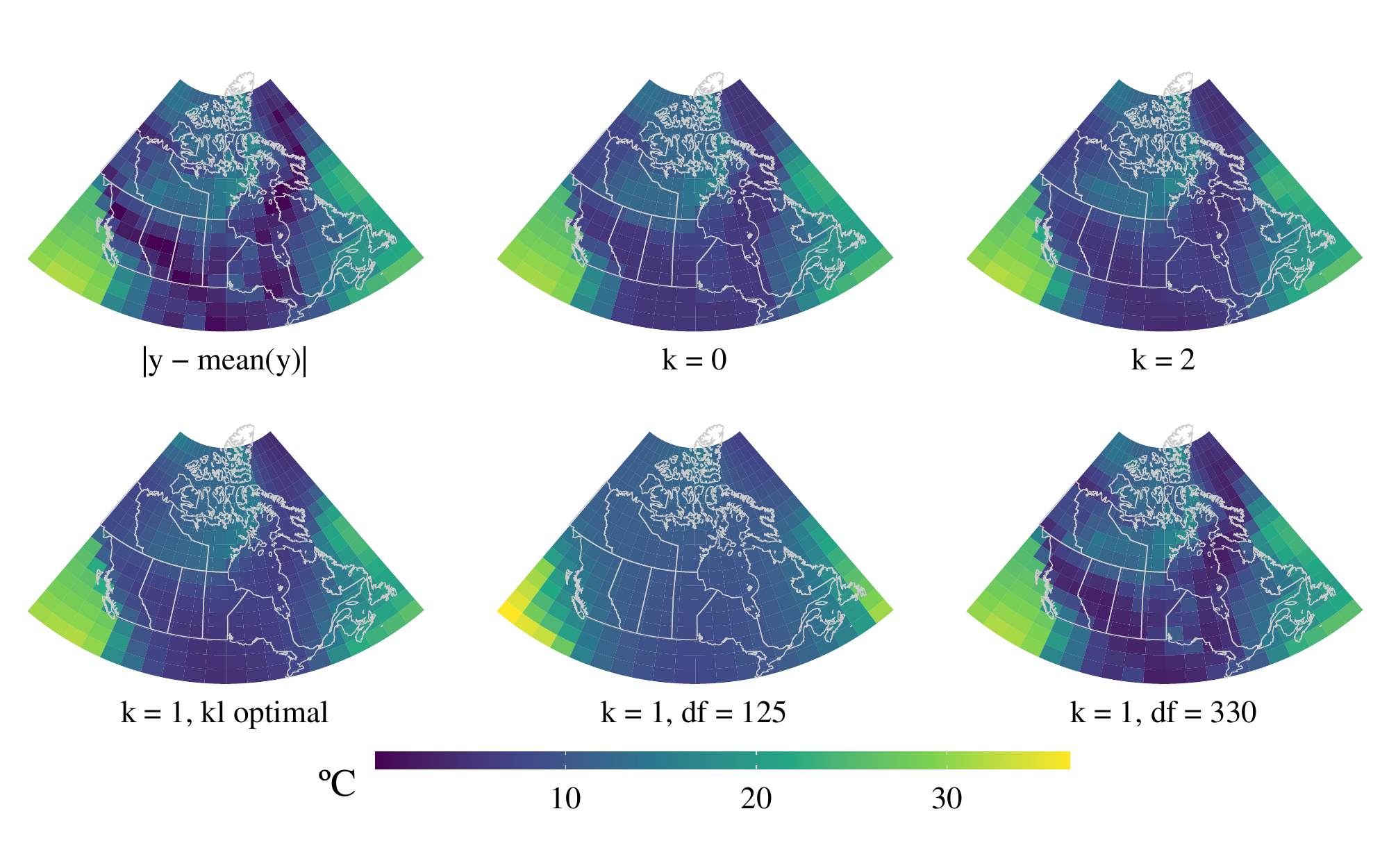}
  \caption{Estimates of the instantaneous temperature variance for 1 January
    2010 over Canada. The top row shows the absolute centered data,
    0th-order trend 
    filter, and 2nd-order trend filter, in the latter 2 cases, with
    reasonable values of the tuning parameter. The bottom row shows the
    1st-order 
    trend filter for different tuning parameters, with the left most map,
    labeled ``optimal'', 
    corresponding to the estimate when the degrees-of-freedom is chosen by
    minimizing an unbiased risk estimate.} 
  \label{fig:canada-temp}
\end{figure}

\paragraph{Notation.}

Throughout this paper, we will focus on lattice graphs in $d$ dimensions, though
we note that our main theoretical results can be extended to arbitrary graphs
with appropriate conditions on the graph-Laplacian. We define a 
 {\em graph difference operator}  $D$ that is crucial for 
defining our estimators. In one dimension, on a chain graph, 
the difference operator \smash{$D_{n, 1}^{(1)}$} is defined by 
\begin{equation}
  (D_{n, 1}^{(1)} \theta)_i = \theta_{i+1} - \theta_{i} \text{ for all }
  i\in [n-1],\ \theta \in \R^n, 
\end{equation}
where $n>1.$ We use the notation $[m]$ to denote the set 
\smash{$\{1, 2, \dots, m\}$} for positive integers $m$.
The \smash{$(k+1)^{\textrm{th}}$} order (forward) difference matrix 
\smash{$D = D_{n,1}^{(k+1)} \in \R^{(n-k-1)\times n}$}
is defined with the recurrence relation
\begin{equation}
  D_{n, 1}^{(k+1)} = D_{n-k, 1}^{(1)} D_{n, 1}^{(k)} \text{ for } k>0,\ n>k.
\end{equation}
For example, the 
\smash{$3^{\textrm{rd}}$}-order differences look like:
\smash{$(D_{n,1}^{(3)}\theta)_i = -\theta_{i+3} +
  3\theta_{i+2}-3\theta_{i+1}+\theta_i.$}
For a general graph, let $D^{(1)}$ denote its incidence matrix.
In \smash{$d>1$} dimensions, we focus on lattice graphs with a length of $N$ on 
each 
side and with a total number of vertices \smash{$n = N^d$}.
In our estimators, unless otherwise specified, we penalize the variation of 
signals only along axis-parallel directions. 
For $d$-dimensional grids, let $(k+1)$ denote the $d$-vector
$(k_1+1,\ldots,k_d+1)$, and define
\begin{equation}
  \label{eq:kron_tf_penalty2}
  D_{n,d}^{(k+1)} = \left[
    \begin{array}{c}
      D_{N,1}^{(k_1+1)} \otimes I_N \otimes \cdots \otimes I_N \\ 
      I_N \otimes D_{N,1}^{(k_2+1)} \otimes \cdots \otimes I_N \\ 
      \vdots \\
      I_N \otimes I_N \otimes \cdots \otimes D_{N,1}^{(k_d+1)}  
    \end{array}
  \right]
\end{equation}
where the Kronecker products consist of $d$ terms each, one term for each 
dimension.

Define \smash{$\snorm{\cdot}_2$} to be the usual Euclidean norm and 
\smash{$\snorm{\cdot}_n =
n^{-\half}\snorm{\cdot}_2$} to be the empirical norm. We will similarly
denote other \smash{$\ell_p$}-norms with an appropriate subscript. When there 
is no
chance of confusion, we will assume that the \smash{$\psi$} function in 
\eqref{eq:exp-fam} applies
component-wise. We use \smash{$a \otimes b$} to denote the Kronecker product of 
vectors
\smash{$a$} and \smash{$b$}, \smash{$a\odot b$} to denote the elementwise 
product, and \smash{$\langle 
a, b
\rangle = a^\top b$} to be the dot product. When clear, we
will use \smash{$f',\ f''$} to denote componentwise first and second 
derivatives of
the function \smash{$f$}. We use \smash{$\vee/\wedge$} for maximum/minimum 
respectively and
\smash{$(x)_+ = x\vee 0$; while $\one\{A\}$}
is the indicator of the event $A$, taking the value one if true and zero
otherwise. We use \smash{$a_n\lesssim b_n$} to mean \smash{$a_n \leq cb_n$} 
eventually for some
constant \smash{$c>0$}, \smash{$a_n =\Omega(b_n)$} to mean that \smash{$a_n 
\geq cb_n$} eventually, and
\smash{$Y_n = O_\P(1)$} to mean that the sequence of random 
variables is bounded in 
probability eventually. We will also use \smash{$Y_n = \tilde{O}_\P(1)$} 
to mean that \smash{$Y_n = O_\P(\log^c(n))$} for some $c > 0$.
Finally, for the graph difference operator, we will write the singular value
decomposition (SVD) of 
\smash{$D=U\Sigma V^\top \in \R^{m\times n}$} where \smash{$U\in 
\R^{m\times m}$}, \smash{$\Sigma \in \R^{m\times n}$} and \smash{$V\in \R^{n\times n}$}, 
and we
write the null-space of \smash{$D$} as \smash{$\cN = \cN(D)$}.

\subsection{Estimators}

We consider two canonical estimators. The first
filters the natural parameter $\theta^*$ based on maximizing the likelihood
while the second filters the mean $\beta^*\coloneqq\psi'(\theta^*)$ directly.
This distinction is important with respect to the nature of the expected
smoothness. If
we were to consider the data without regard for the graphical structure, then
there is a direct correspondence between these two: the MLE for $\beta^*$ is
given by applying $\psi'$ to the MLE for $\theta^*$. Furthermore, this
equivalence holds trivially for estimating the mean of a Gaussian because
$\beta^* = \theta^*$. However, any requirement for smoothness over the graph
destroys this relation for general exponential families.

\paragraph{Penalized MLE.}
We minimize negative log-likelihood with a smoothness imposing penalty:
\begin{equation}
  \label{eq:mle1}
  \hat\theta = \argmin_\theta \frac{1}{n} \sum_{i=1}^n- y_i\theta_i + 
\psi(\theta_i) + \lambda\norm{D\theta}_1.
\end{equation}
Here $\lambda$ is a parameter for balancing fidelity to any anticipated
smoothness over the graph, as encoded by $D$, with fit to the data $y$. Taking
$\lambda \rightarrow 0$ will result in the minimum occurring at $\hat\theta =
\psi^{\prime -1}(y)$ 
while letting $\lambda\rightarrow\infty$ gives the Kullback-Leibler projection
of $y$ on to $\mathcal{N}(D).$

By the likelihood principle, $\hat\theta$ is the natural estimator to use when 
we expect
that $\theta^*$ is smooth with respect to the graph. However, as we will 
demonstrate,
this estimator can have high excess estimation risk when $\psi''(\theta^*)$
approaches $0$. In \autoref{sec:mle_null_space_penalty} we will argue that this issue 
can be addressed by adding a 
penalty on the null-space component of $\theta$.
Specifically, the MLE with TF and null space penalty is 
\begin{equation}
	\label{eq:estimator}
	\hat\theta = \argmin_{\theta} \; \frac{1}{n}\sum_{i=1}^n
	-y_i \theta_i + \psi(\theta_i) +  \lambda_1 \| D\theta\|_1 +
	\lambda_2 \| P_\cN \theta\|_2
\end{equation}
where $\lambda_1, \lambda_2 \geq 0$ are regularization parameters and $P_\cN$ is
the projection operator on to $\cN(D)$.

\paragraph{Mean Trend Filter.}
When the expected smoothness is in the mean rather than the natural parameter, 
it may be more appropriate to penalize the roughness in mean directly.
For such a scenario,
we consider the trend filtering estimator:
\begin{equation}
\label{eq:lsq}
  \hat \beta = \argmin_\beta  \; \frac{1}{2n} \norm{ y - \beta }_2^2 +
  \lambda \norm{D\beta}_1.
\end{equation}
As before, $\lambda$ balances data fidelity with smoothness, but here, the
interpretation as $\lambda\rightarrow\infty$ is more straightforward. In this
case, the
minimum occurs at the orthogonal projection onto the null space of $D$:
$\hat\beta = (I - D^\top(DD^\top)^{-1}D)y$.
This estimator was proposed in \cite{steidl2006splines}, \cite{KimKoh2009} and
statistically analyzed in \cite{tibshirani2014adaptive},
\cite{WangSharpnack2016} and others. We provide a thorough overview of previous
work on mean trend filtering in a later section.

To understand the nature of the penalty in the above formulations,
it is clearly important to understand its null space.  \cite{SadhanalaWang2017}
showed 
that the null space of $D$ consists of Kronecker products of polynomials. We
give a  
generalized version of their Lemma 1 here.
\begin{lemma}
  A basis for the null space of $D$ is given by the family of polynomials
  \[
   \left\{p(x) = x_1^{a_1} \otimes x_2^{a_2} \otimes \cdots \otimes x_d^{a_d} :
     a_j \in \{0,\ldots,k_j\}\right\}
  \]
  where $x_j$ are the coordinates of the observations along the
  $j^{th}$ dimension. The dimension of the null
  space is $\mathrm{nullity}(D) = \prod_{j=1}^d (k_j+1)$.
\end{lemma}
Therefore, writing $P$ as the matrix formed by the evaluations of this 
collection of polynomials over the grid,
we can  
also write the Euclidean projection onto the null space of $D$ as $P_\cN := P(P^\top
P)^{-1}P^\top$.  
When applied to certain kinds of data (for example the satellite temperature
data) it may be useful to imagine that some dimensions of the grid ``wrap'' like
a cylinder. If the grid wraps along some dimension $j \in [d]$,
then $a_j = 0$ regardless of $k_j$ and the contribution to the nullity for
dimension $j$ is as if $k_j=0$.

Characterizing the null space tells us the sorts of vectors $\theta^*$ that have $\snorm{
D \theta^* }_1 = 0$, but it does not say anything about vectors with bounded trend filtering
penalty. 
Consider the $\ell_0$ penalty instead, $\snorm{ D \theta }_0$, for $k_1 =
\cdots=k_d = k$. This is small when there are few changepoints, which are the 
indices $j_1, \ldots, j_M$ at which the $k^{\textrm{th}}$ derivative is non-zero, $(D
\theta^*)_{j_1,\ldots,j_M} \ne 0$. 
Because the $\ell_1$ penalty tends to produce sparse vectors with small $\snorm{D \hat
\theta }_0$, the reconstructed signals are piecewise polynomials with a few
changepoints that are automatically selected. 
The result is that trend filtering produces estimators that are locally
adaptive, which means that the reconstructed signal is not oversmooth in regions
of high signal variability (in $\theta^*$) and not undersmooth in regions of low
variability. 
In short the filter does not have one fixed resolution or bandwidth, but adapts
the resolution to the observed signal. 
For a more complete explanation of this phenomenon, see 
\cite{WangSharpnack2016,Bassett2019fused}. To simplify the theoretical exposition below,
we will assume that $k_1=\cdots=k_d=k$, but our results are
easily modified for other situations.


\subsection{Properties of exponential families}
\label{sec:prop-expon-famil}

In this section, we review properties of exponential families, many of which
will play a key role in our theoretical development. Considering the univariate
random variable $Y$ with density of the form in
\eqref{eq:exp-fam}, we define the
domain $\Theta = \left\{ \theta\in \R : \psi(\theta)<\infty\right\}$ 
and assume that
$\Theta$ has a non-empty interior.  
Recall that the mean and variance of the distributions $p(\cdot \given \theta)$
are $\psi'(\theta)$ and $\psi''(\theta)$ respectively, for natural parameter
$\theta\in\Theta$. Therefore, $\epsilon\coloneqq Y - \psi'(\theta^*)$ has mean
zero and
a simple expression for its moment generating function (MGF)
\begin{equation}
  \label{eq:exp-fam-mgf-centered}
  \Expect{e^{s\epsilon}} =  \exp
  \left\{\psi(\theta^* + s) - \psi(\theta^*)-s \psi'(\theta^*)\right\}
\end{equation}
for $s$ in a neighborhood of 0. Furthermore,
$\psi$ is convex and all its derivatives exist for all $\theta\in 
\Theta$ (see~\citealt{Brown1986}).

We say that a random variable \smash{$X$} with mean \smash{$0$} is 
\emph{sub-exponential} if 
there are
non-negative parameters \smash{$\nu,\ b$} such that 
$$\Expect{\exp\{t X\}} \leq \exp\{\nu^{2} t^{2} / 2\} \quad
\text{ for all }\quad |t| < 1 / b.$$ 
For shorthand, we also say $X$ is SE\smash{$(\nu^2, b)$}. We
can show that random variables following exponential family distributions are 
sub-exponential in this sense.

\begin{lemma}
\label{lem:sub-exponential}
Fix $\theta^{*}$ in the $\mathrm{interior}(\Theta)$, and let $Y$ be from a 
univariate
exponential family with parameter $\theta^{*}$. Then for any $\delta > 0$, $Y - \psi'(\theta^*)$ is
sub-exponential with some parameters $\nu$ and $b$ depending on $\theta^*$ and $\delta$.
Specifically, $\nu$ is related to the variance by 
\smash{$\nu^2 = 
\psi''(\theta^*) + \delta$}. 
\end{lemma}
\autoref{tab:orlicz_subexp_params} gives the log-partition function
$\psi(\theta)$ and sub-exponential parameters for 
Poisson,
exponential, and chi-squared families. These calculations and the proof of
\autoref{lem:sub-exponential} are in
\autoref{sec:app-preliminary}.  In each of the examples in
\autoref{tab:orlicz_subexp_params}, \smash{$\nu^2$} is selected to be a 
multiple of the
variance, but these are not the only choices of \smash{$(\nu,\ b)$} that would 
constitute
valid sub-exponential parameters. 
\autoref{lem:sub-exponential} is not surprising given the form of 
the MGF, but seems not to be well-known. Related results can be seen
in \cite{Brown1986} or \cite{KakadeShamir2010}. Note that many exponential
families have tails which decay faster (e.g., Gaussian or Binomial
distributions), but all exponential families have sub-exponential tails.

Finally, we note that in all of these examples (Poisson, exponential, 
chi-square) the variance, and hence the curvature of $\psi(\theta^*)$ depends 
on $\theta^*$, resulting in heteroskedasticity.
This is one of the main complications of the exponential family setting that we consider in this paper.
Along with the heavy-tailed residuals, this setting is a major departure from 
the sub-Gaussian homoskedastic setting of most prior works.

\begin{table}
	\caption{\label{tab:orlicz_subexp_params} Sub-exponential parameters 
		for some exponential family distributions}
	\centering
	\begin{tabular}{lccc}
		\toprule
		Distribution & $\psi(\theta)$ 
		& $\nu^2, b$\\
		\midrule
		Poisson (mean$=\mu$) & $e^\theta$ 
		& $2\mu$, \; $0.55$ \\ 
		Exponential (mean$=\mu$) & $-\log(-\theta)$ 
		& $4\mu^2\log\frac{4}{e} $,\; $2\mu$ \\
		$\chi^2_k$ (mean$=k$) & $\log \big(\Gamma(\theta+1) 2^{\theta+1}\big)$ 
		& $4k, 4$ \\
		\bottomrule
	\end{tabular}
\end{table}


\paragraph{KL divergence.}

The Kullback-Leibler (KL) divergence between exponential distributions of the
same family has a simple algebraic form in terms of $\psi$; see
\cite{wainwright2008graphical}. 
 The KL divergence with parameter vectors $\theta_0$
and $\theta_1 \in \R^n$ is 
\begin{equation}
  \KL{\theta_0}{\theta_1} := \int p(y \given \theta_0)
  \log\frac{p(y \given \theta_0)}{p(y \given \theta_1)}
  dy.\label{eq:kl-def}
\end{equation}
In the asymptotic setting with $n\rightarrow\infty$, it makes more
sense to examine the average divergence per coordinate. Thus we define
$  \KLbar{\theta_0}{\theta_1} := \frac{1}{n}\KL{\theta_0}{\theta_1}.$
For an exponential family as in \eqref{eq:exp-fam}, the KL
divergence is the Bregman divergence of $\psi$
\begin{equation}
  \label{eq:exp-fam-kl}
  \KL{\theta_0}{\theta_1} = \psi(\theta_1) - \psi(\theta_0) -
  (\theta_1-\theta_0)^\top \psi'(\theta_0).
\end{equation}

\subsection{Summary of our contributions}

Most of the existing work on trend filtering referenced above assumes
sub-Gaussian noise, that is,  
$$
y_i = \beta_i + \epsilon_i,
$$
for $i\in [n]$ where $\epsilon_i$ is mean-zero and sub-Gaussian with common variance $\sigma^2$.
For general exponential families of the form in  \eqref{eq:exp-fam}, 
$y_i - \E y_i$ has heavier than sub-Gaussian tails.
Furthermore, for general exponential
families, the variance, as well as higher moments, are tied to the mean parameter.
Therefore, consideration of heteroskedasticity is a
necessary and fundamental component of our analysis. 

Direct analysis for specific exponential families, such as Poisson
\citep{Bassett2019fused} are rare.
\Citet{vandeGeer2020logistic} analyses a penalized MLE for the logistic family.
However, the logistic family has sub-Gaussian tails and uniformly bounded
variance which allows key parts of the
analysis, such as the Dudley entropy integral bound, to work. In other words, the theoretical
approach there cannot generalize to arbitrary exponential families.

Our results here
apply to the entire exponential family. However, due to this generality, the results are
necessarily weaker than could potentially be achieved under additional, 
more stringent conditions (such as by assuming Gaussian or logistic
distributions, or requiring additional bounds on higher moments).  

A key ingredient in previous analyses in the sub-Gaussian setting is that the Bregman
divergence  
$\psi(\hat\theta) - \psi(\theta^*) - (\hat\theta 
- \theta^*)^\top \psi'(\theta^*),$ 
can be lower bounded by a multiple of $ \| \hat\theta - \theta^* \|_2^2$, 
because $\psi$ is strongly convex. 
However, for general exponential families, $\psi$ is not strongly convex, even if 
$\snorm{D\theta^*}_1$ is well-controlled, unless
$\theta^*$ satisfies additional conditions.
Without such assumptions, $\psi''(\theta^*)$ can be arbitrarily small.
If we make the (rather implausible) assumption that both the estimate
$\hat\theta$ and the parameter $\theta^*$ are bounded, 
then we recover this strong convexity in the relevant region where $\hat\theta$
and $\theta^*$ 
lie. In this case, we can apply the same techniques used to analyze the sub-Gaussian case. 
We derive these bounds in \aref{sec:app_mle_constrained}.
However, without such an assumption, analysis requires entirely different
techniques, and we show these results in \autoref{sec:error-bounds-penal}. 

Our main contributions are the following.
\begin{enumerate}
\item We derive error bounds on excess KL-risk for the penalized maximum
  likelihood 
  estimator for general exponential families with subexponential noise
  (\autoref{sec:theory_mle}). 
  We argue that there is a need to constrain the component of the
  natural parameter vector that falls in the null space of $D$ as in equation
  \eqref{eq:estimator}. 
  
\item We delineate two types of heteroskedasticity that are relevant under
  general assumptions: strong heteroskedasticity and mild heteroskedasticity. We
  show how our general KL-bounds behave under these regimes and how the
  heteroskedasticity interacts with the smoothness constraints and the
  dimensionality of the problem.

\item For $k=0$, we show that the mean trend filter and the MLE with penalty are
  equivalent estimators, and hence, results for the mean trend filter apply immediately in
  this special case (though under different smoothness assumptions;
  \autoref{sec:theory_mean_tf}).   

\item We show that the mean trend filter nearly achieves the minimax optimal rate
  under squared error loss for mildly heteroskedastic data and 
  all smoothness levels $k$ and lattice dimensions $d$
  (\autoref{sec:theory_mean_tf}). This result in fact holds 
  for general sub-exponential noise $\epsilon$, not just for the exponential
  families we consider in the paper. We incur an additional $\log n$ factor in
  the error bound for sub-exponential noise.  It is specific to distributions
  where the mean parameter has bounded trend filtering penalty.
  
\item We give an algorithm for solving all of these cases for arbitrary likelihood,
  smoothness levels, and dimension, with the goal of operating on large data
  (\autoref{sec:algor-impl}). 

\item We give a simple estimator for the out-of-sample prediction risk (at the
    original grid locations) to enable tuning parameter selection without
    requiring complicated forms of cross validation or other re-estimation
    procedures (\autoref{sec:tuning-param-select}). 
\end{enumerate}

It is important to note that the results for MLE trend filtering and mean trend
filtering are not directly comparable because they make different assumptions.
The former constrains the natural parameter, while the latter constrains the
mean parameter.
These only coincide in the Gaussian case.
We present empirical results demonstrating our 
methods on synthetic and real datasets in \autoref{sec:experiments}.
We conclude with a discussion of the results. The remainder of this
section gives a concise overview of our theoretical contributions and a thorough
discussion of related work.

\subsection{Overview of theoretical contributions}
\label{sec:theory-overview}


To better fix the context for our results, we provide here a concise description
of these in the simplest cases (more precise statements are in
Sections \ref{sec:theory_mle} and \ref{sec:theory_mean_tf}). 
Define $\alpha = (k+1)/d$, and define the ``canonical scaling'' as $\| D
\theta^*\|_1 \lesssim n^{1 - \alpha}$. 
The canonical scaling is called such because it holds for evaluations of 
H\"older functions---functions where the $k$th order partial derivatives are
Lipschitz continuous---at the grid locations.
Under the canonical scaling, it is shown \citep{sadhanala2021multivariate}
that for Gaussian data and $\ell_2$
loss, the minimax rate over this class is given by
\begin{equation}
  \label{eq:minimax-rate-gaussian}
  \textrm{MSE}(\Theta) =
  \begin{cases}\Omega(n^{-\alpha}) & \alpha \leq 1/2,\\
    \Omega\big(n^{-\frac{2\alpha}{2\alpha + 1}}\big) & \alpha > 1/2,\end{cases}
\end{equation}
where $\textrm{MSE}(\Theta) =
\inf_{\hat\theta}\sup_{\theta\in\Theta}\frac{1}{n}E{\snorm{\hat\theta -
    \theta}_2^2}$. 
Furthermore, the mean trend filter is rate optimal up to logarithmic
factors in the Gaussian case.

Because, for Gaussian data, $\KL{\theta_0}{\theta_1} \propto
\snorm{\theta_0 - \theta_1}_2^2$, the above immediately provides a lower bound for
$\overline{\mathrm{KL}}(\Theta) =
\inf_{\hat\theta}\sup_{\theta\in\Theta}\frac{1}{n}\KL{\theta}{\hat{\theta}}$
across all exponential families. We show that, under additional
boundedness conditions on $\Theta$ and similar constraints on the estimator,
the MLE trend filter in equation \eqref{eq:mle1} achieves this rate up
to additional logarithmic factors. 
The case of the MLE trend filter without
the artificial boundedness constraint described above is more complicated
(\autoref{sec:error-bounds-penal}). With the additional penalty on the 
null space in Equation \eqref{eq:estimator}, an addition we prove necessary for consistency, we
can achieve the minimax rate for $\alpha \leq  
1/2$. For $\alpha > 1/2$, the upper bound is weaker than for Gaussian noise: we
can show only that \smash{$\mathrm{KL}(\theta^* \Vert \hat\theta) = 
O_\P(n^{-\half})$}. 
While we are
able to show consistency in this setting, 
we suspect that this bound is loose.

We also show that, under homoskedastic subexponential noise, the mean trend filter
achieves the minimax rate up to additional logarithmic factors. The homoskedasticity
condition can be relaxed, and this is examined in \autoref{sec:theory_mean_tf}.
We consolidate these results in Tables \ref{tab:theory-mltf} and
\ref{tab:theory-mtf}.

\begin{table}
  \caption{\label{tab:theory-mltf}Overview of theoretical results
    for the
    Penalized MLE under canonical 
    scaling. Logarithmic factors are ignored with $\tilde O$ notation, and
     additional details are
    described in \autoref{sec:theory_mle}.}
  \begin{tabular}{@{}lcccc@{}}
    \toprule
    Conditions & Regime & Lower bound & Upper bound & Literature\\
    \midrule
    \multirow{2}*{Gaussian} & $\alpha \leq 1/2$ & $\Omega(n^{-\alpha})$ & 
    $\tilde O_\P(n^{-\alpha})$ & 
    \multirow{2}*{\citet{sadhanala2021multivariate}} \\ 
               & $\alpha > 1/2$ & $\Omega(n^{-\frac{2\alpha}{2\alpha + 1}})$ & 
               $\tilde O_\P(n^{-\frac{2\alpha}{2\alpha + 1}})$ &  \\
    \midrule
    Exponential family & $\alpha \leq 1/2$ & $\Omega(n^{-\alpha})$ & 
    $\tilde O_\P(n^{-\alpha})$ & \multirow{2}*{\autoref{prop:mle_constrained}} 
    \\
    (bounded)           & $\alpha > 1/2$ & 
    $\Omega(n^{-\frac{2\alpha}{2\alpha + 1}})$ & 
    $\tilde O_\P(n^{-\frac{2\alpha}{2\alpha + 1}})$ & \\
    \midrule
    Exponential family & $\alpha \leq 1/2$ & $\Omega(n^{-\alpha})$ & 
    $\tilde O_\P(n^{-\alpha})$ & \multirow{2}*{\autoref{cor:MLE_Grids}} \\
    (null-space penalty)  & $\alpha > 1/2$ &  
    $\Omega(n^{-\frac{2\alpha}{2\alpha + 1}})$ & $\tilde O_\P(n^{-\half})$ & \\  
    \bottomrule
  \end{tabular}
  
\end{table}

\begin{table}
  \caption{\label{tab:theory-mtf}
    Overview of theoretical results for the Mean Trend Filter under canonical
    scaling. Logarithmic factors are ignored with $\tilde O$ notation, and
    additional details are
    described in \autoref{sec:theory_mean_tf}.}
  \centering
  \begin{tabular}{@{}lcccc@{}}
    \toprule
    Conditions & Regime & Lower bound & Upper bound & Literature\\
    \bottomrule
    \multirow{2}*{Gaussian} & $\alpha \leq 1/2$ & $\Omega(n^{-\alpha})$ & 
    $\tilde O_\P(n^{-\alpha})$ & 
    \multirow{2}*{\citet{sadhanala2021multivariate}} \\
               & $\alpha > 1/2$ & $\Omega(n^{-\frac{2\alpha}{2\alpha + 1}})$ & 
               $\tilde O_\P(n^{-\frac{2\alpha}{2\alpha + 1}})$ &  \\
    \midrule
    Sub-exponential noise & $\alpha \leq 1/2$ & $\Omega(n^{-\alpha})$ & 
    $\tilde O_\P(n^{-\alpha})$ & \multirow{2}*{\autoref{cor:lsq_homosked}} \\
    (mild heteroskedasticity)  & $\alpha > 1/2$ & 
    $\Omega(n^{-\frac{2\alpha}{2\alpha + 1}})$ & 
    $\tilde O_\P(n^{-\frac{2\alpha}{2\alpha + 1}})$ & \\
    \midrule
    Sub-exponential noise & $d = 2,\ k= 0$ & $\Omega(1)$ & not consistent & 
    \autoref{prop:lower_bd_2dgrids} \\  
    \bottomrule
  \end{tabular}
\end{table}

\subsection{Related work}

Much is known about trend filtering in one dimension (1d). The trend filtering method
in \eqref{eq:lsq} was proposed in \citet{steidl2006splines,KimKoh2009} for 1d problems. 
\cite{tibshirani2014adaptive} connected trend filtering to
locally adaptive regression splines, proposed in \cite{mammen1997locally}, and
analyzed its statistical properties. \cite{tibshirani2020divided} gives an
in-depth background of the key ideas that make trend filtering and related
methods work.
\cite{johnson2013dynamic,KimKoh2009,RamdasTibshirani2016} propose
methods to solve the convex optimization problem in 1d trend filtering.
Trend filtering with $k=0$, or total variation (TV) regularization, is an important technique 
for denoising images (two dimensions). TV methodology and computation was studied in
\citet{rudin1992nonlinear,tibshirani2005sparsity,condat2012direct,barbero2014modular}.
Trend filtering over general graphs was first proposed in \cite{WangSharpnack2016}, and subsequently,
other variants of trend filtering have been studied, for example
depth-first search TV regularization \citep{padilla2018dfs}, kNN TV denoising \citep{madrid2020adaptive}, 
quantile trend filtering \citep{padilla2021risk}, and sequential TV denoising \citep{baby2021optimal}.
These methods use squared error loss, with the exception of \cite{padilla2021risk}, and so are not 
necessarily suitable for general exponential families.

General exponential family distributions have a long history in statistics.
\citet{Brown1986} is a definitive treatment for studying the properties of
exponential families while \cite{mccullagh1989generalized} covers the details of
generalized linear models.
Direct analysis of trend filtering in this setting is more rare than for
Gaussian loss. \Citet{vandeGeer2020logistic} derived error bounds for estimating
Bernoulli family parameters with bounded variation in 1d. In contrast to most
other results, the theory applies without assuming boundedness of the estimated natural parameter.
\cite{KhodadadiMcDonald2018} examine computational approaches for variance
estimation on spatiotemporal grids. \cite{KakadeShamir2010} discuss strong
convexity of general exponential families and use the results to analyze $\ell_1$
penalized maximum likelihood. \cite{VaiterDeledalle2017} examine the geometry of
penalized generalized linear models and derive important results for
general regularizers that we use for specialized risk estimation in
\autoref{sec:algor-impl}. \cite{Bassett2019fused} provides a bound on the
Hellinger error for total variation denoising for the estimation of densities
over edge segments in a general graph. Our results here are the first to analyze
trend filtering over lattice graphs for general exponential families.

An important distinction exists between two varieties of theoretical results for
trend filtering examined in the literature: (1) nearly parametric rates under
sparsity assumptions with $\| D \theta^*\|_0$ bounded; and (2) non-parametric
rates for signals with bounded trend filtering norm $\| D \theta^*\|_1$. In
general, these bounds are difficult to compare because they hold under different
conditions, and either bound can be tighter for specific signals.
\cite{rinaldo2009properties,harchaoui2010multiple,lin2017sharp,
  guntuboyina2017adaptive,ortelli2019prediction} give more general and tighter
error bounds when the true signal is sparse (bounded $L_0$ norm).
Throughout this work, we will focus on establishing non-parametric rates 
with trend filtering norm bounds.

\cite{mammen1997locally} provide one of the earliest theoretical results on 1d trend filtering.
In higher dimensions and on general graphs, researchers have typically confined their theory to special cases---e.g., specific
dimensions, graph structure, and trend filtering order.
\cite{hutter2016optimal,sadhanala2016total} derive error bounds for total
variation denoising (trend filtering with $k=0$) on lattice graphs.
\cite{chatterjee2021newrisk,ortelli2020adaptive} show stronger error bounds when
the signal has axis-parallel patches. \cite{SadhanalaWang2017,
  sadhanala2021multivariate}, extend the analysis to higher-order trend
filtering on lattice graphs of arbitrary dimension.
All of the aforementioned works study squared error loss with sub-Gaussian noise.
\citet{WangSharpnack2016} analyze error bounds for graph trend filtering for specific cases
(lattice graphs with a specific trend filtering order). In that work, the
``eigenvector incoherence'' technique is developed as a tool to analyze the
mean squared error of any graph trend filtering problem.
In this work, we adapt this technique to work with general exponential families.

\section{Penalized MLE}
\label{sec:theory_mle}

In this section, we provide general results for trend filtering on $d$-dimensional
lattice graphs with exponential family observations. As mentioned above, general
exponential families have two interesting features. First, the distributions can
be more heavy tailed than Gaussians, and as we have seen, they are generally
sub-exponential. This is reflected in rates that are typically worse than in the
Gaussian case. Second, the variance (as well as the sub-exponential parameters
$\nu,\ b$) is a function of the natural parameter, which results in
heteroskedasticity. We find that our bounds rely heavily on the ``level'' of
this heteroskedasticity. However, this reliance is most salient with respect to
two asymptotic regimes.

We say {\em mild heteroskedasticity} occurs when both subexponential parameters,
$\nu,\ b$, are bounded as $n$ increases.   Henceforth, let $\nu,\ b$ denote the
vectors $(\nu_i), (b_i)$ for $i \in [n]$ where these are the sub-exponential
parameters of centered $Y$.  That is, if there exists an $\omega$ such 
that $\| \nu \|_\infty, \| b\|_\infty \le \omega$ for all $n$, we say that the problem is
only mildly heteroskedastic. Analysis in this case turns out to be largely
similar to the standard homoskedastic setting. We say that {\em strong
  heteroskedasticity} occurs whenever it is not {\em mild}, however, typically
we can measure the strength via $\|\nu\|_\infty / \|\nu \|_n$. When this is
close to $1$, there is little variation of $\nu$ across coordinates. 
However,
when $\|\nu\|_\infty / \|\nu\|_n$ is close to $\sqrt{n}$, only a few coordinates
dominate. Importantly, smoothness of $\theta^*$ (such as a bound on 
\smash{$\norm{D\theta^*}_1$}) does not generally have
any implications for the level of heteroskedasticity, and furthermore, it is not
generally possible to determine the level from data. Thus, considering both
situations is necessary for a complete understanding.

Much of the difficulty for both estimation and theoretical analysis in the
exponential family setting is that the negative log-likelihood is not strongly convex in
general. 
If we assume that \smash{$\psi''(\theta^*_i) > 1/K$} for all $i$, then we can 
add this
constraint to \eqref{eq:mle1} which will ensure strong convexity.
We provide an analysis of this approach in
\aref{sec:app_mle_constrained}, which is tight in the Gaussian case up to
logarithm factors, see, for example, \cite{sadhanala2021multivariate}. 
Similar results were already derived in the literature, for example, in \cite{prasad2020robust}.
As we will see, however, bounding the curvature in this way excludes important cases, and
cannot be verified from data.
Nonetheless, this assumption has a long history in statistics. For example, the standard approach
to proving estimation consistency in low-dimensional generalized linear models is much the
same~\citep{mccullagh1989generalized}. 

\subsection{Additional penalty on the null space component of 
\smash{$\theta$}}
\label{sec:mle_null_space_penalty}

The boundedness constraint discussed above
is not desirable for at least two reasons. The first is that it is difficult to
calibrate the constraint using data. The second is that strong convexity is an
indirect way to get control of the nullspace of $D$, which is what we actually
need. We now argue why this is the case.

Let the empirical and population risks at a parameter $\theta$ be
\begin{align}
	\label{eq:R_defn}
	R_n(\theta) &= \frac{1}{n} \sum_{i=1}^n \psi(\theta_i) - y_i \theta_i,&\textrm{and}&&
	R(\theta) &= \frac{1}{n} \sum_{i=1}^n \psi(\theta_i) - E[Y_i] \theta_i,
\end{align}
respectively, and note that
	$\KLbar{\theta_0}{\theta_1} = R(\theta_1) - R(\theta_0).$
For Gaussian data, minimization of the empirical risk, the
$\snorm{D\theta}_1$ constraint, and strong convexity of the likelihood together control the
discrepancy between the empirical risk and the population risk. The reason
is that strong convexity controls behaviour of $\hat\theta$ in the nullspace of
$D$. But outside this setting, we no longer have strong convexity, and
unfortunately, the penalty alone does not give sufficient control.
The result is that, for non-Gaussian data,
$
\sup_{\theta \in \Theta} |R_n(\theta) - R(\theta)| 
$
can become arbitrarily large with high probability,
even in simple settings, despite bounds on $\snorm{D\theta}_1$.
Suppose \smash{$\Theta  = \{ \theta : \|D\theta\|_1 \le 1\}$} where \smash{$D = 
D_{n, 1}^{(0)}$}.

\begin{remark}[Degenerate Poisson example]
	\label{prop:pois_degen}
	
	Consider the Poisson family, with true parameter $\theta^*_n = -2\log n \one$
	for any $n\geq 1.$  
	The probability that all $y_i$'s are $0$ is $e^{-\nicefrac{1}{n}}$. 
	On this event (where $y =0 \one$), for any $\lambda \geq 0$,
	$\inf_{\theta} \sum_{i=1}^n e^{\theta_i} -y_i\theta_i+ \lambda \|D\theta\|_1 
	= 
	\inf_{\theta} \sum_{i=1}^n e^{\theta_i} + \lambda \|D\theta\|_1 = 0$
	because
	$\lim_{c\rightarrow -\infty} \sum_{i=1}^n e^{c} + \lambda \|Dc \one \|_1  = 
	0$.
	Furthermore,  observe that as $c\rightarrow -\infty$,
	\[
	R(c\one) \rightarrow \infty \text{ even though } R_n(c\one) \rightarrow 0.
	\]
  Notice that in this example, $\psi''(\theta^*_i) = n^{-2}$, so the strong 
  convexity bound is diminishing with $n$.
\end{remark}

One can observe similar behaviour for the logistic family.
Consider $\theta_n^* = -2\log n \mathbf{1}$ and verify that all $y_i$'s are $0$ 
with probability $\big(1 + n^{-2} \big)^{-n} \approx e^{-\nicefrac{1}{n}}.$ The MLE 
with 
only the $\|D\theta\|_1$ penalty behaves similarly to the Poisson example 
described above.
%

While artificially imposing strong convexity addresses this issue, it
is both more direct and results in a more tractable estimator to 
constrain the
component of \smash{$\theta$} in the null space of \smash{$D$}.
With this additional constraint, we can show the following risk bound. 
The proof is in \aref{sec:app-uniform_risk_null_penalty}.
\begin{proposition}
	\label{prop:uniform_risk_null_penalty}
	Let $\Theta = \{ \theta \in \R^n : \|P_\cN \theta \|_n \le a_n,\ \|D\theta\|_1
	\le c_n n^{1-\alpha}\}$ where  
	$\cN ={\mathrm{null}}(D)$  and $\alpha = (k+1)/d.$
	Suppose $\epsilon_i$ is zero mean
	sub-exponential with parameters $(\nu_i^2,\ b_i)$ for $i\in [n]$.  
	Assume $\|\nu\|_\infty ,\ \| b \|_\infty \leq  c$ where $c$ is a constant. Then
	\begin{equation}
		\label{eq:uniform_risk}
		\sup_{\theta \in \Theta} |R_n(\theta) - R(\theta)| =  
		O_\P \left( \frac{a_n\log n}{ \sqrt{n}}+ \frac{c_n\gamma\log n}
    {n^{\alpha\wedge\half}}\right) 
	\end{equation}
	where $\gamma = \log^{\half} n$ if $2\alpha=1$ and $1$ otherwise.
\end{proposition}
For the above example of degenerate Poisson, we can set $a_n = 2 \log n,\ c_n =
0$ to see that the right hand side converges to 0 as $n \rightarrow \infty$. 
This motivates us to penalize the null space component of $\theta$ in the MLE
and use the estimator defined in \eqref{eq:estimator} rather than that in
\eqref{eq:mle1}. 
In the following, we call this estimator \eqref{eq:estimator}, the MLE and define $\alpha=(k+1)/d$.
The minimizer in the optimization problem is unique because $\psi$ is strictly
convex. 

\subsection{Error bounds for penalized MLE}
\label{sec:error-bounds-penal}

Generally, there are three degrees of freedom when stating results: (1) the 
trend 
filtering order $k$, (2) the dimension $d$, and (3) the exponential family and  
resulting sub-exponential parameters ($\nu,\ b$).
There is a natural trade-off between generality and interpretability of the 
results presented here, so we will prefer to present specific interpretable 
results as corollaries.

We introduce some additional notation to state our results.
Let $\rho_\ell,\ \ell \in [N]$ be the eigenvalues of $D_1^\top D_1$ where 
$D_1 = D_{N, 1}^{(k+1)},\ N = n^{1/d}$.
Abbreviate $D = D_{n,d}^{(k+1)}$ and let  ${\xi_{i}^2: i = (i_1, \dots, i_d) 
\in [N]^d}$ be the eigenvalues of
$D^\top D$. Due to the Kronecker-sum structure of $D^\top D$, we have
$
\xi_i^2 = \sum_{j=1}^d \rho_{i_j}.
$
Let $\kappa = (k+1)^d$ denote the nullity of $D^\top D$.
A nonzero vector $x \in \R^n$ is said to be {\em incoherent} with a constant 
$\mu\ge 1$
if $\|x\|_\infty/\|x\|_n \leq \mu.$ Note that, for arbitrary nonzero
$x\in \R^n,$ $\|x\|_\infty/\|x\|_n \in [1,\ \sqrt{n}].$ 
For $J \subset [N]^d$ containing $[k+1]^d$, define
\begin{equation}
	\label{eq:Lell}
	L_{J,p} = \left( \frac{\mu^2}{n} \sum_{i \in [N]^d \setminus J}^n 
	\frac{1}{\xi_{i}^p} \right)^{1/p}
\end{equation}
where $\mu$ is the constant with which the left singular vectors of \smash{$D$} 
are incoherent.
We can derive the following error bound on the excess risk of the estimator in 
\eqref{eq:estimator}.
\begin{theorem}
	\label{thm:ex_risk_bd}
	Let \smash{$y_i = \beta^*_i + \epsilon_i$} where \smash{$\epsilon_i$} is zero 
	mean
	sub-exponential with parameters \smash{$(\nu_i^2, b_i)$} for \smash{$i\in 
	[n]$}.  
	Let \smash{$L$} be as defined in \eqref{eq:Lell}.
	For \smash{$t\geq 1$}, abbreviate
	$A_n = 2t\mu \sqrt{\nicefrac{\kappa}{n}} \big( \| \nu \|_2 \vee 
	\|b\|_\infty\big),$
	$B_n = 2t\left( 
	\min\left\{ \|\nu\|_\infty  L_{\kappa,2},\ 
	\|\nu\|_2 L_{\kappa,1} \right\}
	\vee  \| b \|_\infty L_{\kappa,1}\right)
	$ 
	where \smash{$L_{\kappa, p} = L_{[k+1]^d, p}$} for $p\ge 1$.
	Let $\hat\theta$ be our estimate in \eqref{eq:estimator} with parameters 
	\smash{$\lambda_2 = 2A_n/n$} and 
	\smash{$\lambda_1 = 2B_n/n.$}
	Then, with probability at least \smash{$1-4nd e^{-t}$},
	\begin{align}
    \KLbar{\theta^*}{\hat\theta}
    &\leq \frac{3}{n} \big( A_n \|P_\cN  
		\theta^*\|_2 + B_n \|D\theta^*\|_1 \big), \quad \text{and}\\
		A_n \|P_\cN \hat\theta\|_2 +  B_n \|D\hat\theta\|_1 &\leq 3 \big( 
		A_n \|P_\cN\theta^*\|_2 + B_n \|D\theta^*\|_1 \big).
	\end{align}
\end{theorem}
\noindent
See the proof in \aref{sec:app_ex_risk_bd}.
For regular grids,
\autoref{lem:kron_tf_evals} (in \aref{sec:eigenvalue-bounds}) controls the magnitude of \smash{$L_{\kappa, 1}, 
L_{\kappa,2}$} and hence 
the bounds in
\autoref{thm:ex_risk_bd}. 
Applying the lemma to the expression for $B_n$ in \autoref{thm:ex_risk_bd}, we 
get the following
corollary 
for regular grids.
\begin{corollary}
	\label{cor:MLE_Grids}
	Assume canonical scaling $\|D\theta^*\|_1 \lesssim n^{1-\alpha},$ and $\|P_\cN 
	\theta^*\|_n \lesssim 1$. Then for $t\ge 1$,
	\begin{equation}
    \overline{\mathrm{KL}}\big( \theta^*  \Vert \hat\theta \big) = O_\P(r_n 
    \log n), 
    \;\; 
    \text{with}\;\;
    r_n = 
		\begin{cases}
			(\| \nu \|_\infty + \| b \|_\infty ) n^{-\alpha}
			& \alpha < \half\\
			\| b \|_\infty  n^{-\alpha} \gamma_1 + \min\{ \| \nu \|_\infty n^{-\frac 
			12} 
			\gamma_2,\ \| \nu \|_2 n^{-\alpha} \gamma_1\}
			& \alpha \in [\half, 1]\\
			\frac{1}{n} \big( \| \nu\|_2 + \|b\|_\infty\big)
			& \alpha > 1
		\end{cases}
	\end{equation}
	and $\gamma_{p} = (\log n)^{1/p} \mathbf{1}(p\alpha=1)$ for $p\ge 1$.
\end{corollary}
For Gaussian errors with \smash{$\nu = \sigma \one,\ b = 0$}, we 
recover  
optimal rates  in the case 
$\alpha \leq \nicefrac 12$ up to logarithmic factors (see for example 
\citealt{sadhanala2021multivariate}). 
However, we get suboptimal rates when 
\smash{$\alpha > \nicefrac 12$}.

\subsection{Penalized MLE in special cases}

We now illustrate \autoref{cor:MLE_Grids} in a few special cases to provide intuition.
As above, we focus on grid graphs with Poisson and Exponential distributions,
and we assume that these are all of width $N$ and dimension 
$d$, so that $n = N^d$.
Recall that for natural parameter $\theta^*$, the Poisson distribution has mean
$\beta^* =  \exp(\theta^*)$, while the Exponential distribution has mean $\beta^* = 
-1/\theta^*$.
For the Poisson distribution, an additive change in $\theta^*$ results in a 
multiplicative change in the mean, and $\nu^2 = 2 \beta^*$, which can easily 
result in strong heteroskedasticity.
Only in special cases does a constraint on $\| D \theta^* \|_1$ result in a 
bound on $\nu^2$, and generally, $\| \nu \|_\infty$ will depend on the signal 
in question.

The first result is an example of weak heteroskedasticity, where the natural
parameter is uniformly bounded.

\begin{corollary}
\label{cor:poisson_weak_hetero}

Consider the Poisson distribution where the natural parameter vector 
$\theta^*$ satisfies $\| \theta^*\|_\infty = O(1)$.
Let $k=1$ and assume that $\theta^*$ satisfies the canonical scaling, $\| D \theta^*
\|_1 = O(n^{1-\nicefrac{2}{d}})$.
Then, we have the following rate bound for penalized MLE trend filtering.

\begin{equation}
  \KLbar{\theta^*}{\hat\theta}
  = O_\P(r_n \log n), \text{ \quad \quad where\quad\quad } r_n =
	\begin{cases}
	  n^{-\half}, & d=1\\
    n^{-\half} \log n, & d = 2\\
    n^{-\half}, & d = 3\\
    n^{-\half} \log^{\half} n, & d = 4\\
		n^{-\nicefrac{2}{d}}, & d> 4.
	\end{cases}
\end{equation}
\end{corollary}

A simple example of such a signal is $\theta^*_i = \frac{2}{N}\sum_{j=1}^d |i_j
- N/2|$, where $i = (i_1,\ldots,i_d)\in [N]^d$. For a proof,
see \aref{sec:proof_mle_bound_coro}.

The next example demonstrates \autoref{cor:MLE_Grids} under strong
heteroskedasticity. 

\begin{corollary}
\label{cor:MLE_strong_hetero}
Consider any exponential family on a $d$-dimensional grid ($d > 1$) with a
natural parameter that satisfies $\| \nu \|_\infty$, $\| b \|_\infty = O(n^c)$ and
$\| \nu \|_2 = O(n^c)$ for some $c > 0$, and the canonical scaling for $k=0$. 
Then 
\begin{equation}
  \KLbar{\theta^*}{\hat\theta}
  = O_\P(r_n \log n), \text{ \quad \quad where\quad\quad } r_n =
  \begin{cases}
    n^{c-\half}, & d = 1\\
    n^{c-\half} \log^{\half} n, & d = 2\\
    n^{c-\nicefrac{1}{d}}, & d> 2.
  \end{cases}
\end{equation}
\end{corollary}
An example of a signal satisfying these conditions is the Exponential
distribution with $\theta^*_i = -n^{-c} \one\{i=0\} - n^{1-\nicefrac{1}{d}} \one\{i \neq
0\}$. The proof is in \aref{sec:proof_mle_bound_coro}.

In this case, $\| \nu \|_\infty$ is 
diverging, and so we have strong heteroskedasticity.
The level of heteroskedasticity, parameterized by $c$, determines the rate of 
convergence and for $c > \nicefrac{1}{d}$ we cannot guarantee convergence.




\section{Error bounds for the Mean Trend Filter}
\label{sec:theory_mean_tf}

When \smash{$k=0$}, remarkably, it turns out that the penalized MLE in
\eqref{eq:mle1} is  
equivalent to the mean trend filtering estimator \eqref{eq:lsq}.
 In fact, this equivalence between the two 
estimators holds over arbitrary graphs, not just
grids.
\begin{theorem}
	\label{thm:mle_ls_k0}
	Suppose $k=0$ and let $D$ be the graph incidence matrix. Then, the penalized
  MLE $\hat\theta$ in \eqref{eq:mle1} and the least squares  
	estimator $\hat\beta$ in \eqref{eq:lsq} satisfy $\hat \beta = \psi' (\hat\theta)$.
\end{theorem}

The proof is in \autoref{sec:mle_ls_k0_pf}.
Therefore, in the case \smash{$k=0$},
the penalized MLE can be solved quickly by solving the equivalent mean trend 
filter problem.

For \smash{$k\geq 1$}, equivalence between the two estimators need not hold in 
general, with the exception of the mean parameterized Gaussian family, where it holds
trivially. The remainder of this section will focus on the general case.
For the estimator in \eqref{eq:lsq}, we derive the following error bound.

\begin{theorem}
\label{thm:lsq_rates}
Let $y_i = \beta^*_i + \epsilon_i$ where $\epsilon_i$ is zero mean
sub-exponential with parameters $(\nu_i^2,\ b_i)$ for $i\in [n]$.  
Let $J \subset [N]^d$ and $L$ be as defined in \eqref{eq:Lell}.
For $t\geq 1$, abbreviate
\smash{$A_n = 2t\mu \sqrt{\nicefrac{|J|}{n}} \big( \| \nu \|_2 \vee 
\|b\|_\infty\big),$}
\smash{
$B_n = 2t\left( 
	\min\left\{ \|\nu\|_\infty  L_{J,2},\ 
					  \|\nu\|_2 L_{J,1} \right\}
	\vee  \| b \|_\infty L_{J,1}\right).
 $}
For any $J \subset [N^d]$ containing $[k+1]^d$, the estimator \eqref{eq:lsq} 
with
$\lambda =  B_n / n$, satisfies 
\begin{equation}
\frac{1}{n}\snorm{ \hat\beta - \beta^* }_2^2 
\leq \frac{4A_n^2}{n} +
\frac{8B_n}{n} \norm{ D \beta^*}_1
\end{equation}
with probability at least $1-4nd e^{-t}$ for $t\geq 1$. 
\end{theorem}
The set of indices $J$ can be chosen to minimize the bound.
The proof is in \aref{sec:lsq_rates_proof} and follows an approach similar to
that in 
\cite{WangSharpnack2016}.
Tail bounds on sums of sub-Gaussian variables in their results are replaced 
with those on sums 
of sub-exponential variables. This results in additional $\log n$ factors
 in the error bound compared to the sub-Gaussian setting.

The proof technique for  \autoref{thm:lsq_rates} relies on the properties of
$D$. A potential
alternative route to get error rates is via bounding the empirical process $
\frac{1}{n}\epsilon^\top (\hat\theta - \theta^*)$ with the Dudley entropy integral.
However, the empirical process in our case is not sub-Gaussian and we could only
derive a trivial upper bound in this way.  
This should not be entirely surprising however, because the entropy method was
also used in \cite{WangSharpnack2016} in the sub-Gaussian 
noise setting, and it also failed to give a tight characterization in that
context.


\subsection{Error bounds with canonical scaling}

We simplify this bound in some special cases.
Assuming that $\nu,\ b$ are uniformly bounded, 
we get the following result for regular grids.
Denote \smash{$\gamma_p = \log^{\nicefrac{1}{p}} (n)$} if \smash{$p\alpha=1$} and 
\smash{$1$} otherwise.
\begin{corollary}
\label{cor:lsq_homosked}
Assume $\|\nu\|_\infty \leq \omega,\ \| b \|_\infty \leq \omega$. 
Let $\alpha = (k+1)/d$. 
For $d$-dimensional grids, assume that $\|D\beta^* \|_1 \asymp n^{1-\alpha}$ 
and let $m = 
d(n-n^{1-\nicefrac{1}{d}})$ denote the number of rows in $D.$
Then there is a choice of $\lambda$ such that for $\alpha \le \half$,
\begin{equation}
\frac{1}{n}\| \hat \beta - \beta^* \|_2^2 = 
O_\P \left( \frac{\omega^2 \log^2 n}{n} + \frac{\omega\gamma_2 \log n}{n^\alpha} 
\right)
\end{equation}
and for $\alpha > \half$ and $ n^{-\alpha}  \leq \omega \log n   \lesssim \sqrt 
n$
\begin{equation}
\frac{1}{n}\| \hat \beta - \beta^* \|_2^2 = 
O_\P   \left( \Big(\frac{\omega^2  \log^2 n}{n} \Big) 
^{\frac{2\alpha}{2\alpha +1}} + \frac{\omega\gamma_1 \log n}{n^\alpha} \right).
\end{equation}
\end{corollary}

The proof is in \aref{sec:lsq_homosked_proof}. This corollary does not
discuss the case where \smash{$\alpha > \half$} and \smash{$\omega \log n$} is  
outside of \smash{$[n^{-\alpha}, \sqrt{n}]$}. In that case, when the noise is 
high 
(\smash{$\omega \log n \gtrsim \sqrt{n}$}),
the polynomial projection estimator  
\smash{$\hat\beta = P_\cN y$}
gives the tightest bound, and, when the noise is  low (
\smash{$\omega \log n < n^{-\alpha}$}),
the identity estimator gives the tightest bound.

The following corollary examines this result for some special cases.


\begin{corollary}
\label{cor:mean_tf_weak_hetero}
Consider the Poisson and Exponential families on a $d$-dimensional grid ($d >
1$) where the mean parameter is constrained. 
Specifically, suppose that $\| \beta^* \|_\infty = O(1)$ such that the canonical
scaling holds with $k=1$. 
Then mean trend filter satisfies
\[
  \frac 1n \|\hat \beta - \beta^*\|_2^2 = O_\P(r_n)\;\; \text{where}\;\; r_n =
  \begin{cases}
      (n/\log^2 n)^{-4/(4+d)},& d=1,\ 2,\ 3\\
      n^{-\half} \log^{\nicefrac{3}{2}} n,  & d = 4\\
      n^{-\nicefrac{2}{d}} \log n, & d > 4.\\
  \end{cases}
\]
\end{corollary}

This result matches with rates in the 
homoskedastic Gaussian case up to logarithmic factors, shown for example in 
\cite{sadhanala2021multivariate}. An example of a signal satisfying the
conditions is a grid graph
with width $N$ and dimension $d$, so that $n = N^d$ and  
$\beta_i^* = \frac dN + \frac 2N \sum_{j=1}^d |i_j - \frac{N}{2}|,$
where $i = (i_1, \dots, i_d) \in [N]^d$.
The proof is given in \aref{sec:lsq_homosked_proof}.


While the previous result treated the (effectively) homoskedastic case by controlling the
largest components of $\nu$, $b$, the following corollary specializes 
\autoref{thm:lsq_rates} to canonical scaling under strongly heteroskedastic 
noise.

\begin{corollary}
	\label{cor:mtf_grids_hetero2}
	Let \smash{$\sigma =  (\| \nu \|_2\vee \| b \|_\infty) /\sqrt{n}$},
	\smash{$\sigma_\infty = \| \nu \|_\infty\vee \| b \|_\infty$}.
	Suppose \smash{$\|D \beta^* \|_1 \lesssim n^{1-\alpha}$}, and 
  assume $\sigma^2\lesssim n / \log^2 n$ and $\sigma_\infty \lesssim n^\alpha /
  (\gamma_1\gamma_2\log n)$.
  Then, the estimator $\hat\beta$ in \autoref{thm:lsq_rates}
  satisfies 
  \begin{equation}
    \frac{1}{n} \snorm{\hat\beta - \beta^*}_2^2 =
    \begin{cases}
      O_\P\left(\frac{\sigma^2\log^2 n}{n} + 
      \frac{\sigma_\infty\gamma_2\log n}{n^\alpha}
      \right), & \alpha\leq \half\\
      O_\P\left( \frac{\sigma^2\log^2 n}{n} +
      \frac{\min\{\sigma_\infty,\ \sigma\gamma_1n^{1-\alpha}\} \log n}{n^{\half}}\right),
               & 1/2 < \alpha \leq 1, \\
      O_\P\left(\left(\frac{\sigma^2\log^2 
      n}{n}\right)^{\frac{2\alpha-1}{2\alpha}} +
      \frac{\sigma_\infty\log n}{n^\alpha}
      \right), & \alpha > 1.
    \end{cases}
  \end{equation}
\end{corollary}

This result is most useful under strong heteroskedasticity where $\sigma_\infty
/ \sigma \propto \sqrt{n}$, and slightly stronger rates with weaker
heteroskedasticy can be obtained in the $1/2<\alpha\leq 1$ case
(see \autoref{cor:mtf_grids_hetero} in \aref{sec:lsq_homosked_proof}).
Suppose $\epsilon_i$ in \autoref{thm:lsq_rates} is mean-zero Laplace noise
with standard deviation parameter $\tau_i$ and that $\norm{D\beta^*}_1$ 
satisfies
canonical scaling. For this case, $\nu_i = b_i = c \tau_i$ for a constant $c$
independent of $\beta^*_i$, while $\sigma = c \norm{\tau}_n$ and
$\sigma_\infty = c \norm{\tau}_\infty$ with the natural constraint that
$\sigma_\infty/\sqrt{n} \leq \sigma \leq \sigma_\infty$. For $\alpha < 1$,
the scaling requirement on $\sigma_\infty$ is stronger, meaning that the estimator
can only tolerate heteroskedasticity on the order of $\sigma_\infty / \sigma
\propto n^{\alpha}<\sqrt{n}$. On the other hand, for $\alpha > \half$, the constraint 
on $\sigma$ is stronger, meaning that we can tolerate $\sigma_\infty / \sigma \propto
\sqrt{n}$. The associated rates of convergence will necessarily be much slower than in
the homoskedastic sub-Gaussian case.

Importantly, \autoref{cor:mtf_grids_hetero2} illustrates that without control of
the amount of heteroskedasticity, we cannot guarantee convergence of the
estimator. In other words, while the estimator can tolerate strong
heteroskedasticity as we have defined it here, it cannot tolerate arbitrary
heteroskedasticity. Simply controlling $\norm{D\beta^*}_1$ is not generally
enough to guarantee estimation consistency. In the next section, we make this
precise, illustrating that in certain settings,  there is no estimator that can
achieve consistency 
without additional constraints.

\subsection{Lower bounds for mean trend filtering}
\label{sec:lower-bounds-mean}

We now show that the upper bound in \autoref{cor:lsq_homosked} is minimax 
optimal up to logarithmic factors.
Consider the observation model
\begin{equation}
\label{eq:data_model}
	y_i = \beta_i + \epsilon_i,\ \ i\in [n]
\end{equation}
where $\beta \in \R^n$ is the true signal and $\epsilon_i,\ i\in[n]$ are mean-zero 
noise terms. 
For a set $S \subset \R^n$ denote its minimax risk
\begin{equation}
  \textrm{MSE}(S)
  = \inf_{\hat\beta} \sup_{\beta \in S}  \E\left[ \| \hat \beta - \beta 
\|_n^2 \right] 
\end{equation}
where $\hat\beta$ is measurable in the observations $y \in \R^n.$
Consider the Kronecker total variation (KTV) set 
\begin{equation}
	T_{n,d}^k(C_n) = \left\{ \beta : \|D_{n,d}^{(k+1)} \beta \|_1 \leq C_n \right\},
\end{equation}
for integers \smash{$k\ge 0,\ d\ge 1,\ n\ge (k+1)^d$} and 
\smash{$C_n \ge 0$}.
Let $\lap(\mu, \sigma)$ denote the Laplace distribution centered at $\mu 
\in \R$ and with scale parameter $\sigma > 0$ with
density 
\smash{$p(x) = \frac{1}{2\sigma} e^{-|x - \mu|/\sigma} $} over $\R$.
\begin{proposition}
\label{prop:lowerbd_homosked}
Consider the observation model in \eqref{eq:data_model}
where $\epsilon_i,\ i \in [n]$ are \iid $\lap(0, \sigma)$
for a parameter $\sigma > 0$. 
Then,
\begin{equation}
  \mathrm{MSE}
  \left( \ktvset_{n,d}^k(C_n)\right) 
 =
\Omega \left(
\frac{ \sigma^2}{n}
	+ \frac{\sigma C_n}{n}  \log \left(\frac{\sigma n}{C_n}\right) 
	+ \left( \frac{C_n}{n} \right)^{\frac{2}{2\alpha+1}} 
\left(\sigma^{\frac{4\alpha}{2\alpha+1}} \wedge \sigma^2\right)
\right)
\end{equation}
where the $\Omega$ notation absorbs constants depending only on $k,\ d.$
\end{proposition}

The first term in the bound is due to the null space of $D.$ 
To derive the second term, 
we embed an \smash{$\ell_1$} ball in \smash{$\ktvset_{n,d}^k (C_n)$} and adapt 
arguments from 
\cite{birge2001gaussian}.
The final term is obtained similarly to \citet[Theorem 4]{SadhanalaWang2017},  
by
embedding a H\"older ball of appropriate size in
\smash{$ \ktvset_{n,d}^k (C_n)$}.
The proof is in \aref{sec:lowerbd_homosked_proof}.
\noindent

Let us compare the lower bound in  \autoref{prop:lowerbd_homosked} with
the upper 
bound in 
\autoref{cor:lsq_homosked}.
The Laplace distribution with scale parameter $\sigma$ is sub-exponential with
parameters  
$\nu = c\sigma,$ $b = c\sigma$ for some constant $c>0$.
Plugging in $C_n = n^{1-\alpha}$ in the lower bound, and $\omega=c\sigma$ in the
upper bound  
stated in \autoref{cor:lsq_homosked},  we can verify that the bounds match 
up to logarithmic factors.

The lower bound in  \autoref{prop:lowerbd_homosked} is for homoskedastic noise. 
When the noise is heteroskedastic, the estimation can be harder, in the sense 
that the minimax risk can be larger.
Specifically, we can show the following lower bound on a TV class of signals 
for the Exponential family.
\begin{proposition}
\label{prop:lower_bd_2dgrids}
Assume $C_n > 1$. Consider the class of signals over a 2d grid
\begin{equation}
\Theta(C_n) 
	= \left\{ \beta \in 
		\R^n : 
	\| \ktfmat_{n,2}^{(1)} \beta \|_1 \leq C_n, \ 
	\|\beta\|_\infty \leq 2C_n
	\right\}
\end{equation}
and the observation model $y_i \sim \mathrm{Exp}(\mathrm{mean} = 
\beta_i)$ for $i\in[n]$. Then
\begin{equation}
  \mathrm{MSE}\big( \Theta(C_n) \big)
  \;\geq \;
  \frac{3}{256} \frac{C_n^2}{n}.
\end{equation}
\end{proposition}
\noindent
The proof is in \aref{sec:lower_bd_2dgrids_proof}.
With canonical scaling $C_n \asymp n^{1-\alpha} = \sqrt{n}$, this means a lower
bound of $\Omega(1)$. 
In other words, there is no consistent estimator for the class of signals
$\Theta(\sqrt{n})$.  
This result also hints at the difficulty of handling various regimes of noise
parameters $\nu,\ b.$ 

\section{Algorithmic implementation}
\label{sec:algor-impl}

In this section, we discuss our algorithmic implementation, focusing on the
multivariate setting for the MLE trend filter for which there are not currently
generic procedures. 
For the Mean Trend Filter, there are many standard approaches that can apply
immediately since this is a quadratic program. In the one dimensional case with
$k=0$, \citet{KimKoh2009} use 
a Primal-Dual Interior-Point method. \citet{RamdasTibshirani2016} examine a fast
ADMM algorithm for $k>0$. \citet{WangSharpnack2016} develop ADMM and Newton
methods for general graphs and arbitrary $k$. We follow the approach of
\citet{KhodadadiMcDonald2018} for the MLE trend filter 
\eqref{eq:mle1} and use an
algorithm called linearized ADMM. A more complete description is given in
\autoref{sec:app-algor-deets}. First, rewrite Equation \eqref{eq:mle1}
(substituting $x$ for $\theta$) as
\begin{equation}
	\min_{Dx=z} \frac{1}{n}\sum \psi(x_i) - y_i x_i + \lambda \left\lVert z 
	\right\rVert_1.
\end{equation}
This is equivalent to \eqref{eq:mle1} but with additional variables.
The scaled form of the augmented Lagrangian for this problem is
\begin{equation}
	L_{\rho}(x,z,u) = \frac{1}{n}\sum \psi(x_i) - y_i x_i +
	\lambda \left\lVert z \right\rVert_1 + \frac{\rho}{2}\left\lVert Dx - z + 
	u\right\rVert_2^2
	- \frac{\rho}{2}\left\lVert u \right\rVert_2^2.
\end{equation}
\noindent
The scaled ADMM algorithm iteratively solves this problem by minimizing over 
$x$,
then $z$ and then updating $u$ with gradient ascent. However
the $x$ solution involves a matrix inversion due to the quadratic in $Dx$ which
is best avoided when $n$ is large. So we linearize $L_\rho(x,z,u)$ around the
current value $x^o$   resulting in the following update for $x$
\begin{equation}
	\label{eq:linearized-x-update}
	x \leftarrow \argmin_x \frac{1}{n}\sum \psi(x_i) - y_i x_i +
	\rho \left(D^\top D x^o - D^\top z + D^\top u\right)^\top x +
	\frac{\mu}{2}\left\lVert x-x^o \right\rVert_2^2, 
\end{equation}
where $\mu$ is chosen as the largest eigenvalue of $D^\top D$.
To include the null space penalty,
the changes only impact the $x$ update,
and \eqref{eq:linearized-x-update} is adjusted accordingly with a subgradient of
the penalty at $x^o$ (when $P_\cN x^o = 0$, choose the subgradient to be $0$).


\begin{algorithm}[tb!]
  \caption{Linearized ADMM for the MLE trend filter}
  \label{alg:mle-tf}
  \begin{algorithmic}[1]
    \STATE {\bfseries Input:} $y,\ \phi, D,\ \lambda_1>0,\ \lambda_2\ge 0$
    \STATE {\bfseries Set:} $x^o = \phi^{'-1}(y),\ \rho = \lambda_1,\ z = u =
    0,\ \mu = \lambda_{\max}(D^\top D)$
    \WHILE{Not converged}
    \STATE Set $b = y - \rho D^\top (Dx^o - z + u) + \mu x^o + \lambda_2 
    P_{\cN} x^o / \|  P_{\cN} x^o \|_2$
    \STATE Update $x$ by solving $\psi'(x_i) + \mu x_i = b_i$ for $i\in [n]$.
    \STATE Update $z \leftarrow \mathrm{Soft}_{\lambda/\rho}(Dx + u)$ with
    $\mathrm{Soft}_a(v) = \mathrm{sign}(v)(|v| - a)_+$.
    \STATE Update $u\leftarrow u + Dx - z$
    \ENDWHILE
    \RETURN{$z$}
  \end{algorithmic}
\end{algorithm}

The solution for the $z$-update is easily shown to be given by elementwise soft-thresholding,
and the $u$-update is 
simply vector addition. Solving the $x$-update is potentially more challenging.
Note that the form of \eqref{eq:linearized-x-update} is the same for each $i$,
so we can solve $n$ one-dimensional problems.  
The KKT stationarity condition requires
$$
	0 =\left(\psi'(x_i) - y_i\right)  + \rho \left( D^\top \left(D x^o -
	z+u\right)\right)_i + \mu( x_i-x_i^o).
$$
Therefore, for any negative loglikelihood as given by $\psi$, we want to solve
$
\psi'(x_i) + \mu x_i = b_i,
$
for each $i\in[n]$. For many functions $\psi$, the solution has a closed form. The
binomial distribution with $\psi(x) = \log(1+e^x)$ is an exception, though
standard root finding methods have no difficulties.
To include the nullspace penalty, the $x$ update changes slightly, but the
logic is the same. This procedure is shown in \autoref{alg:mle-tf}. In practice,
we have found the algorithm to converge quickly when initialized from a small value
of $\lambda_1$ (because the solution will be close to the MLE) and then
calculated for an increasing sequence with the solution at smaller
$\lambda_1$ used as a warm start. This is the opposite of most pathwise 
procedures which use a decreasing sequence of $\lambda_1$.

\section{Degrees of freedom and tuning parameter selection}
\label{sec:tuning-param-select}

We describe an unbiased estimator for the KL divergence
between the estimate and the truth for the purposes of tuning parameter
selection. Additional justification and description of its derivation are given
in \autoref{sec:app-comp-deets}.
If \smash{$Y \sim \mbox{N}(\theta^*, \sigma^2)$}, a now common method of risk
estimation makes use of Stein's Lemma.
The utility of this result comes from examining the decomposition of
the mean squared error of \smash{$\hat\theta(Y)$} as an estimator of 
\smash{$\theta^*$}.
\begin{align}
	\label{eq:9}
	\Expect{\snorm{\theta^*-\hat\theta(Y)}_2^2}
	&= \Expect{\snorm{Y-\hat\theta(Y)}_2^2} -n\sigma^2 + 2
	\trace\Cov(Y,\hat\theta(Y))\\
	&= \Expect{\snorm{Y-\hat\theta(Y)}_2^2} -n\sigma^2 + 2\sigma^2
	\Expect{\trace J\hat\theta(z) \big\vert_Y},
\end{align}
where $J$ denotes the Jacobian.
This characterization motivates the definition of degrees-of-freedom
for linear predictors: $\textrm{df} :=\frac{1}{\sigma^2} \trace
J\hat\theta(z)\big\vert_y$ \citep{Efron1986}, where 
\smash{$\hat\theta(y)=Hy$}. Using
Stein's Lemma, assuming \smash{$\sigma^2$} is known, we have Stein's Unbiased
Risk Estimator
\begin{equation}
	\label{eq:10}
	\mathrm{SURE}(\hat\theta) = \snorm{y-\hat\theta}_2^2 -n\sigma^2 + 
	2\sigma^2\trace\left( J\hat\theta(z) \big\vert_y\right), 
\end{equation}
which satisfies \smash{$\E[\textrm{SURE}(\hat\theta)] =
\E\snorm{\theta^*-\hat\theta(Y)}_2^2$}. 
Note that this is the risk for estimating the \smash{$n$}-dimensional parameter
\smash{$\theta^*$}. This estimator is appropriate for the mean trend filter,
but, for the MLE trend filter, we prefer ``Stein's Unbiased KL'' estimator due to 
\citet{Deledalle2017} that applies to continuous exponential families.
\begin{lemma}[Theorem 4.1 in \citealt{Deledalle2017}]
	\label{lem:sukls}
	Assume $h$ is weakly differentiable and that
	$\hat\theta(Y)$ is weakly differentiable with essentially bounded partial
	derivatives. Then
	\begin{equation}
		\mathrm{SUKL}(\hat\theta) = \Big\langle \hat\theta + \frac{\nabla 
			h(y)}{h(y)},\
		\hat\beta\Big\rangle + \trace\Big( J\hat\beta(z)
		\big\vert_y\Big)
		- \psi(\hat\theta)
	\end{equation}
	is unbiased for \smash{$\E [\mathrm{KL}(\hat\theta(Y)\ \Vert\ \theta^*)] - 
	\psi(\theta^*)$}.
\end{lemma}
Because \smash{$\psi(\theta^*)$} does not depend on \smash{$\hat\theta$}, we 
can ignore 
it for 
the
purposes of choosing \smash{$\lambda_1,\ \lambda_2$} in the MLE trend filter.
To evaluate \smash{$\mathrm{SUKL}(\hat\theta)$} we need an expression for 
\smash{$J\hat\beta(y)$}.
This is given in the following result (the proof is deferred to \autoref{sec:app-comp-deets}).
\begin{theorem}
	\label{thm:simple-divergence}
	For the MLE trend filter, the
	divergence of $\hat\beta(y)$, defined to be the trace of the Jacobian of
	$y\mapsto \hat\beta(y)$, written as $\trace\left( J \hat\beta(y) \right)$,
	is given by 
	\begin{equation}
		\label{eq:2}
		\trace\left(J \hat\beta(y)\right) = \trace\left(
		\diag\left( \psi''(\hat\theta) \right)\Pnd
		\left(\Pnd\diag\left(\psi''(\hat\theta)\right) \Pnd
		+ \lambda_2 P_\cN \right)^\dagger \Pnd\right),
	\end{equation}
	where $\Pnd$
	is the projection onto the null-space of \smash{$\breve{D}$}, and
  $\breve{D}$ contains the rows of $D$ such that  $D\hat\theta = 0$.
\end{theorem}

Unfortunately, estimating the risk in this manner is not known to be possible
for general discrete exponential 
families, though a few specific cases are possible. One such is the Poisson distribution. The
following result more closely resembles an empirical derivative of 
\smash{$\hat\beta$}
rather than the theoretical expression for \smash{$J\hat\beta(y)$} used in the 
previous
results. 

\begin{lemma}[Theorem 4.2 in \citealt{Deledalle2017}]
	Assume \smash{$Y$} is Poisson and that
	\smash{$\hat\theta(y)$} is weakly differentiable with essentially bounded 
	partial
	derivatives. Then
	\begin{equation}
		\mathrm{PUKL}(\hat\theta) = \snorm{\hat\beta}_1 - \langle  y,\
		\log \hat\beta_{\downarrow}(y) \rangle, 
	\end{equation}
	is unbiased for \smash{$\E[\mathrm{KL}(\theta^*\ \Vert\ \hat\theta(Y))] -
    z(\theta^*)$}  
	where
	\smash{$[\hat\beta_{\downarrow}(y)]_i = [\hat\beta(y-e_i)]_i$}, where 
	\smash{$e_i$} 
	is the
	\smash{$i^{th}$} standard basis vector, and \smash{$z$} is a known
	function of the true parameter.
\end{lemma}

With these expressions in hand, we can select the tuning parameters 
\smash{$\lambda_1,\
\lambda_2$} with minimal additional computations by minimizing
\smash{$\mathrm{SUKL}(\hat\theta)$} or 
\smash{$\mathrm{PUKL}(\hat\theta)$} as appropriate.

\section{Empirical results}
\label{sec:experiments}

 We demonstrate the performance of both the MLE and the Mean trend
filter estimators in a small 
scale simulation designed to compare the two in challenging settings. We also
examine two applications: modeling hospital admissions by age due to COVID-19
in Davis, California; and describing changes in temperature measurements for the
Northern hemisphere.

\subsection{Simulation study}

We briefly investigate the relative performance of the Mean Trend Filter and the
MLE Trend Filter on a few synthetic examples. Our intention is to push the
limits of both, thereby illustrating that the user should choose between the two
based on whether smoothness is desired in the mean or in the natural parameter.
We focus on one dimension for ease of visualization and $k = 1$. We examine both
the exponential distribution and the Poisson distribution.

To create the true signal, we begin with a v-shaped function on the unit
interval:
\[
f_n(x) = \frac 1n + \left( 1 - \frac 2n \right) \left| x - \frac 12\right|
\]
Evaluating this at $n$ equally-spaced points for any \smash{$n$} gives a signal 
with
\smash{$\snorm{Df_n(x)}_1$} having the canonical scaling of \smash{$\nicefrac{1}{n}$}.

For the exponential distribution, we set either \smash{$\theta^*$} or 
\smash{$\beta^*$} equal to
\smash{$f_n(x)$} and evaluate both the Mean Trend Filter and the MLE Trend 
Filter on
sample data. When \smash{$\theta^*$} is controlled, the mean at \smash{$x = 0.5$} 
approaches
infinity as \smash{$n$} grows, making estimation very challenging. The reverse 
occurs
if \smash{$\beta^*$} is controlled. For the Poisson, because the mapping 
from natural parameter to mean is exponential, controlling one does not 
particularly
challenge the opposite procedure with the above $f_n$. To increase the 
discrepancy, we 
use 
\smash{$g_n(x) = 0.5 - f_n(x) + \log(n)$}. The signal should create more 
discrepancy between the
estimators as \smash{$n$} grows, but results are less dramatic than those in 
the 
exponential case.

\begin{figure}[t]
  \centering
  \includegraphics[width=.9\textwidth]{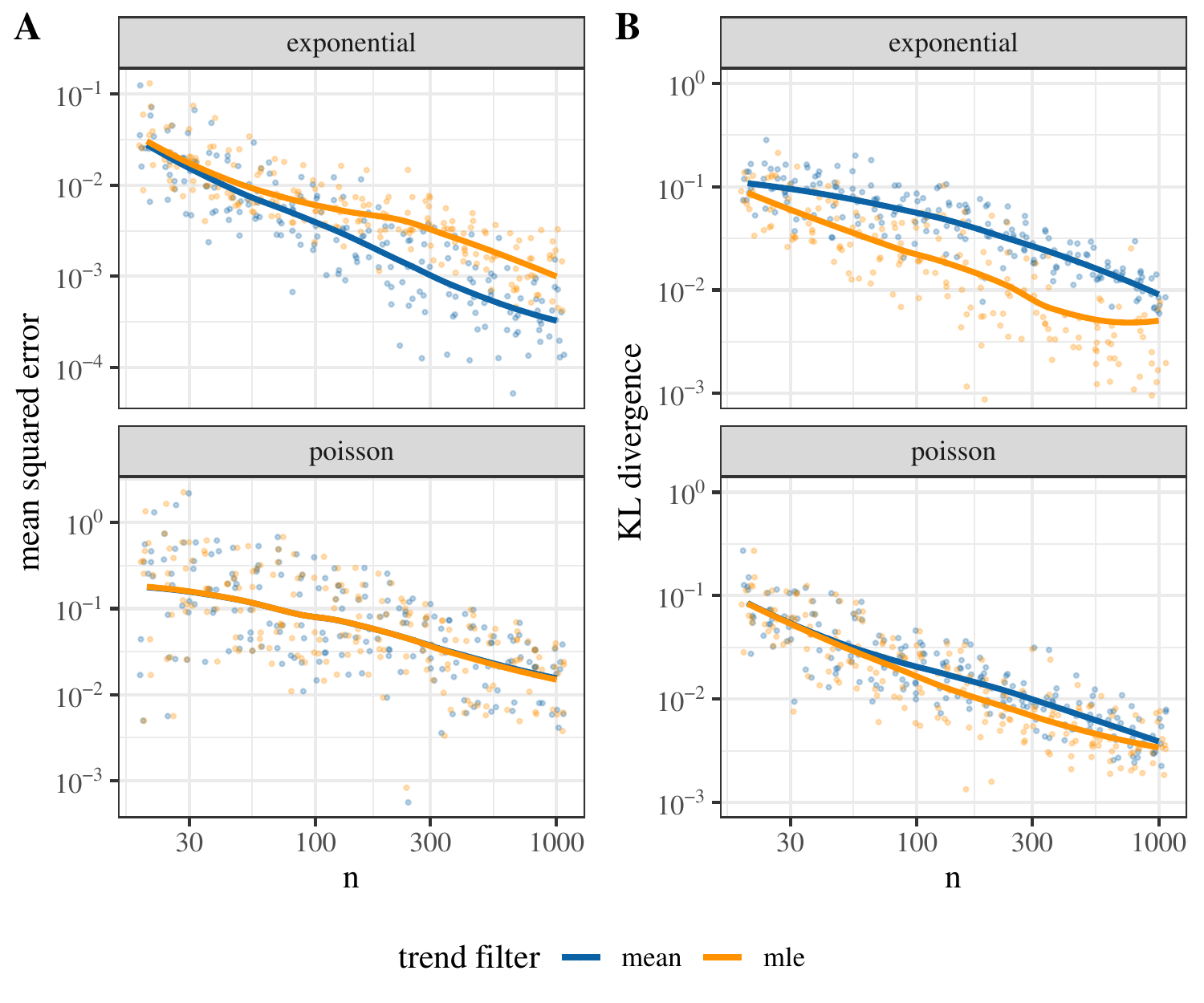}
  \caption{Estimation accuracy for both types of trend filters. The left column
    (panel A) compares the estimators when the mean is smooth. The right column
    (panel B) compares the estimators when the natural parameter is smooth.
    Solid lines show the average error across replications while the points show
    the error for each replication.}
  \label{fig:sim-errors}
\end{figure}

\begin{figure}[t]
  \centering
  \includegraphics[width=.9\textwidth]{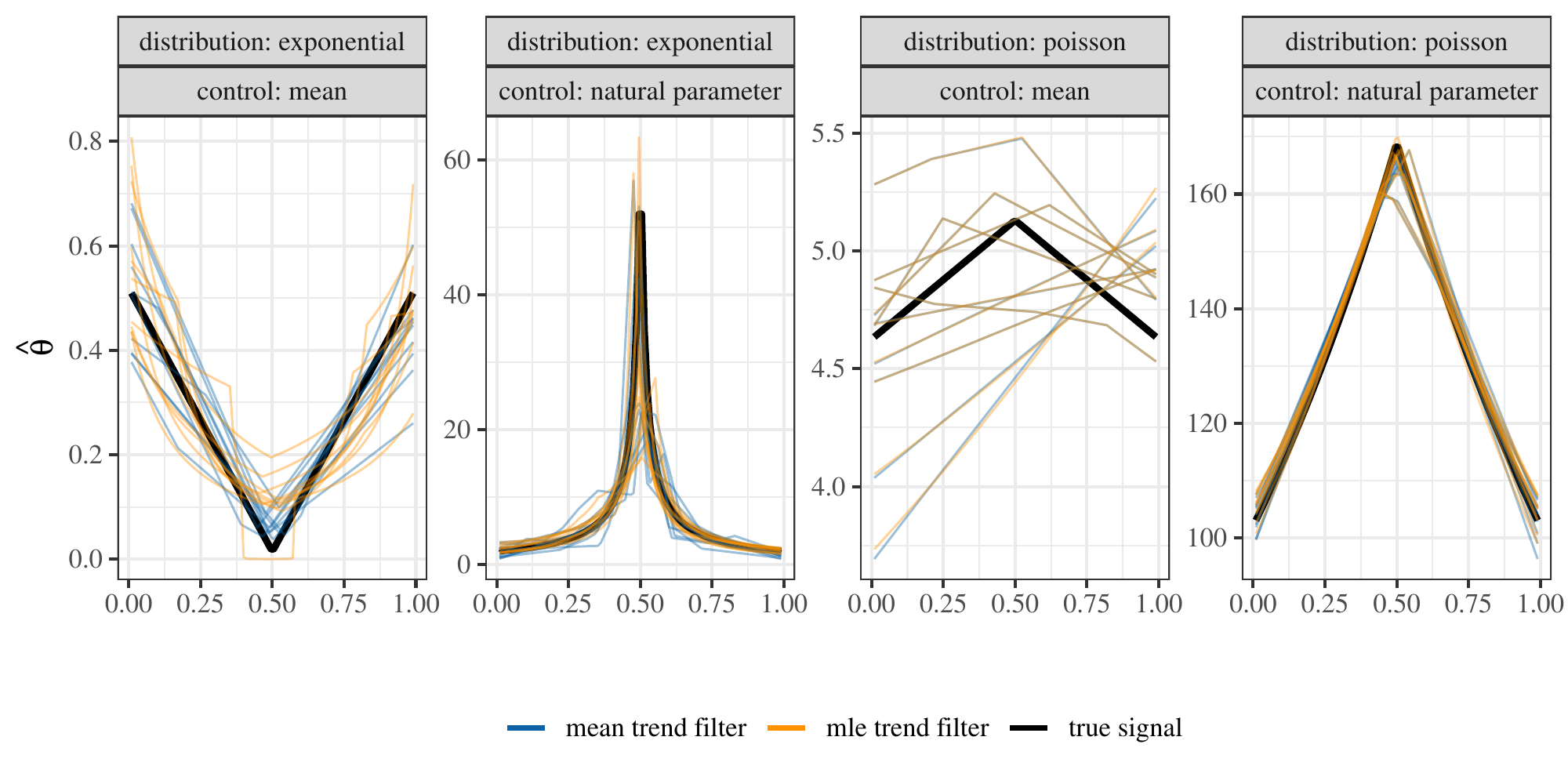}
  \caption{Estimates from both trend filters for the 4 scenarios when 
  \smash{$n=104$}.}
  \label{fig:sim-estimates}
\end{figure}

\autoref{fig:sim-errors} shows estimation accuracy for both trend filters across
four different scenarios. In all cases, we generated data using the signals
described above for 20 values of $n$ ranging from 20 to 1000. The values are
evenly spaced on the logarithmic scale. For each $n$, we repeated the experiment
10 times. The left column (panel A) shows results for both distributions when
the mean is smooth (mean is given by the smooth functions above) and error is 
measured using the mean-squared error
between the estimate and the truth. In the exponential case, the mean trend
filter is slightly more accurate for larger $n$, but the overall error also
decreases with $n$ since the problem is becoming easier. In the Poisson case,
the estimates (and therefore their errors) are nearly the same. The right column
(panel B) shows results when 
the natural parameter is smooth. Here, for both distributions, the MLE trend
filter performs better (as measured by KL divergence), but the difference is
again more pronounced for the exponential distribution.
\autoref{fig:sim-estimates} shows all the estimates for all four scenarios when
$n=104$. In the left two panels, for the exponential distribution, it is clear
that whether the mean or natural parameter is smooth makes a substantial
difference for the 
accuracy of the estimator. For the Poisson case (right two panels), there is
much less discrepancy. In the case that the mean is smooth, both estimators
appear relatively poor, though the MSE remains small in both cases. The reason
is that the mean and the variance are the same, and both nearly constant. The
difficulty is further exacerbated due to the discreteness of the data and only
a small handful of values with non-negligible probability. 
Therefore, this setting is actually quite challenging. For context, on the
typical dataset, the average
absolute difference between observations 
at neighbouring points is about 2.5 compared with a 0.01 change in the signal.

\subsection{Example applications}
\label{sec:example-applications}

We apply our estimators to two real-world datasets for illustrative purposes.
The first examines Poisson trend filtering for estimating the age-time
hospitalization rates due to COVID-19 in the University of California system.
The second estimates the instantaneous temperature variability over the Northern
hemisphere from publicly available observations.

\subsubsection{UC COVID-19 hospitalization data}
\label{sec:hospital-data}

We analyzed the COVID-19 hospitalization rate within five hospitals in the
University of California system: UC Davis, UC Los Angeles, UC Irvine, UC San
Diego, and UC San Francisco. 
The data is based on 4,730 patients, all 18 years old or greater, that were
admitted between February 12, 2020 to January 6, 2021. 
We aggregate the hospitalization counts at the weekly level---there are 48 weeks
in total---and by age (in 15 bins of 5 years each). 
This results in noisy and sparse hospitalization counts at the week-by-age
level with an average count-per-bin of $6.57$.  
The data was obtained from the authors of \cite{nuno2021covid}, where they
perform a more comprehensive analysis. 
It is used under a data use agreement and has not been made available to the
public due to privacy concerns.

We apply $k=1$ trend filtering with the Poisson exponential family in 2
dimensions to COVID-19 hospitalizations.
We tune the $\lambda$ parameter by minimizing $\mathrm{PUKL}(\hat\theta)$.
One can see the results in \autoref{fig:hospital}, where the smoothed version
is on the left. 
Due to the low average count per bin, trends in hospitalization rate are much
more clearly visible after applying trend filtering. 
We have marked the local maxima in the smoothed signal which produces only 4 
points---this would not have been possible in the raw data.

Some broad trends are clearly visible from \autoref{fig:hospital}.
First, we can see two distinct waves for COVID-19 hospitalizations in
summer 2020 and winter 2020--2021.   
Moreover, we can see that the highest hospitalization rates within the summer
2020 wave are among those aged 50--65, while in the winter 2020--2021 wave the
highest rates are both within the 50--65 age range but also the 80$+$ age 
range. 
This suggests that the age distribution is not stationary, and changes with
successive waves. 
This may be due to a number of factors, such as behavioral shifts and holiday 
effects.

\begin{figure}[t!]
  \centering
  \includegraphics[width=.9\textwidth]{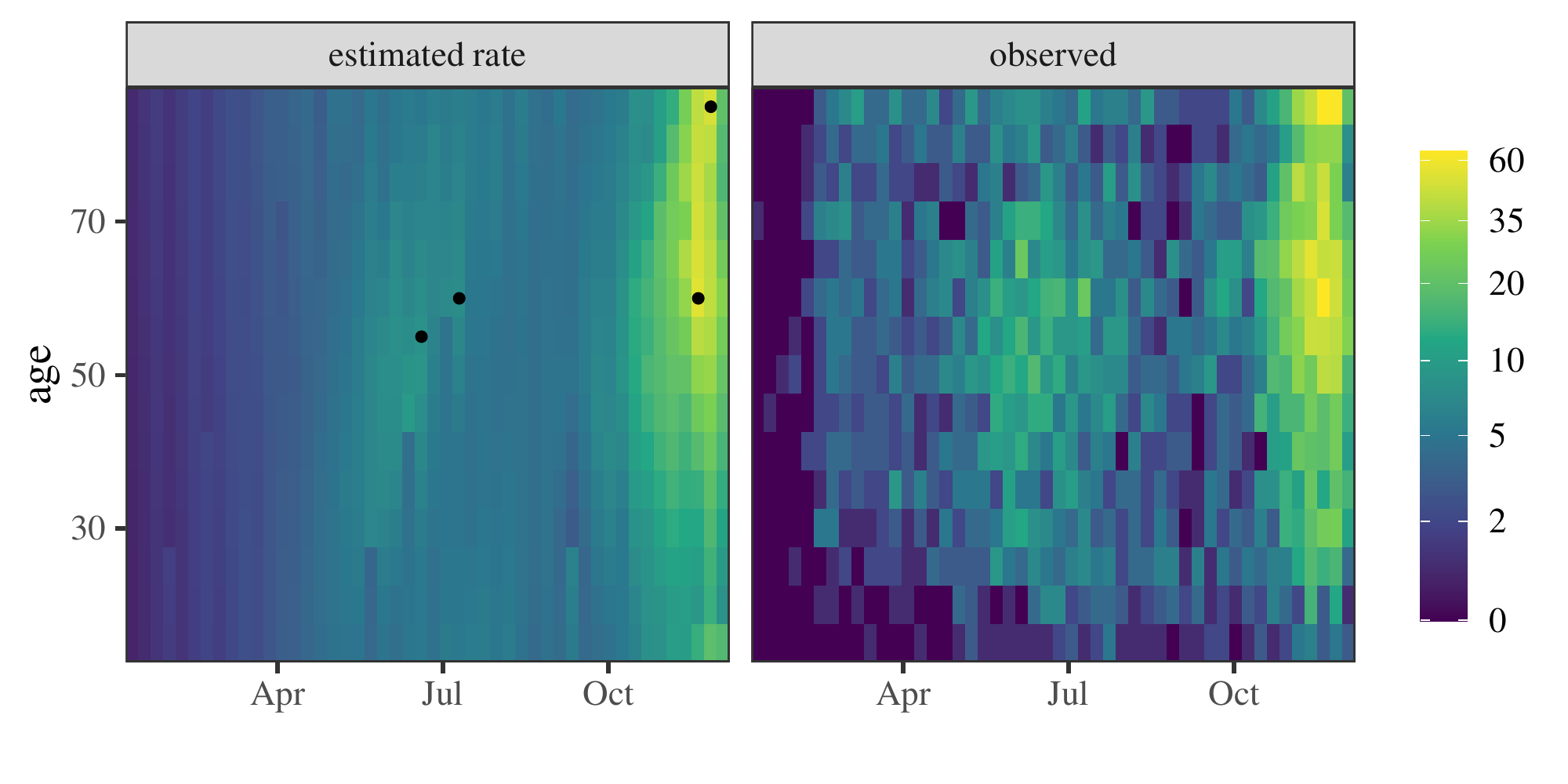}
  \caption{Estimated daily hospitalization rate due to COVID-19 by 5 year age
    group and week in five UC hospitals.  We apply $k=1$ trend filtering with
    the Poisson exponential family (left) to the raw count data (right).}
  \label{fig:hospital}
\end{figure}

\subsubsection{Temperature variability}
\label{sec:temperature-ex}

\begin{figure}
  \centering
  \includegraphics[width=.8\textwidth]{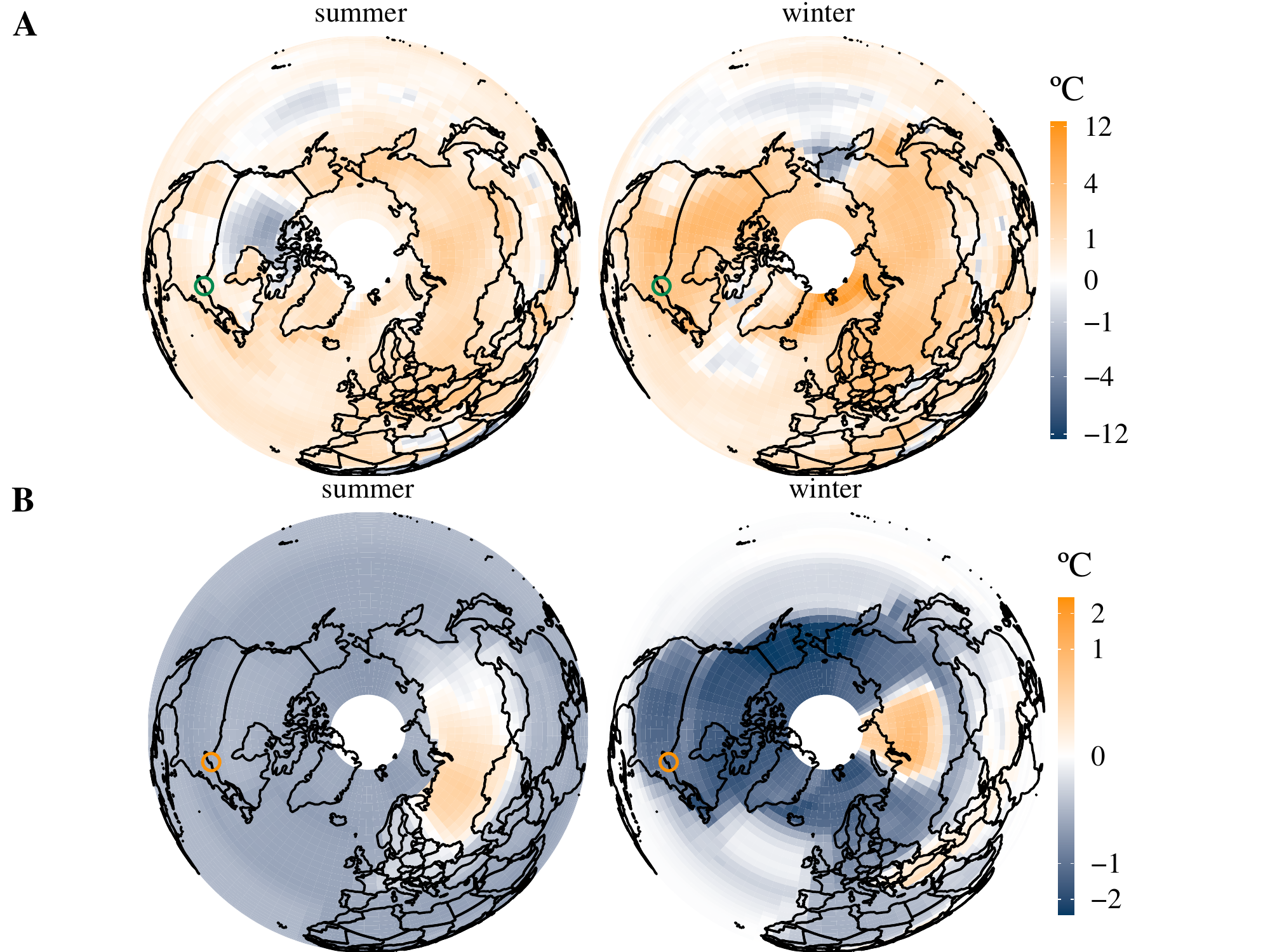}
  \caption{Panel A shows the change in average temperatures observed in the northern
    hemisphere from the 1960s relative to the 2000s in degrees Celsius.
    Panel B shows the change in estimated standard deviation (using the KL trend filter
    with $k=1$ in the temporal dimension and $k=2$ over space) from the 1960s
    relative to the 2000s. Standard deviations were estimated at each
    spatio-temporal grid location before being averaged separately over
    winter/summer months over the appropriate decade.}
  \label{fig:globes}
\end{figure}

\begin{figure}
  \centering
  \includegraphics[width=.9\textwidth]{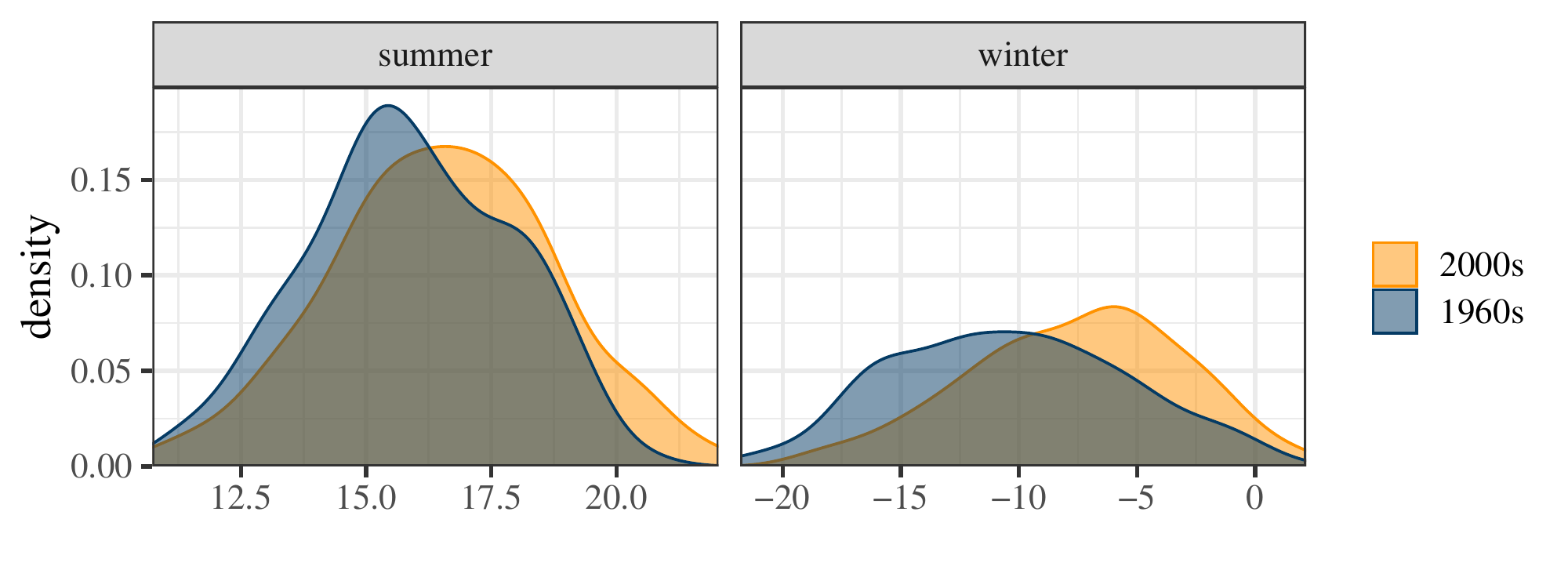}
  \caption{Density estimates for temperatures ($x$-axis, in degrees Celsius) in
    Toronto in the 1960s and 2000s 
    (winter and summer months separately). Consistent with
    \autoref{fig:globes}, the mean
    increases over the period while the standard deviation decreases, resulting
    in the loss of ``colder'' days. This phenomenon is most pronounced in the
    winter.}
  \label{fig:toronto-densities}
\end{figure}

Trends in temperature variability (rather than in mean) have direct
implications for plant and animal life~\citep{huntingford_no_2013},
because changes in variability also impact the probability of extreme weather
events~\citep{VasseurDeLong2014}. 
\citet{hansen_perception_2012} and \citet{huntingford_no_2013}
suggest that adaptation to extremes is more difficult than
to gradual increases in the mean temperature. 
Nevertheless, research examining trends in the volatility of
spatio-temporal climate data is relatively scarce. 
\citet{hansen_perception_2012} studied
changes in 
the standard deviation (SD) of  surface temperatures
at each spatial location relative to that
location's SD over a base period and showed
that these estimates are increasing.
\citet{huntingford_no_2013} took a similar
approach for a different data set. They argued that, while there
is an increase in the SDs from 1958-1970 to 1991-2001, it
is much smaller than found by
\citet{hansen_perception_2012}. \citet{huntingford_no_2013} also computed the
time-evolving global SD from the detrended time-series at each
position and argued that the global SD has been
stable. 

The first row in \autoref{fig:globes} shows the change in mean temperature averaged
over the winter and summer months separately in the 1960s relative to the 2000s
using the ERA 20C dataset~\citep{PoliHersbach2016}. It shows strong increases in
average temperatures in both periods over the majority of the hemisphere.
The second row shows the estimated standard deviations from
the 
KL trend filter over the same period. We use $k=1$ in the temporal dimension and
$k=2$ spatially. These estimated SDs are then averaged over the two periods for
summer and winter separately and we plot the difference. There is a slight
decrease in the SD during the summer and a more pronounced pole-ward decrease
during the winter with the exception of Siberia which shows a dramatic increase
over both periods. To further examine the effect of increasing mean and
decreasing standard deviation, we look at the temperature distribution over both
periods for Toronto, Canada (circled on both maps). Clearly, as shown in
\autoref{fig:toronto-densities}, the distributions 
for both summer and winter have shifted toward higher temperatures in 50 years.
But at the same time, especially in winter, the standard deviation has declined.
Thus, there are far fewer cold days (temperatures between $-10^\circ$C and
$-20^\circ$C) in the 2000s than in the 1960s.

\section{Discussion}
We studied estimation error bounds for two estimators with a trend filtering
penalty on grid graphs. One estimator minimizes squared distance from the mean
and the other maximizes log likelihood. The bounds are more involved, compared
to, say, the homoskedastic sub-Gaussian noise case. Such cumbersome bounds are
due 
to the fact there are many more parameters that influence the estimation error.
We illustrated the bounds in several interesting regimes of signals with
heteroskedastic and homoskedastic noise. We analyzed two datasets with our
models showing the applicability of our methodology to real world problems. We
showed that both estimators achieve minimax optimal error rates in some
scenarios, though unfortunately, addressing all cases remains for future work.

Because our analysis examines the entire class of observations corrupted by
subexponential noise, the result is a general bound on the error for all exponential families. But,
this is a large class, and far from the only way to study the estimation error.
More specific analysis in specific cases will likely result in sharper bounds.
For example, \cite{vandeGeer2020logistic} gets sharper rates for the Bernoulli
family and \cite{brown2010nonparametric} examines a set of 6 families where the
variance can be written as a quadratic function of the mean. However, those
analyses are much less comprehensive than ours.

Other possible extensions are ``mixed'' loss and penalties. One could try to
penalize the mean parameter combined with likelihood loss or the opposite.
Preliminary investigations into the first case revealed similar issues as with
the penalty on the natural parameter, namely an inability to control the error
in the null space of $D$. Another natural avenue for future work would note that
all of these 
(the estimators examined here and the mixed versions) have connections to state
space models in 
time series. So the relationship between trend filtering and Kalman-type filters
may yield new theoretical insights and computational algorithms.

%% file: stef-appendix.tex
\section{Proofs for preliminary results}
\label{sec:app-preliminary}
\subsection{Proof of \autoref{lem:sub-exponential}}
\begin{proof}[Proof of \autoref{lem:sub-exponential}]
Without loss of generality assume $Y$ has mean zero. We have
\begin{align*}
E[e^{s Y}] &= \int e^{s y} h(y) e^{y \theta^{*} - \psi(\theta^{*})} \, dy\\
&= \int h(y) e^{y (s + \theta^{*}) - \psi(s + \theta^{*}) + \psi(s + 
\theta^{*}) - \psi(\theta^{*})} \, dy\\
&= e^{\psi(s + \theta^{*}) - \psi(\theta^{*})} \int h(y) e^{y(s + \theta^{*}) - 
\psi(s + \theta^{*})} \, dy
\end{align*}
Since $\theta^{*} \in \mathrm{interior}(\Theta)$, there is $b$ such that $|s| < 
\frac{1}{b}$ gives $h(y) e^{y(s + \theta^{*}) - \psi(s + \theta^{*})}$ is a 
member of the exponential family and hence integrates to $1$. Therefore the 
above display equals $e^{\psi(s + \theta^{*}) - \psi(\theta^{*})}$.
A Taylor expansion of $\psi(s + \theta^{*}) - \psi(\theta^{*})$ is possible 
because $\theta$ is infinitely differentiable \citep{Brown1986}:
  $$\psi(s + \theta^{*}) - \psi(\theta^{*}) = \nabla \psi(\theta^{*}) s + 
\frac{1}{2} \nabla^{2} \psi(\theta^{*}) s^2 + \frac{1}{2} R(\theta^{*}, s) s^2$$
  where $\lim_{s \to 0} R(\theta^{*}, s) \to 0$.
  Combined with the fact that $E[Y] = \nabla \psi(\theta^{*}) = 0$, we have that
  $$\psi(s + \theta^{*}) - \psi(\theta^{*}) = \frac{1}{2} \left(\nabla^{2} 
\psi(\theta^{*}) +  R(\theta^{*}, s)\right) s^2$$
  Fixing $\delta>0$, we can then choose a $b$, which depends on $\delta$, such 
that $\sup_{|s| < \frac{1}{b}} \left| R(\theta^{*}, s) \right| < \delta$. We 
conclude that there exists a $b$ (where we increase $b$ from our previous 
choice guaranteeing $s + \theta^{*} \in \Theta$ as necessary) such that for all 
$|s| < \frac{1}{b}$
  $$\frac{1}{2} \left(\nabla^{2} \psi(\theta^{*}) - \delta\right) s^{2} \leq 
\psi(s + \theta^{*}) - \psi(\theta^{*}) \leq \frac{1}{2} \left(\nabla^{2} 
\psi(\theta^{*}) + \delta\right) s^{2}.$$
  This gives the second claim of the lemma. Taking $\nu^{2} = \nabla^{2} 
\psi(\theta^{*}) + \delta$ gives
$E[e^{s Y}] \leq e^{\frac{s^{2} \nu^{2}}{2}}$ and proves the result.
\end{proof}
\subsection{Subexponential parameters for some standard distributions}
For a Poission random variable $X$ with mean $\mu$, note that for $s\in \R$,
\begin{equation*}
\E e^{s(X-\mu)} = e^{\mu(e^s-s-1)}.
\end{equation*}
Therefore $ \E e^{s(X-\mu)} \leq e^{\mu s^2}$ for $s$ satisfying $e^s- 1 - s 
\leq  s^2$. Let $s^*$ be the  non-zero solution to $e^x = 1+x+x^2$. Then 
$s^*\approx 1.793$.
From this, we can show that
$$
X-\mu \;\;{\rm is}\;\; {\rm SE}(\nu^2, b) \;\;{\rm with }\;\;  \nu^2=2\mu, 
b=1/s^* \leq 0.55.
$$

\begin{table}
    \caption{\label{tab:orlicz_subexp_params_copy} Sub-exponential parameters 
    for some exponential family distributions}
  \centering
  \begin{tabular}{lccc}
    \toprule
    Distribution & $\psi(\theta)$ 
    & $\nu^2, b$\\
    \midrule
    Poisson (mean$=\mu$) & $e^\theta$ 
    & $2\mu$, \; $0.55$ \\ 
    Exponential (mean$=\mu$) & $-\log(-\theta)$ 
    & $4\mu^2\log\frac{4}{e} $,\; $2\mu$ \\
    $\chi^2_k$ (mean$=k$) & $\log \big(\Gamma(\theta+1) 2^{\theta+1}\big)$ 
     & $4k, 4$ \\
\bottomrule
\end{tabular}
\end{table}


For exponential distribution, we can do a similar calculation to get the results
in \autoref{tab:orlicz_subexp_params_copy}.
For an exponential variable $X$ with mean $\mu$, for $s\in \R,$
\begin{align}
\E e^{s (X-\mu)} &=  \frac{e^{-\mu s}}{1-\mu s}.
\end{align}
We can verify 
that $X-\mu$ is sub-exponential with parameters $(\nu^2,b)$ given in 
\autoref{tab:orlicz_subexp_params_copy}. To arrive at these parameters, we set 
$b=2\mu$ and find the $\nu^2$ of
the form $c\mu^2$ for a constant $c$ such that $\E e^{s(X-\mu)} \leq e^{\nu^2 
s^2/2}$ for $|s| \leq 1/b$.
In a similar fashion, one can also verify the sub-exponential 
parameters for the $\chi^2$ distribution specified in the bottom row of the 
table.

\subsection{Some properties of Sub-exponentials}

\subsubsection*{Tail bounds on linear combinations of sub-exponentials}
We use the following exponentially decaying tail bound for sums of 
sub-exponential variables at multiple places in our proofs.
\begin{lemma}
  \label{lem:linear-comb-sub-exp}
  Let $\nu$
  and $c$ be vectors such that $\epsilon_{i}$ is sub-exponential with
  parameters $(\nu_{i}, c_{i})$. Given a matrix $A \in \R^{n\times r}$,
  assume we have $K$ and $H$ such that
  $\sup_{i=1,..,r} \snorm{\nu \odot A_{i}}_2 \leq K$ and $\sup_{i=1,..,r}
  \snorm{c \odot A_{i}}_{\infty} \leq H$, where $A_{1},\ldots,A_{r}$ are
  the columns of $A$. Then
  \begin{equation}
    P\left(\snorm{A^\top \epsilon}_{\infty} \geq t \right) \leq
    \begin{cases} 2 r \exp\left(-\frac{t^2}{2 K^2}\right) &  t < \frac{K^2}{H}\\
      2 r \exp\left(-\frac{t}{H} + \frac{K^2}{2 H^2} \right) & t \geq
      \frac{K^2}{H}
    \end{cases}
  \end{equation}
\end{lemma}
\noindent
The proof is similar to that of Bernstein inequality from Theorem 2.8.1 in 
\citep{Vershynin2018}.

\begin{proof}[Proof of \autoref{lem:linear-comb-sub-exp}]
  We have
  \begin{align}
    \log \E\left[\exp\left(s \|A^\top
    \epsilon\|_{\infty}\right)\right]
    & = \log \E\left[\exp\left(s
      \max\{|A_{1}^\top\epsilon|,\ldots,|A_{r}^\top\epsilon|\}\right)\right]\\
    \leq \log \E\left[\exp\left(s \sum_{i=1}^{r} |A_{i}^\top
      \epsilon|\right)\right]
    &= \log \E\left[\exp\left(s \sum_{i} |\sum_{j=1}^{n} a_{ij}
      \epsilon_{j}|\right) \right].
  \end{align}

  Note that $A_{i}^{\top} \epsilon$ is mean zero with parameters
  $(\snorm{\nu \otimes A_{i}}_2, \snorm{c \otimes A_{i}}_{\infty})$. This is 
because
  \begin{equation}
    \log \E\left[\exp(s A_{i}^\top \epsilon)\right]
    =\log \E\bigg[\exp\big(s \sum_{j} a_{ij}
        \epsilon_{j}\big)\bigg]
    =\sum_{j} \log \E\left[ \exp(s a_{ij} \epsilon_{j})\right],
  \end{equation}
  by independence of $\epsilon_{j}$.
  When $|s| < \frac{1}{a_{ij} c_{j}}$ for all $j$, which is satisfied when $|s| 
< \frac{1}{\|c \otimes A_{i}\|_{\infty}}$,
  \begin{equation}
    \sum_{j} \log \E\left[ \exp(s a_{ij} \epsilon_{j})\right] \leq
    \sum_{j} \frac{\nu_{j}^2 (s a_{ij})^2}{2}= \frac{\|\nu \otimes
      A_{i}\|_2^2 s^2}{2}.
  \end{equation}

  Therefore, for $|s| < \frac{1}{\sup_{i=1,...,r} \|c \otimes 
A_{i}\|_{\infty}}$,
  \begin{align}
    \log \E\left[ \exp(s \|A^\top \epsilon\|_{\infty})\right]
    &=\log \E\left[ \exp\left(s \max\{|A_{1}^\top \epsilon|,\ldots , 
|A_{r}^\top \epsilon|\}\right) \right]\\
    &\leq \log \sum_{i=1}^{r} \E \left[\exp\left( s | A_{i}^\top 
\epsilon|\right)\right]\\
    &\leq \log \sum_{i} \E\left[\exp\left(s A_{i}^\top \epsilon\right) + 
\exp\left(-s A_{i}^\top \epsilon \right)
    \right]\\
    &\leq \log \left( 2 \sum_{i=1}^{r} \exp\left(\frac{\|\nu \otimes
    A_{i}\|^2 s^2}{2}\right) \right).
  \end{align}
  Therefore, using the Chernoff bound, we have
  \begin{equation}
    P\left(\|A^{\top} \epsilon\|_{\infty} > t\right)
    \leq \exp(-ts)\left(2 \sum_{i=1}^{r} \exp\left(\frac{\|\nu
          \otimes A_{i}\|^2 s^2}{2} \right) \right)
  \end{equation}
  for $|s| < \frac{1}{\sup_{i=1,...,r} \|v_{i} \otimes
    c\|_{\infty}}$, which we minimize in $s$ to get our
  bound. This is intractable, so we  require $\|\nu
  \otimes A_{i}\|_2 \leq K$ and $\|c \otimes A_{i}\|_{\infty} \leq H$
  for all $i \in [r]$.  We then have
  \begin{equation}
    P(\|A^\top \epsilon\|_{\infty} > t) \leq 2 \sum_{i=1}^{r} \exp\left(-ts + 
\frac{\|\nu \otimes A_{i}\|^{2} s^2}{2} \right)
    \leq 2 r \exp\left(-t s + \frac{K^2 s^2}{2}\right).
  \end{equation}
Minimizing in $s$, for $|s | < \frac{1}{H}$, we have $s = t/K^2$ or $1/H$ 
depending on which is smaller. Therefore,
  \begin{equation}
    P\left(\snorm{A^\top \epsilon}_{\infty} \geq t \right) \leq
    \begin{cases} 2 r \exp\left(-\frac{t^2}{2 K^2}\right) &  t < \frac{K^2}{H}\\
      2 r \exp\left(-\frac{t}{H} + \frac{K^2}{2 H^2} \right) & t \geq
      \frac{K^2}{H}
    \end{cases}
  \end{equation}
\end{proof}

We state a few convenient ways of using Bernstein's tail bound inequality on 
linear combinations of sub-exponential random variables. 
Denote the sub-exponential tail bound function
\begin{equation}
\label{eq:subexp_tail}
\phi(t;\nu^2, b) = 2 \exp\bigg( -\frac{1}{2} \min\left\{ \frac{t^2}{\nu^2}, 
\frac{t}{b} \right\} \bigg)
\end{equation}
for \smash{$t\geq 0$} with parameters \smash{$\nu>0, b>0$}.
Note that
if \smash{$\P(|X|>t) \leq \phi(t;\nu^2, b)$} for all \smash{$t\geq 0$}, then  
\begin{equation}
\label{eq:subexp_tail_highprob_bd}
|X| \leq 2(\nu \vee b) u \quad 
\end{equation}
with probability at least \smash{$1-2e^{-u}$} for \smash{$u\geq 1.$} 
\begin{lemma}
\label{lem:linear-comb-sub-exp-vec}
Let \smash{$\epsilon_i$} be independent, mean-zero, sub-exponential variates 
with 
parameters \smash{$(\nu_i^2, b_i)$} for \smash{$i\in[n]$}. Let 
\smash{$a\in\R^n$} be a 
fixed vector.
Let \smash{$\phi$} be the sub-exponential tail bound function defined in 
\eqref{eq:subexp_tail}.
 Then for \smash{$t\geq 0$},
\begin{align}
\label{eq:linear-comb-sub-exp-general}
\P ( |a^\top\epsilon| > t) 
&\leq \phi \big(t; \sum_{i=1}^n a_i^2 \nu_i^2, \max_{i\in[n]} |a_i| b_i \big)\\
\label{eq:linear-comb-sub-exp-tailbd}
&\leq \phi \big(t; \|a\|_2^2 \|\nu\|_\infty^2, \|a\|_\infty \|b\|_\infty \big)
\end{align}
Also, 
\begin{equation}
\label{eq:linear-comb-sub-exp-tailbd-incoherence}
\P ( |a^\top\epsilon| > t) 
\leq \phi \big(t; \|a\|_\infty^2 \|\nu\|_2^2, \|a\|_\infty \|b\|_\infty \big).
\end{equation}
Further, if $\nu=b$, then for $t\geq 1$, with probability at least $1-2e^{-t}$, 
both the following hold:
\begin{align}
\label{eq:linear-comb-sub-exp-highprob}
|a^\top\epsilon| &\leq 2 \|a\|_2 \|b\|_\infty t \\
\label{eq:linear-comb-sub-exp-highprob-incoherence}
|a^\top\epsilon| &\leq 2 \|a\|_\infty \|b\|_2 t
\end{align}
\end{lemma}
\begin{proof}[Proof of \autoref{lem:linear-comb-sub-exp-vec}]
\eqref{eq:linear-comb-sub-exp-general} follows by applying Bernstein's 
inequality 
from Theorem 2.8.1 in \citep{Vershynin2018}. 
The  inequalities \eqref{eq:linear-comb-sub-exp-tailbd}, 
\eqref{eq:linear-comb-sub-exp-tailbd-incoherence} follow from 
\eqref{eq:linear-comb-sub-exp-general} by applying H\"older's inequality to the 
first parameter in different ways.

From \eqref{eq:subexp_tail_highprob_bd}, observe that  for \smash{$t\geq 1$},
\begin{equation}
|a^\top\epsilon| \leq 2 \big( \| a \odot b \|_2 \vee \| a \odot b \|_\infty 
\big) 
t \leq 2 \| a \odot b \|_2 t
\end{equation}
holds with probability at least \smash{$1-2e^{-t}$}, where \smash{$a\odot b \in 
\R^n$} with 
\smash{$(a\odot b)_i = a_i b_i, i\in[n]$}. By applying H\"older's inequality in 
two 
different ways we get the high probability bounds 
\eqref{eq:linear-comb-sub-exp-highprob} and 
\eqref{eq:linear-comb-sub-exp-highprob-incoherence}.
\end{proof}

\subsubsection*{Tail bound on maximum of sub-exponentials}
\begin{lemma}
\label{lem:max_subexp_tailbd}
Suppose $X_i$ are sub-exponential with parameters $(\omega^2,\omega)$ for 
$i\in[m]$.
Then for $t\geq 1$
\begin{equation}
\P \big( \max_{i\in[m]} |X_i| \leq  2\omega (\log 2m + t) \big) \geq 1-2e^{-t}.
\end{equation}
\end{lemma}
\begin{proof}[Proof of \autoref{lem:max_subexp_tailbd}]
Denote \smash{$X_{m+j} = -X_j$} for \smash{$j\in[m]$}. 
By union bound, for \smash{$u>0$},
\[
\P(\max_{j\in[m]} |X_j| > u) = \P(\max_{j\in[2m]} X_j > u) \leq \sum_{j=1}^{2m} 
\P(X_j > u)
\leq 4m \exp\Big( -\Big\{ \frac{u^2}{2\omega^2} \wedge 
\frac{u}{2\omega}\Big\}\Big).
\]
Set \smash{$u=2\omega (\log 2m + t)$} to get the desired bound.
\end{proof}

\section{Proofs of upper bounds}
\label{sec:app-error-bounds}

\subsection{Proof of \autoref{thm:ex_risk_bd}}
\label{sec:app_ex_risk_bd}
We first state a basic inequality.
\begin{lemma}[Basic inequality]
\label{lem:basic-ineq-2}
Let \smash{$R$} be as defined in \autoref{sec:mle_null_space_penalty} and let \smash{$\hat\theta$} 
be the estimate 
in \eqref{eq:estimator}. Then,
\[
R(\hat\theta) - R(\theta^*) + \lambda_2 \|P_\cN \hat\theta\|_2 + \lambda_1 
\|D\hat\theta\|_1\leq \frac{1}{n} \epsilon^\top (\hat \theta - 
\theta^*) + \lambda_2 \| P_\cN\theta^*\|_2 + \lambda_1 \| D \theta^* \|_1.
\]
Further, this inequality is true if we replace \smash{$\hat\theta$} with 
\smash{$\hat\theta_t = t 
\hat\theta + (1-t) \theta^*$} for any \smash{$t\in[0,1].$}
\end{lemma}
\begin{proof}[Proof of \autoref{lem:basic-ineq-2}]
Optimality of \smash{$\hat\theta$} and the equality 
\smash{$R_n(\theta) - R(\theta) = -\frac{1}{n}\epsilon^\top\theta$} gives
\begin{align*}
R_n(\hat\theta)  + \lambda_2 \|P_\cN \hat\theta\|_2 + \lambda_1 \|D\hat\theta\|_1 
&\leq   R_n(\theta^*) + \lambda_2 \| P_\cN\theta^*\|_2 + \lambda_1 \| D \theta^* 
\|_1 \\
\Leftrightarrow R(\hat\theta) - \frac{1}{n}\epsilon^\top\hat\theta + \lambda_2 
\|P_\cN \hat\theta\|_2 + \lambda_1 \|D\hat\theta\|_1
 &\leq   R(\theta^*) - \frac{1}{n}\epsilon^\top\hat\theta^*+ \lambda_2 \| 
P_\cN\theta^*\|_2 + \lambda_1 \| D \theta^* \|_1
\end{align*}
This is equivalent to the main statement in the lemma.
The  inequality for \smash{$\hat\theta_t$} follows from the fact that 
\smash{$ \theta \mapsto 
R_n(\theta) + \lambda_1  \|D\theta\|_1 + \lambda_2 \|P_\cN \theta\|_2$} is convex.
\end{proof}

\begin{proof}[Proof of \autoref{thm:ex_risk_bd}]
For brevity, define the shorthand
\begin{equation}
\tau(\theta, \lambda_1,\lambda_2) =   \lambda_1 \|D\theta\|_1 + \lambda_2 \| 
P_{\cN} \theta \|_2
\end{equation}
for $\theta\in \R^n, \lambda_1, \lambda_2 \ge 0$ 
.
From the basic inequality in \autoref{lem:basic-ineq-2},
\begin{equation}
\label{eq:basic_ineq_ex_risk_bd}
R(\hat\theta) - R(\theta^*) + \tau(\hat\theta,\lambda_1,\lambda_2)
\leq \frac{1}{n} \epsilon^\top (\hat\theta-\theta^*)+ 
\tau(\theta^*,\lambda_1,\lambda_2).
\end{equation}
Applying \autoref{lem:emp_proc_combined} with $J = [k+1]^d$,
\begin{align}
\frac{1}{n}\epsilon^\top (\hat\theta-\theta^*) 
&\leq \frac{A}{n} \| P_\cN(\hat\theta-\theta^*) \|_2 + \frac{B}{n} \| D 
(\hat\theta-\theta^*) \|_1\\
&= \tau(\hat\theta-\theta^*, B/n, A/n)
\end{align}
where 
$A = 2t\mu \sqrt{\frac{\kappa}{n}} \big( \| \nu \|_2 \vee \|b\|_\infty\big),$
$B = 2t\left( 
	\min\left\{ \|\nu\|_\infty  L_{J,2}, 
					  \|\nu\|_2 L_{J,1} \right\}
	\vee  \| b \|_\infty L_{J,1}\right),
 $
for $t\geq 1$, on an event $\Omega(t)$ with probability at least 
$1-2(m+\kappa)e^{-t}.$
Here $m$ is the number of rows of $D$ and $\kappa = (k+1)^d$.
Therefore, on the event $\Omega(t)$,
\begin{align}
R(\hat\theta) - R(\theta^*) + \tau(\hat\theta,\lambda_1,\lambda_2)
&\leq \tau(\hat\theta-\theta^*,B/n, A/n ) + \tau(\theta^*,\lambda_1,\lambda_2)\\
&\leq \tau(\hat\theta,B/n, A/n ) + \tau(\theta^*,B/n, A/n ) + 
\tau(\theta^*,\lambda_1,\lambda_2)
\end{align}
where we used triangle inequality in the second line. 
If we choose \smash{$\lambda_1 \geq 2 B/n, \lambda_2 \geq 2 A/n$},
by linearity of $\tau$ in regularization parameters, we have 
\smash{$\tau(\theta,B/n, A/n )\leq \tau(\theta, \lambda_1/2, 
\lambda_2/2 ) = \frac{1}{2} \tau(\theta, \lambda_1, \lambda_2 )$} for any 
$\theta \in \R^n$. Therefore
\begin{equation}
R(\hat\theta) - R(\theta^*) + \frac{1}{2}\tau(\hat\theta, \lambda_1, \lambda_2)
\leq 
\frac{3}{2}\tau(\theta^*,\lambda_1, \lambda_2)
\end{equation}
As $\theta^*$ minimizes $R$, we should have 
\smash{$R(\hat\theta) \geq R(\theta^*)$}. 
That means, both  the terms \smash{$R(\hat\theta) - R(\theta^*)$} and 
\smash{$\frac{1}{2}\tau(\hat\theta, \lambda_1, \lambda_2)$} are non-negative. 
Therefore, 
\begin{align}
R(\hat\theta) - R(\theta^*) &\leq \frac{3}{2}\tau(\theta^*,\lambda_1,  
\lambda_2) \quad \text{and}\\
\frac{1}{2} \tau(\hat\theta, \lambda_1, \lambda_2) &\leq \frac{3}{2} 
\tau(\theta^*,\lambda_1, \lambda_2).
\end{align}
This completes the proof as these inequalities hold with probability 
\smash{$\P(\Omega(t)) \geq 1 - 2(m+\kappa) e^{-t}.$}
\end{proof}

\subsection{Proofs of Corollaries to \autoref{thm:ex_risk_bd}}
\label{sec:proof_mle_bound_coro}
\begin{proof}[Proof of \autoref{cor:poisson_weak_hetero}.]
	
	We have the following bounds,
	\begin{align}
		&\| \nu \|_2 = O(\sqrt n)\\
		&\| \nu \|_\infty, \|b \|_\infty = O(1)
	\end{align}
	When $d=1$ then $\alpha = 2$, and so
	\[
	\frac 1n (\|\nu\|_2 + \|b\|_\infty) = O(n^{-1/2}).
	\]
	When $d = 2$, $\alpha = 1$ and $\gamma_1 = \log n$, $\gamma_2 = 1$, thus
	\[
	\| b \|_\infty  n^{-\alpha} \gamma_1 + \min\{ \| \nu \|_\infty n^{-1/2} 
	\gamma_2,\ \| \nu \|_2 n^{-\alpha} \gamma_1\} = O(n^{-1/2} \log n).
	\]
	When $d=3$ then $\alpha = 2/3$ and $\gamma_1 = \gamma_2 = 1$,
	\[
	\| b \|_\infty  n^{-\alpha} \gamma_1 + \min\{ \| \nu \|_\infty n^{-1/2} 
	\gamma_2,\ \| \nu \|_2 n^{-\alpha} \gamma_1\} = O(n^{-1/2}).
	\]
	When $d = 4$ then $\alpha = 1/2$ and $\gamma_1 = 1, \gamma_2 = \log^{1/2} n$ 
	and 
	\[
	\| b \|_\infty  n^{-\alpha} \gamma_1 + \min\{ \| \nu \|_\infty n^{-1/2} 
	\gamma_2,\ \| \nu \|_2 n^{-\alpha} \gamma_1\} = O(n^{-1/2} \cdot \log^{1/2} 
	n).
	\]
	Finally, when $d > 4$ then $\alpha = 2/d < 1/2$ and
	\[
	(\|\nu \|_\infty + \|b\|_\infty) n^{-\alpha} = O(n^{-2/d}).
	\]

	Next we show that the example signal satisfies the necessary conditions.

	Consider the Poisson distribution where the natural parameter vector 
	$\theta^*$ is constrained.
	For $i = (i_1, \ldots, i_d) \in [N]^d$, let
	\[
	\theta^*_i = \frac 2N \sum_{j=1}^d |i_j - N/2|. 
	\]
	Then the mean vector is
	\[
	\beta^*_i = \prod_{j=1}^d \exp\left(\frac 2N |i_j - N/2|\right). 
	\]
	Because the distribution is Poisson, we have $\| b \|_\infty$ is constant while
	$\nu_i^2 = 2 \beta_i^*$ (see \autoref{tab:orlicz_subexp_params}). 
	Thus, $\| \nu \|_\infty = \sqrt{2} e^{d/2}$ which is achieved at $i = 
	(0,\ldots,0)$.
	The canonical scaling holds for $\| D \theta^*\|_1 \lesssim n^{1-\alpha}$ with 
	$k=1$ because there are on the order of $N^{d-1}$ points at which the Laplacian 
	is non-zero and they are on the order of $1/N$. \qedhere
\end{proof}

\begin{proof}[Proof of \autoref{cor:MLE_strong_hetero}]
	For $d=1$ we have that $\alpha = 1$, and $\gamma_1 = \log n$, $\gamma_2 = 1$, 
	thus
	\[
	\| b \|_\infty  n^{-\alpha} \gamma_1 + \min\{ \| \nu \|_\infty n^{-1/2} 
	\gamma_2,\ \| \nu \|_2 n^{-\alpha} \gamma_1\} = O(n^{c-1/2}).
	\]
	For $d = 2$ we have that $\alpha = 1/2$ and $\gamma_1 = 1, \gamma_2 = 
	\log^{1/2} n$ and 
	\[
	\| b \|_\infty  n^{-\alpha} \gamma_1 + \min\{ \| \nu \|_\infty n^{-1/2} 
	\gamma_2,\ \| \nu \|_2 n^{-\alpha} \gamma_1\} = O(n^{-1/2} \cdot \log^{1/2} 
	n).
	\]
	For $d > 2$ we have that $\alpha = 1/d < 1/2$ and $\gamma_1 = \gamma_2 = 1$, 
	thus
	\[
	( \| \nu\|_\infty + \|b\|_\infty ) n^{-\alpha} = O(n^{c - 1/d}).
	\]

	To show that the specified signal satisfies the necessary properties, let
  \smash{$d>1, k=0$} and $c>0$. Consider the Exponential distribution with  
  natural parameter
  \begin{equation} 
    \label{eq:k0thetaex}
    \theta^*_i = -n^{-c} \one\{i=0\} - n^{1-1/d} \one\{i \neq 0\}.
  \end{equation}
  where $i$ indexes the lattice.
  We have that for $k=0$, $\|D \theta^* \|_1 \leq  d (n^{1-1/d} - n^{-c}) \asymp 
  n^{1-\alpha}$,
  so the canonical scaling holds. 
  We apply MLE trend filtering with $k=0$.
  From \autoref{tab:orlicz_subexp_params},
  we have 
  that $\| \nu \|_\infty, \| b \|_\infty \leq 2 n^c$  and $\| \nu \|_2^2 \leq 2 
  (n^{2c} + n^{1/d -1})$. \qedhere
\end{proof}

\subsection{Uniform risk bound with null space penalty}
\label{sec:app-uniform_risk_null_penalty}

\begin{proof}[Proof of \autoref{prop:uniform_risk_null_penalty}]
From the definitions of $R, R_n$, 
\begin{equation*}
|R(\theta) - R_n(\theta)| = \frac{1}{n} | \epsilon^\top \theta |.
\end{equation*}
Applying \autoref{lem:emp_proc_combined} with $J = [k+1]^d$, we get
\begin{equation*}
| \epsilon^\top \theta | \leq  A_n \|P_\cN \theta \|_2 + B_n\| D \theta \|_1 
\end{equation*}
where 
$A_n = 2t\mu \sqrt{\frac{\kappa}{n}} \big( \| \nu \|_2 \vee \|b\|_\infty\big),$
$B_n = 2t\left( 
	\min\left\{ \|\nu\|_\infty  L_{J,2}, 
					  \|\nu\|_2 L_{J,1} \right\}
	\vee  \| b \|_\infty L_{J,1}\right),
 $
with probability at least $1- 4nd e^{-t}$, for $t\ge 1$.
Here $\kappa = (k+1)^d$ and we used the fact that $m < dn$.
By definition of $\Theta$, $\theta$ should satisfy $\|D\theta\|_1 \le c_n 
n^{1-\alpha}$
 and $\|P_\cN \theta\|_n \le a_n$.
Therefore,
\begin{equation}
| \epsilon^\top \theta | \leq 
A_n a_n \sqrt{n} + B_n c_n n^{1-\alpha}
\end{equation}
From the assumptions $\| \nu \|_\infty, \|b\|_\infty \leq c$, we can write 
$A \leq 2 t \mu c \sqrt{\kappa}.$
From \autoref{lem:kron_tf_evals}, for $p\geq 1$,
$
L_{\ell,p}^p \leq c_1 n^{ (p\alpha - 1)_+} (\log n)^{\mathbf{1} \{p\alpha=1\} 
}.
$
This yields the following bound on $B_n:$
\begin{equation}
B_n \leq 2 t c_1 c\gamma n^{(\alpha - \frac{1}{2} )_+}.
\end{equation}
Therefore, with probability at least \smash{$1-4 nd e^{-t}$},
\begin{align*}
\frac{1}{n}| \epsilon^\top \theta | 
&\leq 
c_2 t c \big( a_n n^{-\frac{1}{2}} + 
	c_n \gamma n^{-\alpha} n^{(\alpha - \frac{1}{2} )_+}  \big)\\
&= c_2 t c \big( a_n n^{-\frac{1}{2}} + 
	c_n \gamma n^{-\min\{\alpha, \frac{1}{2}\}}  \big)
\end{align*}
for a constant \smash{$c_2$} depending only on \smash{$k,d$}.
This is sufficient to show the desired bound.
\end{proof}

\subsection{Proof of \autoref{thm:mle_ls_k0}}
\label{sec:mle_ls_k0_pf}
\begin{proof}[Proof of \autoref{thm:mle_ls_k0}]
	Writing the KKT conditions, $\hat\theta$ and $\hat\beta$ are solutions to 
	\eqref{eq:mle1} and \eqref{eq:lsq} iff
	\begin{align}
		\label{eq:kkt_mle}
		\psi'(\hat \theta) - y + n \lambda D^\top S(D\hat\theta) &\ni 0\\
		\label{eq:kkt_ls}
		\hat\beta - y + n \lambda D^\top S(D\hat\beta) &\ni 0
	\end{align}
	where \smash{$S(u)$} is the set of subgradients of \smash{$x \mapsto 
		\|x\|_1$}. 
	\smash{$S(u)$} depends only \smash{$\mathrm{sgn}(u)$}. 
	As \smash{$\psi'$} is a strictly increasing function, for any \smash{$a,b\in 
		\R$}, 
	\smash{$\mathrm{sgn}(\psi'(a)-\psi'(b)) = \mathrm{sgn}(a-b).$}
	Therefore
	\begin{equation}
		\mathrm{sgn}(D\psi'(\hat\theta)) = \mathrm{sgn}(D\hat\theta),
	\end{equation}
	and hence the subgradients \smash{$S(D\psi'(\hat\theta)) = S(D\hat\theta)$}.
	Plugging this in \eqref{eq:kkt_mle}, 
	we see that the KKT conditions for the least squares problem are satisfied by
	\smash{$\psi'(\hat\theta)$} and therefore it is a solution to the least
	squares problem \eqref{eq:lsq}.  
	The solution to the least squares optimization problem \eqref{eq:lsq} is 
	unique
	because the objective is strictly convex. Therefore, by definition of
	\smash{$\hat\beta$}, 
	\smash{$
		\hat\beta = \psi'(\hat\theta).
		$}
\end{proof}

\subsection{Proof of \autoref{thm:lsq_rates}}
\label{sec:lsq_rates_proof}

\begin{proof}[Proof of \autoref{thm:lsq_rates}]
	
	The proof follows the strategy in Theorem 6 in \citet{WangSharpnack2016}. 
	
	Abbreviate $\hat\delta =\hat\beta-\beta^*$.
	From the optimality in the definition of $\hat\beta$, 
	\begin{equation*}
		\frac{1}{2n}\| y - \hat\beta \|_2^2 + \lambda \| D\hat\beta\|_1 
		\leq 
		\frac{1}{2n}\| y - \beta^* \|_2^2 + \lambda \| D\beta^*\|_1 
	\end{equation*}
	Rearranging and substituting $y = \beta^* + \epsilon$,
	\begin{equation*}
		\frac{1}{2n}\|  \hat\beta - \beta^* \|_2^2 
		\leq 
		\frac{1}{n}\epsilon^\top (\hat \beta-\beta^*) + 
		\lambda \| D\beta^*\|_1 - \lambda \| D\hat\beta\|_1.
	\end{equation*}
	Bound the empirical process term on the right hand side using 
	\autoref{lem:emp_proc_combined}. By \autoref{lem:emp_proc_combined}, for 
	$t\geq 
	1$ and $J \subset [N]^d$, 
	the following holds with probability at least $1-2(m+|J|) e^{-t}:$ 
	\begin{equation}
		\frac{1}{2n}\|  \hat\beta - \beta^* \|_2^2 
		\leq 
		\frac{A}{n}  \| P_{J} (\hat\beta-\beta^* )\|_2 + \frac{B}{n}  \| D 
		(\hat\beta-\beta^*) \|_1
		+
		\lambda \| D\beta^*\|_1 - \lambda \| D\hat\beta\|_1
	\end{equation}
	where 
	$A = 2t\mu \sqrt{\frac{|J|}{n}} \big( \| \nu \|_2 \vee \|b\|_\infty\big),$
	$B = 2t\left( 
	\min\left\{ \|\nu\|_\infty  L_{J,2}, 
	\|\nu\|_2 L_{J,1} \right\}
	\vee  \| b \|_\infty L_{J,1}\right).
	$
	Applying Young's inequality on the first term and setting $\lambda \geq 
	\frac{B}{n}$,
	\begin{align}
		\frac{1}{2n}\|  \hat\beta - \beta^* \|_2^2 
		&\leq 
		\frac{1}{4n} \| \hat\beta-\beta^* \|_2^2 + \frac{A^2}{n} +
		\lambda \| D (\hat\beta-\beta^*) \|_1
		+
		\lambda \| D\beta^*\|_1 - \lambda \| D\hat\beta\|_1\\
		&\leq \frac{1}{4n} \| \hat\beta-\beta^* \|_2^2 + \frac{A^2}{n} +
		2\lambda \| D \beta^*\|_1
	\end{align}
	We used triangle inequality on the penalty terms to get the second line. 
	Canceling terms,
	\begin{equation}
		\frac{1}{n}\|  \hat\beta - \beta^* \|_2^2 
		\leq \frac{4A^2}{n} +
		8\lambda \| D \beta^*\|_1.
	\end{equation}
	This bound holds with probability at least 
	\smash{$1-2(m+|J|) e^{-t} \geq 1 - 4nd e^{-t}$}, and so the proof is complete.
\end{proof}

\subsection{Proofs of Corollaries to \autoref{thm:lsq_rates}}
\label{sec:lsq_homosked_proof}
Denote $\sigma^2 = \frac{1}{n} (\| \nu \|_2^2 \vee \|b\|_\infty)$.
	From \autoref{thm:lsq_rates}, for any $J \subset [N]^d$ containing $[k+1]^d$,
	assuming the scaling $\|D\beta^* \|_1 = O( n^{1-\alpha} )$,
\begin{equation}
	\label{eq:lsq_rates_result}
	\frac{1}{n}\snorm{ \hat\beta - \beta^* }_2^2 
	= O_\P \left( \frac{|J| t^2 \sigma^2 }{n} + \frac{t B_n}{n^\alpha}\right)
\end{equation}
where $t=\log n$, 
\begin{equation}
	\label{eq:Bn}
B_n = 2t\left( 
\min\left\{ \|\nu\|_\infty  L_{J,2}, 
\|\nu\|_2 L_{J,1} \right\}
\vee  \| b \|_\infty L_{J,1}\right).
\end{equation}
 Compared to the bound in \autoref{thm:lsq_rates}, additional $\log n$ 
 factors 
 are incurred when translating from the high-probability 
statement to $O_\P$ notation.
$B_n$ can be bound more explicitly by writing down bounds for $L_{J,1}, 
L_{J,2}$
using \autoref{lem:kron_tf_evals}. For $r \in [1, N\sqrt{d}]$, we can write
\begin{equation}
	\label{eq:LJ2}
	L_{J,2}^2 \leq \begin{cases}
		c \mu^2 \gamma_2^2  & \alpha \leq 1/2, J = [k+1]^d\\
		c \mu^2 (n/r^d)^{2\alpha - 1} & \alpha > 1/2, J = \{ i \in  [N]^d : \| 
		(i - k - 2)_+ \|_2 < r\}
	\end{cases}
\end{equation}
and 
\begin{equation}
	\label{eq:LJ1}
	L_{J,1} \leq \begin{cases}
		c \mu^2 \gamma_1  & \alpha \leq 1, J = [k+1]^d\\
		c \mu^2 (n/r^d)^{\alpha - 1} & \alpha > 1, J = \{ i \in  [N]^d : \| 
		(i - k - 2)_+ \|_2 < r\}.
	\end{cases}
\end{equation}
where \smash{$\gamma_p = \log^{1/p} (n)$} if \smash{$p\alpha=1$} and 
\smash{$1$} otherwise.


\begin{proof}[Proof of \autoref{cor:lsq_homosked}]
\textbf{Case $\alpha \leq 1/2$:}
	Set $\alpha \leq 1/2, J = [k+1]^d$ in \eqref{eq:LJ1}, \eqref{eq:LJ2},
	plugin the resulting bounds for $L_{J,1}, L_{J,2}$ in equation \eqref{eq:Bn} :
	\begin{equation}
		\label{eq:B_with_L_substituted}
		B_n =  O(\min\{\|\nu\|_\infty \gamma_2, \| \nu\|_2 \gamma_1 \} \vee 
		\|b\|_\infty \gamma_1) t.
	\end{equation}
	
	Then use 
	the assumptions 
	$\| \nu \|_\infty, \| b \|_\infty \leq \omega$,
	to write
	\smash{
$
B_n = O(t \omega \gamma_2)
$
}
where \smash{$t = \log n$}.
Plug this expression for \smash{$B_n$} in \eqref{eq:lsq_rates_result}, again 
use the 
assumption that \smash{$\| \nu \|_\infty, \| b \|_\infty \leq \omega$}, to write
	\begin{equation}
	\frac{1}{n}\snorm{ \hat\beta - \beta^* }_2^2 
	= O_{\P} \left( \frac{t^2 \omega^2}{n}
			+ \frac{t \omega \gamma_2}{n^\alpha}   \right).
\end{equation}
	
\noindent
\textbf{Case $\alpha > 1/2$:}
We can write
\[
B_n 
= 2t\left( 
\min\left\{ \|\nu\|_\infty  L_{J,2}, 
\|\nu\|_2 L_{J,1} \right\}
\vee  \| b \|_\infty L_{J,1}\right) 
\leq 2t \big( \| \nu \|_\infty L_{J,2} + \| b \|_\infty L_{J,1} \big)
\leq 2t\omega (L_{J,2} + L_{J,1}).
\]
Let $J = \{ i \in [N]^d : \| (i-k-2)_+\|_2 < r \}$ for an $r$ to be 
chosen later from $[1, \sqrt{d} N]$.
Plugging in the bounds for $L_{J,1}, L_{J,2}$ from \eqref{eq:LJ1}, 
\eqref{eq:LJ2} with $\alpha > 1/2$, and then using 
\eqref{eq:lsq_rates_result},
	\begin{equation}
		\label{eq:mse_bd_simplified}
		\frac{1}{n}\snorm{ \hat\beta - \beta^* }_2^2 = O_\P \left( \frac{(r+k+2)^d 
		t^2}{n} 
		\omega^2 +\frac{t}{n^\alpha} \left( \omega 
		(n/r^d)^{\alpha-1/2}\gamma_2 + 
		\omega (n/r^d)^{(\alpha-1)_+} \gamma_1 \right) \right)
	\end{equation}
	where \smash{$t = \log n$}.
 Select \smash{$r$} such that
	\[
	\frac{r^d t^2}{n} \omega^2 \asymp \frac{t \omega}{n^\alpha} 
	(n/r^d)^{\alpha-1/2}.
	\]
	Then the following is sufficient,
	\[
	r^d = \left\lfloor n (n^\alpha t \omega)^{-2/(2\alpha + 1)} \right \rfloor.
	\]
and the following condition ensures that this choice of $r$ is in $[1, \sqrt{d} 
N]$:
	\[
	n^{-\alpha} \le t \omega \le \sqrt n.
	\]
	Plugging this choice of $r$, the first two terms in 
	\eqref{eq:mse_bd_simplified} are bounded by
	\[
	c_1 \frac{r^d t^2}{n} \omega^2 = c_2 (t \omega)^2 (n^\alpha t 
	\omega)^{-2/(2\alpha + 1)} \le c_2 \left(\frac{t^2 \omega^2}{n} 
	\right)^{2\alpha/(2\alpha+1)}
	\]
	where $c_1, c_2$ are universal constants.
	Furthermore, the remaining term is bounded by
	\[
	\frac{t}{n^\alpha} \omega \gamma_1 \textrm{ if }  \alpha \le 1 \quad \textrm{ 
		and } \quad
	n^{-\frac{3 \alpha}{2 \alpha + 1}} (\omega t)^{\frac{4 \alpha - 1}{2 \alpha 
			+1}} \textrm{ if } \alpha > 1.
	\]
	When $t\omega \geq n^{-\alpha}$, 
	$n^{-\frac{3 \alpha}{2 \alpha + 1}} 
		(\omega t)^{\frac{4 \alpha - 1}{2 \alpha +1}} 
	\leq (t^2\omega^2/n)^\frac{2\alpha}{2\alpha+1}$ 
	and so the desired bound holds. \hfill\qedhere
\end{proof}

\begin{proof}[Proof of \autoref{cor:mean_tf_weak_hetero}]
  In both the Poisson and Exponential cases $\| \nu \|_\infty, \| b \|_\infty = O(1)$.
  For $d = 1,2,3$ we have that $\alpha > 1/2$ and 
  \[
    \Big(\frac{\omega^2  \log^2 n}{n} \Big)^{\frac{2\alpha}{2\alpha +1}} + \frac{\omega\gamma_1 \log n}{n^\alpha} = O\left( \left( \frac{\log^2 n}{n} \right)^{\frac{2 \alpha}{2 \alpha + d}} \right).
  \]
  For $d = 4$, $\alpha = 1/2$,
  \[
    \frac{\omega^2 \log^2 n}{n} + \frac{\omega\gamma_2 \log n}{n^\alpha} = O\left( \frac{\log^{3/2} n}{n^\alpha} \right).
  \]
  For $d \ge 5$, $\alpha < 1/2$,
  \[
    \frac{\omega^2 \log^2 n}{n} + \frac{\omega\gamma_2 \log n}{n^\alpha} = O\left( 
      \frac{\log n}{n^\alpha} \right).
  \]

  To show that the example signal satisfies the conditions, consider the Poisson
  and Exponential families where the mean parameter is  
  constrained. 
  Consider a grid graph with width $N$ and dimension $d$, so that $n = N^d$.
  For $i = (i_1, \dots, i_d) \in [N]^d$, let 
  \[
    \beta_i^* = \frac dN + \frac 2N \sum_{j=1}^d |i_j - N/2|.
  \] 
  For the Poisson distribution $\nu^2_i \asymp \beta^*_i$ hence $\| \nu\|_\infty 
  = O(1)$.
  Similarly, for the Exponential distribution $\|\nu\|_\infty, \| b\|_\infty = 
  O(1)$.\qedhere
  
\end{proof}

\begin{corollary}
	\label{cor:mtf_grids_hetero}
	Let \smash{$\sigma = \max\{ \| \nu \|_2,\ \| b \|_\infty \} / \sqrt{n}$}, and
	\smash{$\sigma_\infty = \max\{ \| \nu \|_\infty,\ \| b \|_\infty \}$}.
	Suppose \\ \smash{$\|D \beta^* \|_1 \lesssim n^{1-\alpha}$}.
  If $\alpha \leq 1/2$, then the estimator $\hat \beta$ in 
		\autoref{thm:lsq_rates} satisfies
		\begin{equation}
		\label{eq:beta_cases_alpha_leq_half}
			\frac{1}{n}\| \hat \beta - \beta^* \|_2^2 = O_\P \left( 
			\frac{\sigma^2 \log^2 n }{n} + \frac{\sigma_\infty\gamma_2\log n}{n^\alpha}
			\right).
		\end{equation}
    If \smash{$\alpha > 1/2$} and \smash{$ \sigma^2 / \sigma_\infty \lesssim 
    \sqrt{n} / \log n$}, then
    \begin{equation}
      \label{eq:beta_cases_alpha_gtr_half1}
      \frac{1}{n}\| \hat \beta - \beta^* \|_2^2 = O_\P \left( 
        \Big[ \frac{\sigma^2 \log^2 n }{n} \Big]^{\frac{2\alpha}{2\alpha + 1}} 
        \Big[\frac{\sigma_\infty}{\sigma} \Big]^{\frac{2}{2\alpha+1}} + 
        \frac{\sigma_\infty\gamma_1\log n}{n^\alpha}
      \right).
	\end{equation} 
	Simultaneously, if \smash{$\alpha > 1/2$},
		\begin{equation}
		\label{eq:beta_cases_alpha_gtr_half2}
		\frac{1}{n}\| \hat \beta - \beta^* \|_2^2 = 
		\begin{cases}
			O_\P \Big( \frac{\sigma^2 \log^2 n }{n} + 
				 n^{-\nicefrac{1}{2}} 
				 \big( \sigma_\infty \wedge \sigma\gamma_1 n^{1-\alpha}  \big) 
				 \log n
				\Big) 
				& \text{ if } \alpha \leq 1\\
		 O_\P \Big( 
			\Big[ \frac{\sigma^2 \log^2 n }{n} \Big]^{1 - \frac{1}{2\alpha}} 
			+ 
			\frac{\sigma_\infty\log n}{n^\alpha}
			\Big)
			& \text{ if } \alpha > 1, \sigma^2 \lesssim n/\log^2 n.
		\end{cases}
	\end{equation}
\end{corollary}

In some situations we can get improved results using \eqref{eq:beta_cases_alpha_gtr_half2}, particularly in situations when $\sigma \lesssim \sigma_\infty$.
This can happen for the Poisson family when the signal $\beta^*$ is dominated by
a few components.

\begin{proof}[Proof of \autoref{cor:mtf_grids_hetero}]

	Throughout let $t = \log n$.
	Start from the bound \eqref{eq:lsq_rates_result}:
	\begin{equation}
		\frac{1}{n}\snorm{ \hat\beta - \beta^* }_2^2 
		= O_\P \left( \frac{|J| t^2 \sigma^2 }{n} + \frac{t B_n}{n^\alpha}\right)
	\end{equation}
	 In the case $\alpha \leq 
	1/2$, set $J=[k+1]^d$ and recall the bound \eqref{eq:B_with_L_substituted} 
	for $B_n$. This gives the desired result in this case. 	In the other case of 
	\smash{$\alpha > \frac 12$}, we prove the bounds 
	\eqref{eq:beta_cases_alpha_gtr_half1} and 
	\eqref{eq:beta_cases_alpha_gtr_half2} 
	now. Set $J = \{ i : \| (i-k-2)_+\|_2 < r\}$ for an $r$ that we choose later.

	
	\paragraph{Bound \eqref{eq:beta_cases_alpha_gtr_half1}.}
	Recall from \eqref{eq:Bn} that
	\begin{align*}
		B_n &= 2t\left( 
		\min\left\{ \|\nu\|_\infty  L_{J,2}, 
		\|\nu\|_2 L_{J,1} \right\}
		\vee  \| b \|_\infty L_{J,1}\right)\\
		&\leq 
		2t\left( \|\nu\|_\infty  L_{J,2} \vee \| b \|_\infty L_{J,1} \right) 
	\end{align*}
	where we get the inequality by taking only the 
	first 
	term of the inner minimum.
	Plugin the bounds for $L$ terms from \eqref{eq:LJ2}, \eqref{eq:LJ1} to write
	\begin{equation}
		B_n = O\left( \| \nu \|_\infty \left(\frac{n}{r^d}\right)^{\alpha - \frac 
		12} + \left(\frac{n}{r^d}\right)^{(\alpha - 1)_+} \gamma_1 \right) t.
	\end{equation}
Plug this back in \eqref{eq:lsq_rates_result} to get
	\begin{equation}
		\label{eq:mse_simp3}
		\frac{1}{n}\snorm{ \hat\beta - \beta^* }_2^2 
		= O_\P \left( \frac{ (r+k+2)^d t^2  \sigma^2}{n} 
		+ 
		\frac{t}{n^\alpha}\left \{  \|\nu\|_\infty 
		\left(\frac{n}{r^d} \right)^{\alpha-\frac 12} 
		\lor \| b \|_\infty \left(\frac{n}{r^d} \right)^{(\alpha-1)_+} \gamma_1 
		\right\} 
		\right) 
	\end{equation}
	For $\alpha \neq 1$, when possible we will choose $r \in [1, N\sqrt{d}]$ 
	such that
	\[
	\frac{r^d t^2  \sigma^2}{n} \asymp \frac{t}{n^\alpha} \sigma_\infty 
	\left(\frac{n}{r^d} 
	\right)^{\alpha-\frac 12}.
	\]
	which is equivalent to
	\[
	r^d \asymp \Big( \frac{\sqrt{n} \sigma_\infty}{t \sigma^2}  
	\Big)^{\frac{2}{2\alpha+1}}.
	\]
	Selecting this $r$ when possible gives the bound in 
	\eqref{eq:beta_cases_alpha_gtr_half1} and the assumption 
	\smash{$\frac{\sqrt{n} \sigma_\infty}{t \sigma^2} \gtrsim 1$} ensures that we 
	are not choosing an impossibly small $r$. When $\alpha = 1$, we can 
	retrace 
	the argument with the additional $\gamma_1$ factor in \eqref{eq:mse_simp3} to 
	get the bound.

	\paragraph{Bound \eqref{eq:beta_cases_alpha_gtr_half2}.}
	When \smash{$\alpha \leq 1$}, set $J=[k+1]^d$ to get the stated bound.
	Now consider \smash{$\alpha > 1$}.
	Simplify \eqref{eq:Bn} by taking only the 
	second 
	term of the minimum, plug the bound for $B_n$ in \eqref{eq:lsq_rates_result} 
	to get
	\begin{equation}
		\label{eq:mse_simp4}
		\frac{1}{n}\snorm{ \hat\beta - \beta^* }_2^2 
		= O_\P \left( \frac{(r+k+2)^d t^2  \sigma^2}{n} 
		+ 
		t n^{-\alpha + \frac{1}{2}} \sigma \left(\frac{n}{r^d} \right)^{\alpha-1} 
		\right) 
	\end{equation}
	When possible we will choose $r \in [1, 
	N\sqrt{d}]$ 
	to balance the two terms above, that is,
	\[
	\frac{r^d t^2  \sigma^2}{n} \asymp t n^{-\alpha + \frac{1}{2}} \sigma 
	\big(\frac{n}{r^d} \big)^{\alpha-1}
	\]
which means,
	\[
	r^d \asymp \Big(\frac{n}{\sigma^2 t^2}  \Big)^{\frac{1}{2\alpha}}.
	\]
	This choice of $r$ gives the desired bound. Our assumption  that 
	\smash{$\frac{n}{\sigma^2 t^2} \gtrsim 1$} makes sure that this 
	choice of $r$ is not impossibly small.
	This completes the proof.
\end{proof}

\begin{proof}[Proof of \autoref{cor:mtf_grids_hetero2}]
This is a direct result of \autoref{cor:mtf_grids_hetero}, simplifying the 
cases.
\end{proof}

\subsection{Error rates assuming that the estimate is bounded}
\label{sec:app_mle_constrained}

Consider the penalized maximum likelihood estimator (MLE)
\begin{equation}
	\label{eq:mle}
	\hat\theta = \argmin_\theta \frac{1}{n} \sum_{i=1}^n (\psi(\theta_i) - y_i 
	\theta_i) + \lambda \|D \theta \|_1.
\end{equation}
The minimum may not be achieved at an interior point of the domain. In that 
case, we set $\hat \theta$ to a limit point of a sequence on which the 
objective converges to the infimum.

If we assume that $\hat \theta$ in \eqref{eq:mle} is constrained in such a way 
that $\psi''(\hat \theta)$ is bounded away from $0$, then the error 
bounding analysis essentially reduces to that in the Gaussian family case.
Consider the constrained estimator 
\begin{equation}
	\label{eq:mle_constrained}
	\hat\theta = \argmin_{\theta \in \Theta(K)^n } \; \sum_{i=1}^n -y_i \theta_i 
	+ 
	\psi(\theta_i) + \lambda \|D\theta\|_1
\end{equation}
where $\Theta(K) = \{ \theta \in \R : \psi''(\theta) \geq \frac{1}{K} \}$ for 
some $K>0$. Assume that $\Theta(K)$ is a convex set for any $K>0$.
This can be verified for Poisson, exponential and logistic families.
Suppose 
\begin{equation}
	\label{eq:theta_best_approx}
	\tilde\theta = \argmin_{\theta \in \Theta(K)^n } \; \sum_{i=1}^n -\E [Y_i] 
	\theta_i + \psi(\theta_i)
\end{equation}
is the best approximation of $\theta^*$ within $\Theta(K)^n$. Also define 
$\tilde\beta = \nabla \psi(\tilde\theta)$.
Then the constrained estimator in \eqref{eq:mle_constrained} satisfies the 
following error bound.
\begin{proposition}
	\label{prop:mle_constrained}
	Let $y_i = \beta^*_i + \epsilon_i$ where $\epsilon_i$ is zero mean
	sub-exponential with parameters $(\nu_i^2, b_i)$ for $i\in [n]$.  
	Let $L_{ J, p}$ be as defined in \eqref{eq:Lell} for 
	\smash{$J \subset [N]^d, p\geq 1$}.
	Abbreviate
	$A_n = \mu \sqrt{\frac{|J|}{n}} \big( \| \nu \|_2 \vee 
	\|b\|_\infty\big) \log n ,$
	$B_n = \left( 
	\min\left\{ \|\nu\|_\infty  L_{J,2}, 
	\|\nu\|_2 L_{J, 1} \right\}
	\vee  \| b \|_\infty L_{J, 1}\right) \log n.
	$
	Then the estimator 
	\eqref{eq:mle_constrained} 
	with $\lambda =  \frac{B_n}{n}$, satisfies
	\begin{equation}
		\KLbar{\tilde\theta}{\hat\theta}
		= \frac{1}{n} O_\P \big( K A_n^2 +
		B_n \| D \tilde \theta\|_1 + K \| \tilde \beta - 
		\beta^*\|_2^2 \big).
	\end{equation}
\end{proposition}

\noindent
The proof is below.
We choose $J\subset [N]^d$ to minimize the bound.
If we set $K = 1/v_{\min}$ where 
$v_{\min} = \min_{i\in[n]} \psi''(\theta^*_i)$, 
then $\tilde\theta = \theta^*, \tilde\beta = \beta^*$
and the above bound reads
\begin{equation}
	\KLbar{\theta^*}{\hat\theta}
	= \frac{1}{n} O_\P \Big(\frac{A_n^2}{ v_{\min}} +
	B_n \| D \theta^*\|_1 \Big).
\end{equation}

\begin{proof}[Proof of \autoref{prop:mle_constrained}]
	Similar to the argument in \autoref{thm:lsq_rates}, from the optimality of 
	\smash{$\hat\theta$}, we have the basic inequality,
	\begin{equation}
		\label{eq:basic_ineq_mle_constrained}
		R(\hat\theta) - R(\tilde\theta) \leq \frac{1}{n} \epsilon^\top (\hat\theta 
		- 
		\tilde\theta) + \lambda \|D\tilde\theta\|_1 - \lambda \|D\hat\theta\|_1
	\end{equation}
	To lower bound the left hand side, we see that
	\begin{align*}
		n R(\hat\theta) - n R(\tilde\theta) 
		&= \one^\top \psi(\hat\theta) - \beta^*\hat\theta - \one^\top 
		\psi(\tilde\theta) + \beta^* \tilde \theta\\
		&= \one^\top  \psi(\hat\theta) - \one^\top \psi(\tilde\theta) - \tilde 
		\beta 
		(\hat\theta - \tilde\theta) + (\tilde\beta -\beta^*)^\top (\hat\theta - 
		\tilde\theta) \\
		&\geq \frac{1}{2K} \| \hat \theta - \tilde \theta \|_2^2 + (\tilde\beta 
		- \beta^*)^\top (\hat\theta - \tilde\theta)\\
		&\geq \frac{1}{2K} \| \hat \theta - \tilde \theta \|_2^2 - K 
		\|\tilde\beta - \beta^*\|_2^2 - \frac{1}{4K}\| \hat\theta - 
		\tilde\theta\|_2^2\\
		&= \frac{1}{4K} \| \hat \theta - \tilde \theta \|_2^2 - K \|\tilde\beta 
		- \beta^*\|_2^2 
	\end{align*} 
	In the above display, the first inequality holds because both 
	\smash{$\hat\theta, 
		\tilde \theta \in \Theta(K)^n$} and \smash{$\Theta(K)^n$} is convex. 
	(For \smash{$i\in[n]$}, write 
	\smash{$\psi(\hat\theta_i) - \psi(\tilde\theta_i) - \tilde \beta_i 
		(\hat\theta_i - 
		\tilde\theta_i) 
		= \psi''(u_i) (\hat\theta_i - \tilde\theta_i)^2$}
	for some \smash{$u_i$} between 
	\smash{$\hat\theta_i$} and 
	\smash{$\tilde\theta_i$}. As \smash{$\Theta(K)$} is convex and \smash{$u_i$} 
	lies 
	between \smash{$\hat\theta_i$} and \smash{$\tilde\theta_i$}, we should have 
	\smash{$u_i \in \Theta(K)$} 
	and so \smash{$\psi''(u_i)$} should be at least \smash{$1/K$}.)
	The second inequality follows from the fact that 
	\smash{$2ab \geq - c a^2 - \frac{1}{c}b^2$}, for any 
	\smash{$a,b,c\in \R$} with \smash{$c>0$}. Applying this to half of 
	the 
	left hand side of  \eqref{eq:basic_ineq_mle_constrained},
	\begin{equation}
		\frac{1}{2} \big( R(\hat\theta) - R(\tilde\theta) \big) +  \frac{1}{8nK} \| 
		\hat \theta - \tilde \theta \|_2^2 - 
		\frac{K}{2n} \| \tilde \beta -\beta^*\|_2^2
		\leq \frac{1}{n} \epsilon^\top (\hat\theta - \tilde\theta) + \lambda 
		\|D\tilde\theta\|_1 - \lambda \|D\hat\theta\|_1
	\end{equation}
	Rearranging,
	\begin{equation}
		\frac{1}{2} \big( R(\hat\theta) - R(\tilde\theta) \big) - \frac{K}{2n} \| 
		\tilde\beta - \beta^* \|_2^2 
		\leq -  \frac{1}{8nK} \| \hat \theta - \tilde \theta \|_2^2 + \frac{1}{n} 
		\epsilon^\top (\hat\theta - \tilde\theta) + \lambda \|D\tilde\theta\|_1 - 
		\lambda 
		\|D\hat\theta\|_1
	\end{equation}
	By \autoref{lem:emp_proc_combined}, for $t\geq 1$ and $J \subset [N]^d$, 
	the following holds with probability at least $1-2(m+|J|)e^{-t}$,
	\begin{align*}
		\frac{1}{2} \big( R(\hat\theta) - R(\tilde\theta) \big) - \frac{K}{2n} \| 
		\tilde\beta - \beta^* \|_2^2 
		&\leq 
		-  \frac{1}{8nK} \| \hat \theta - \tilde \theta \|_2^2 + \frac{A}{n}  \| 
		P_{[\ell]} (\hat\theta-\tilde\theta) \|_2 \\
		&+ \frac{B}{n}  \| D (\hat\theta-\tilde\theta) \|_1
		+ \lambda \| D\tilde \theta\|_1 - \lambda \| D\hat\theta\|_1
	\end{align*}
	where
	$A_n = 2t\mu \sqrt{\frac{|J|}{n}} \big( \| \nu \|_2 \vee \|b\|_\infty\big),$
	$B_n = 2t\left( 
	\min\left\{ \|\nu\|_\infty  L_{J,2}, 
	\|\nu\|_2 L_{J,1} \right\}
	\vee  \| b \|_\infty L_{J,1}\right).
	$
	The sum of the first two terms on the right hand side can be bound by 
	completing squares:
	\begin{align}
		-  \frac{1}{8nK} \| \hat \theta - \tilde \theta \|_2^2 + \frac{A}{n}  \| 
		P_{[\ell]} (\hat\theta-\tilde\theta) \|_2
		&\leq 
		-  \frac{1}{8nK} \| \hat \theta - \tilde \theta \|_2^2 + \frac{A}{n} \| 
		\hat\theta-\tilde\theta \|_2\\
		&\leq 
		\frac{2 K A^2}{n}.
	\end{align}
	Plug this into the bound in the previous display to get
	\begin{equation*}
		\frac{1}{2} \big( R(\hat\theta) - R(\tilde\theta) \big) - \frac{K}{2n} \| 
		\tilde\beta - \beta^* \|_2^2 
		\leq 
		\frac{2KA^2}{n} + \frac{B}{n} \| D (\hat\theta-\tilde\theta) \|_1
		+ \lambda \| D\tilde \theta\|_1 - \lambda \| D\hat\theta\|_1 
	\end{equation*}
	The argument from here is similar to that in the proof of 
	\autoref{thm:lsq_rates}.
\end{proof}

\subsection{Empirical process bound}
Let $D = D^{(k+1)}_{n,d} = U \Sigma V^\top$ be the singular value decomposition of $D$.
For $j\in [N]^d$, let $V_j$ denote $\tilde V_{j_1} \otimes  \dots \otimes \tilde V_{j_d}$ where
$\tilde V_{\ell}$ is the eigenvector of $ \big( D^{(k+1)}_{N,1} )^\top D^{(k+1)}_{N,1}$ 
corresponding to its $\ell$th smallest eigenvalue.
For $J \in [N]^d$, let $V_J$ denote a $|J| \times n$ matrix formed by picking the columns of 
$V$ corresponding to $J$. 
Let $P_J = V_J V_J^\top$ be the projection matrix onto those columns.
\begin{lemma}
\label{lem:emp_proc_combined}
Let $y_i = \beta^*_i + \epsilon_i$ where $\epsilon_i$ is zero mean
sub-exponential with parameters $(\nu_i^2, b_i)$ for $i\in [n]$.  
Let $J \subset [N]^d$ and $L$ be as defined in \eqref{eq:Lell}.
Let $m$ be the number of rows in $D$.
For any $J \subset [N^d]$ containing $[k+1]^d$,
and $t\geq 1$,
with probability at least $1 - 2(m+|J|) e^{-t}$, the following holds uniformly 
for all $\theta \in \R^n:$
\begin{equation}
 |\epsilon^\top \theta| \leq  A \|P_{J} \theta\|_2 + B \| D \theta\|_1
\label{eq:emp_proc_bd_1}
\end{equation}
where 
$A = 2t\mu \sqrt{\frac{|J|}{n}} \big( \| \nu \|_2 \vee \|b\|_\infty\big),$
$B = 2t\left( 
	\min\left\{ \|\nu\|_\infty  L_{J,2}, 
					  \|\nu\|_2 L_{J,1} \right\}
	\vee  \| b \|_\infty L_{J,1}\right).
 $
\end{lemma}
\begin{proof}[Proof of \autoref{lem:emp_proc_combined}]
Decompose
\begin{align}
|\epsilon^\top \theta|
&= 
|\epsilon^\top P_{J} \theta +
\epsilon^\top (I-P_{J}) \theta |\\
&= 
|\epsilon^\top P_{J} \theta +
\epsilon^\top (I-P_{J}) D^\dagger D \theta |\\
&\leq 
\| P_{J} \epsilon \|_2 \| P_{J} \theta \|_2 +
\| (D^\dagger)^\top (I-P_{I}) \epsilon\|_\infty \|D\theta\|_1
\end{align}
where we applied H\"older's inequality on each of the two terms separately.
We give high probability bounds for $\| P_{J} \epsilon \|_2$ and 
$\| (D^\dagger)^\top (I-P_{J}) \epsilon\|_\infty$ separately. A union bound 
will yield the stated  result.

\paragraph{Bounding $\| P_{J} \epsilon \|_2.$}
For $j\in J$, 
$V_j^\top \epsilon$ is SE$(\| \nu \odot V_j \|_2^2, \| b \odot V_j \|_\infty)$.
Therefore, from \eqref{eq:subexp_tail_highprob_bd},
\begin{equation}
|V_j^\top \epsilon| \leq 2t (\| \nu \odot V_j \|_2 \vee \| b \odot V_j 
\|_\infty)
\end{equation}
should hold with probability at least $1-2e^{-t}$ for any $t\geq 1$.
From the incoherence property ($\|V_j \|_\infty \leq \frac{\mu}{\sqrt{n}}$), 
we get
$\| \nu \odot V_j \|_2 \leq \frac{\mu}{\sqrt{n}} \| \nu \|_2$ 
and 
$\| b \odot V_j \|_\infty \leq \frac{\mu}{\sqrt{n}} \|b\|_\infty$.
Therefore,
\begin{equation}
|V_j^\top \epsilon| \leq 2 t \frac{\mu}{\sqrt{n}} \big( \| \nu \|_2 \vee 
\|b\|_\infty \big).
\end{equation}
By union bound over $j\in J$, for any $t\geq 1$,
\begin{equation}
\label{eq:eps_null_proj_bd}
\| P_{J} \epsilon\|_2^2 = \sum_{j\in J} (V_j^\top \epsilon)^2 
\leq 
|J|
\big(2 t \frac{\mu}{\sqrt{n}} \big( \| \nu \|_2 \vee \|b\|_\infty \big)  \big) 
^2
\end{equation}
should hold with probability at least $1-2|J| e^{-t}$.

\paragraph{Bounding $\| (D^\dagger)^\top (I-P_{J}) \epsilon\|_\infty.$}
Rewrite this term as
$$
\| (D^\dagger)^\top (I-P_{J}) \epsilon\|_\infty = \max_{j\in[m]} | g_j^\top 
\epsilon |
$$
where $g_j = (I-P_{J}) D^\dagger e_j$ for $j\in[m]$ 
and where $m$ is the number of rows in $D.$
From \autoref{lem:linear-comb-sub-exp-vec}, one can deduce that
\begin{equation}
\max_{j\in[m]} |g_j^\top \epsilon|  \leq 2t \big( \max_{j\in[m]} \| \nu \odot 
g_j 
\|_2 \vee  \| b \odot g_j \|_\infty \big).
\end{equation}
holds with probability at least $1-2m e^{-t}$ for $t\geq 1$. 
Observe that 
$\| b \odot g_j \|_\infty \leq \|b\|_\infty \|g_j\|_\infty$
and 
\begin{equation}
\| \nu \odot g_j \|_2 \leq \min\left\{ \|\nu\|_\infty \|g_j\|_2, \|\nu\|_2 
\|g_j\|_\infty\right\}.
\end{equation}
Therefore, substituting the bounds on $ \|g_j \|_2, \|g_j\|_\infty$ from 
\autoref{lem:gj_bounds}, we get 
\begin{equation}
\max_{j\in[m]} |g_j^\top \epsilon|  \leq 2t  \left( 
	\min\left\{ \|\nu\|_\infty  L_{J,2}, 
					  \|\nu\|_2 L_{J,1} \right\}
	\vee  \| b \|_\infty L_{J,1}\right).
\end{equation}
with probability at least $1-2me^{-t}.$
\end{proof}

\begin{lemma}
\label{lem:gj_bounds}
Define $g_j = (I-P_{J}) D^\dagger e_j$ for $j\in[m]$ 
and where $m$ is the number of rows in $D.$ Then for all $j\in [m],$
\begin{align}
\|g_j \|_2 &\leq L_{J,2}, \\
\|g_j \|_\infty &\leq L_{J,1}.
\end{align}
\end{lemma}
\begin{proof}[Proof of \autoref{lem:gj_bounds}]
Let $\tilde \Sigma \in \R^{m\times n}$ denote the diagonal matrix such that 
$\tilde \Sigma _{i,i} = \xi_i$ for $i \in J$ and $0$ otherwise. 
Let $\dot \Sigma = \Sigma - \tilde \Sigma,$ which is also diagonal $m\times n$. 
Then 
\[
g_j = V \dot \Sigma^\dagger U^\top e_j.
\]
Therefore, we can write 
\begin{align}
\|g_j\|_2^2
= \| V \dot \Sigma^\dagger U^\top e_j \|_2^2
= \| \dot \Sigma^\dagger U^\top e_j \|_2^2
= \sum_{i \in [N]^d \setminus J} U_{ij}^2 \frac{1}{\xi_i^2} 
\leq \frac{\mu^2}{n} \sum_{i \in [N]^d \setminus J} \frac{1}{\xi_i^2} 
=  L_{J,2}^2.
\end{align}
The sole inequality in the above display follows from the incoherence property 
of $U$. This shows the upper bound on the \smash{$\ell_2$} norms of 
\smash{$g_j, j\in[m]$}.

For the \smash{$\ell_\infty$}-norm bound, we write,
\begin{equation}
\|g_j\|_\infty = \max_{\|z\|_1 = 1} z^\top g_j
= \max_{\|z\|_1 = 1} z^\top V \dot \Sigma^\dagger U^\top e_j
\leq  \max_{\|z\|_1 = 1} \| V^\top z \|_\infty \| \dot \Sigma^\dagger U^\top 
e_j \|_1
\end{equation}
using H\"older's inequality. Because every entry of \smash{$V$} is at most 
\smash{$\mu/\sqrt{n}$}, we have
\begin{equation}
\max_{\|z\|_1 = 1} \| V^\top z \|_\infty \leq \frac{\mu}{\sqrt{n}}.
\end{equation}
From the incoherence property of \smash{$U$},
\begin{equation*}
\| \dot \Sigma^\dagger U^\top e_j \|_1
\leq \frac{\mu}{\sqrt{n}} \sum_{i=\ell+1}^n \frac{1}{\xi_i}.
\end{equation*}
Therefore
\begin{equation*}
\|g_j\|_\infty 
\leq 
\frac{\mu^2}{n} \sum_{i \in [N]^d \setminus J} \frac{1}{\xi_i}  =
L_{J,1}.\hfill\qedhere
\end{equation*}
\end{proof}

\subsection{Eigenvalue bounds}
\label{sec:eigenvalue-bounds}

\begin{lemma}
	\label{lem:kron_tf_evals}
	Let \smash{$\{\xi_i^2: i = (i_1,\dots,i_d) \in [N]^d \}$} be the eigenvalues
	of 
	\smash{$D^\top D$} where \smash{$D = D_{n,d}^{(k+1)}$} and let \smash{$p \ge 
	1$}, \smash{$\alpha = (k+1)/d$}. Then
	\begin{equation*}
			\sum_{i \in [N]^d \setminus [k+1]^d} \frac{1}{\xi_{i}^p} 
			\leq c
			\begin{cases}
					n & \text{if } p \alpha < 1\\
					n \log n & \text{if } p\alpha = 1
				\end{cases}
		\end{equation*}
		for large enough $n$, where $c>0$ is a constant depending only on  
	$k, d$. 
In the case \smash{$p\alpha > 1$}, for any \smash{$r_0 \in [1, \sqrt{d}N]$},
	\begin{equation*}
		\sum_{i \in [N]^d : \| (i-k-2)_+ \|_2 \ge r_0 } \frac{1}{\xi_{i}^p} \leq 
		c n (n/r_0^d)^{p\alpha - 1}.
	\end{equation*}
\end{lemma}

\begin{proof}[Proof of \autoref{lem:kron_tf_evals}]
This is a generalization of Lemma 6 in \cite{sadhanala2021multivariate}, 
which
states the bound for only $p=2$. In their proof,
if we change the power
applied to the singular values in the summation to a general $p\geq 1$
we get (a) the bound in the second display and (b) a bound slightly weaker than 
the first display:
	\begin{equation}
		\label{eq:lem6_latticetf}
	\sum_{i \in [N]^d \setminus [k+2]^d} \frac{1}{\xi_{i}^p} 
	\leq c
	\begin{cases}
		n & p \alpha < 1\\
		n \log n & p\alpha = 1
	\end{cases}
\end{equation}
for large enough $n$, where $c>0$ is a constant depending only on  
$k, d$. Notice that the summation excludes indices in $[k+2]^d$ whereas the 
statement in \autoref{lem:kron_tf_evals} requires only those in $[k+1]^d$ to be 
excluded. We claim that the additional terms from indices $[k+2]^d \setminus 
[k+1]^d$ do not change the rates in the bound.
Thanks to the Kronecker-sum structure of $D^\top D$, we can write $\xi_i^2 = 
\sum_{j=1}^d \rho_{i_j}$ where $\rho_1, \dots, \rho_{N}$ 
are the eigenvalues of $ \big(D^{(k+1)}_{N,1}\big)^\top D^{(k+1)}_{N,1}$.
Note that for $i\in [N]^d \setminus [k+1]^d$, we can write $\xi_i^2 \geq 
\rho_{k+2}$. Therefore,
\begin{align*}
	\sum_{i \in [k+2]^d \setminus [k+1]^d} \frac{1}{\xi_i^p}
\leq \sum_{i \in [k+2]^d \setminus [k+1]^d} 
	\frac{1}{\rho_{k+2}^{\nicefrac{p}{2}}}
\leq \sum_{i \in [k+2]^d \setminus [k+1]^d} 	N^{p(k+1)}
\leq ((k+2)^d - (k+1)^d) c n^{p\alpha}
\end{align*}
where we used \autoref{lem:eig_lowerbd_kron_DDt} for the second inequality.
In the case $p\alpha \leq 1$, this and \eqref{eq:lem6_latticetf} are sufficient 
to prove the lemma.
\end{proof}

\begin{lemma}
	\label{lem:eig_lowerbd_kron_DDt}
	For $k\ge 1, N > 2k+2,$ the smallest eigenvalue of $ D^{(k)}_{N,1} 
	\big(D^{(k)}_{N,1} \big)^\top$ 
is at least $c/N^{2k}$ for some constant $c>0$ depending only on $k$.
\end{lemma}
\begin{proof}

	For the purpose of this lemma, let $\lambda_i(A)$ denote the $i$th smallest 
	eigenvalue of $A$.
	
	\textbf{Case: $k$ is odd.} By Cauchy interlacing argument in Lemma 7 of 
	\cite{sadhanala2021multivariate}, we have $\lambda_1(D^{(k)}_{N,1} 
	\big(D^{(k)}_{N,1} \big)^\top) \ge \lambda_1(GG^\top)$ 
	where $G$ is the 
	graph trend filtering operator of order $k$ on a chain of length $N$.
	Recall that $G = D^{(1)}_{N,1} L^{(k-1)/2}$  where $L$ is the graph 
	Laplacian of a chain of length $N$.
	Note that, for odd $k$, $G^\top G = L^{k}$.
	The set of 
	\textit{nonzero} eigenvalues of $GG^\top$ and $G^\top G$ should be the same. 
	We know that
	$\lambda_1(L) = 0, \lambda_2(L) > 0$ and so 
	$\lambda_1(G^\top G) = 0, \lambda_2(G^\top G) > 0$.
	$GG^\top$ has full rank. 
	Therefore,
	$$\lambda_1(GG^\top) = \lambda_2(G^\top G) = \lambda_2(L^{k}) = 
	\big(\lambda_2(L) \big)^{k}.$$
	Plugging in $\lambda_2(L) = 4 \sin^2 \nicefrac{\pi}{2N}$ and using the 
	inequality $\sin x \geq \nicefrac{x}{2}$ for $x\in [0, \nicefrac{\pi}{2}]$, 
	we have
	$\lambda_1(GG^\top) \geq c/N^{2k}$. 
	As $\lambda_1(D^{(k)}_{N,1} 
	\big(D^{(k)}_{N,1} \big)^\top) \ge \lambda_1(GG^\top)$, we get 
	$\lambda_1(D^{(k)}_{N,1} \big(D^{(k)}_{N,1} \big)^\top)	\geq c / N^{2k}$.

  \textbf{Case: $k$ is even.}
  Apply \autoref{lem:tf_second_smallest_eig_lowerbd_oddk} to get the bound in 
  this case.

\end{proof}

\begin{lemma}
	\label{lem:tf_second_smallest_eig_lowerbd_oddk}
	let $\lambda_i(A)$ denote the $i$th smallest 	eigenvalue of $A$.
	For $k\ge 1$, and $N > 2k+2$,
	\[
	\lambda_{2k+1} \big( (D^{(2k)}_{N,1})^\top  D^{(2k)}_{N,1} \big) \ge 
	\big( 4 \sin^2 \frac{\pi}{2N-2} \big)^{2k}.
	\]
\end{lemma}
\begin{proof}
	Let $L_m$ denote the Laplacian of cycle graph with $m$ vertices.
It's smallest nonzero eigenvalue is $ 4 \sin^2 \nicefrac{\pi}{m}$. Its 
eigenvectors are given $ 
	(v_\ell)_j = e^{2\pi i \ell j / m}.$
	
	Let $u \in \R^N$ be the eigenvector of $(D^{(2k)}_{N,1})^\top  
	D^{(2k)}_{N,1}$ 
	corresponding to its $(2k+1)$th eigenvalue. By 
	\autoref{lem:tilde_u_construction}, there 
	exists a $v \in \R^{2N-2}$ 
	satisfying the following properties:
	\begin{align}
		\| L^{k} v \|_2^2 &\leq 2 \| (D^{(2k)}_{N,1} u \|_2^2,\\
		\langle v, \one \rangle &= 0,\\
		\| v \|_2^2 &\geq 2.
	\end{align}
	With such a $v$,
	\begin{align*}
		\lambda_{2k+1} \big( \big(D^{(2k)}_N \big)^\top D^{(2k)}_N \big) 
		= 
		\| D^{(2k)}_N u \|_2^2 
		\ge 
		\frac{1}{2} \| L^{k} v \|_2^2 
		\ge 
		\frac{1}{2} \lambda_2( L^{2k}) \| v \|_2^2 
		\ge 
		\lambda_2^{2k}(L).
	\end{align*}
	The equality holds by definition of $u$. The three inequalities follow in 
	order from the three properties satisfied by $v$ above.
	This is sufficient to complete the proof because we know that $\lambda_2(L) = 
	4 \sin^2 \frac{\pi}{2N-2}$.
\end{proof}

\begin{lemma}
	\label{lem:tilde_u_construction}
	Let $u \in \R^N$ be the eigenvector of $(D^{(2k)}_{N,1})^\top  
	D^{(2k)}_{N,1}$ 
	corresponding to its $(2k+1)$th eigenvalue. There exists a $v \in 	
	\R^{2N-2}$ 	satisfying the following properties:
	\begin{align}
	\| L^{k} v \|_2^2 &\leq 2 \| (D^{(2k)}_{N,1} u \|_2^2,\\
	\langle v, \one \rangle &= 0,\\
	\| v \|_2^2 &\geq 2.
\end{align}
\end{lemma}
\begin{proof}
	Define $\cU = \{ u \in \R^N : u_1 = u_N = 0\}$.

\noindent	
	\paragraph{$\Delta$ and $\Delta^{-1}$:}
	Define the following truncated discrete difference operator,
	\[
	\Delta u = (0, (D^{(2)}_{N,1} u)_1, D^{(2)}_{N,1} u)_2, \ldots, D^{(2)}_{N,1} 
	u)_{N-2}, 0)
	\]
	for $u \in \cU$ so that $\Delta : \cU \to \cU$.
	We can write
\begin{align}
	\label{eq:Delta}
	\Delta = 
	\left( \begin{array}{cccccc}
		0 & 0 & 0 & \ldots & 0 & 0 \\
		-1 & 2 & -1 & 0 &\ldots &0 \\
		0 & -1 & 2 & -1 & \ldots & 0\\
		& & \ldots & & & \\
		0 & \ldots & 0 & -1 & 2 & -1 \\
		0 & 0 & 0 & 0 & 0 & 0\\
	\end{array} \right)
	\end{align}
	Then we can construct the inverse as the following truncated discrete 
	integral using the following:
	Let $u \in \cU$, and define the cumulative sum operator,
	\[
	(I u)_i := \sum_{j=1}^{i-1} j u_{i-j}, \quad \text{ and }
	a := \frac{1}{N-1} (Iu)_N.
	\]
	Define
	\[
	z_i := (i-1) a - (I u)_i, \quad i=1,\ldots, N,
	\]
and note that $z_1 = z_N = 0.$
	Then we have that $\Delta z = u$ for $u \in \cU$.
	To see this let $i = 2, \ldots, N-1$,
	\begin{align}
		- (\Delta z)_i &= - (2i a - (i-1) a - (i+1) a) + 2 (Iu)_i - (Iu)_{i-1} - 
		(Iu)_{i+1}\\
		&= 2 \sum_{j=1}^{i - 1} j u_{i - j} - \sum_{j=1}^{i-2} j 
		u_{i-1-j} - \sum_{j=1}^{i} j u_{i-j+1}\\
		&= 2 \sum_{j=1}^{i - 1} j u_{i - j} - \sum_{j=2}^{i-1} (j-1) u_{i-j} - 
		\sum_{j=0}^{i-1} (j+1) u_{i-j} = 2 u_{i-1} - u_i - 2 u_{i-1} = -u_i.
	\end{align}
	Also, $(\Delta z)_{1} = (\Delta z)_{N}=0 = u_1 = u_N.$

	\paragraph{Constructing $v$:}
	Construct $\tilde u \in \R^N$ such that
	\[
	 \tilde{u}_i = u_i - u_1 - \frac{u_N - u_1}{N-1} \dot (i-1), \quad i = 
	 1,\dots, N
	\]
	Define $ w \in \R^N$ such that 
	$w_i = (\Delta^k \tilde u)_i$ 
	for $i=1,\dots, k$ and $i=N, N-1, N-k+1$; and $w_i = 0$ for other $i \in [N]$.
	Define $p = \Delta^{-k} w$ and note that $w, p \in \cU$.
		Let $\mathrm{ext}(x)$ denote the periodic extension of $x \in \R^N$,
	defined by $y \in \R^{2N-2}$ where $ y_{1:N} = x, y_{N+i} = -x_{N-i}$ for 
	$i=1,\dots,N-2$.
	Set
	$$v = \mathrm{ext}(\tilde u - p).$$ 
	\paragraph{Verifying the three properties:}
	As $\tilde u - p \in \cU$, by 
	\autoref{lem:L_Delta_on_U}, 
	 \[
	 (L^k v)_{1:N}  = \Delta^k (\tilde u - p) = [0_{k\times 1}; D^{(2k)}_{N,1} u; 
	 0_{k\times 1}].
	 \]
	 By construction of $v$ via $\mathrm{ext}$, 
	 $(L^k v)_{2:N} = - (L^k v)_{2N-2:N}$. So
	 $\| (L^k v)_{N+1:2N-2} \|_2^2 = \| (L^k v)_{1:N} \|_2^2$.
	 Therefore $v$ satisfies the first desired property in the statement of the 
	 lemma:
	 \[
		\| L^k v \|_2^2 = 2 \| D^{(2k)}_{N,1} u \|_2^2.
	 \]
	 As $v = \mathrm{ext}(\tilde u - p)$ and $\tilde u - p \in \cU$, we get 
	 $\langle v, \one \rangle = 0$ from the definition of $\mathrm{ext}$.
	 Again due to the definition of $\mathrm{ext}$, $\| v \|_2^2 = 2 \| \tilde u 
	 - p \|_2^2$. 
	 Write $\tilde u - p = u + 
	 (\tilde u - u - p)$ and note that $u \perp \cN(D^{(2k)}_{N,1})$, 
	 $\tilde u - u$ is linear and hence in $\cN(D^{(2k)}_{N,1})$ and further 
	 $p \in \cN(D^{(2k)}_{N,1})$ by construction. (Note that if strip out the top 
	 and bottom $k$ rows from $\Delta^k$, we get $D^{(2k)}_{N,1}$. So 
	 $D^{(2k)}_{N,1} p = (\Delta^k p)_{k+1:N-k} = w_{k+1:N-k} = 0$.)
	 Therefore we get the third desired property for $v$:
	 \[
	 \| v \|_2^2 \geq 2 \| u \|_2^2 + 2 \| 	 \tilde u - u - p \|_2^2 \geq 2.
	 \]	 
	 Therefore $v$ satisfies all the three properties stated in the lemma.
\end{proof}

\begin{lemma}
	\label{lem:L_Delta_on_U}
	Let $\cU = \{ u \in \R^N: u_1 = u_N = 0 \}$.
	Let $\mathrm{ext}(u)$ denote the periodic extension of $u \in \R^N$,
	defined by $v \in \R^{2N-2}$ where $ v_{1:N} = u, v_{N+i} = -u_{N-i}$ for 
	$i=1,\dots,N-2$.
	Let $L, \Delta$ be as defined in \autoref{lem:tilde_u_construction} and 
	\eqref{eq:Delta} respectively.
	Then $(L^k \mathrm{ext}(u))_{1:N} = \Delta^k u$ for $u \in \cU$.
\end{lemma}
\begin{proof}
	Let $v := \mathrm{ext}(u)$ and let $\cS = \{ \mathrm{ext}(u) : u \in \cU\}$.
	We need to show that $(L^k v)_{1:N} = \Delta^k u$ 
	for $k \ge 1$.
	First notice that $(\Delta u)_i = 2u_i - u_{i-1} 
	- u_{i+1}$, $i=2,\ldots, N-1$.
	Furthermore, $(\Delta u)_1 = (\Delta u)_{N} = 0$ because the first and last 
	rows of $\Delta$ are zeros
	and $(Lv)_1 = (Lv)_N = 0$ because $v \in \cS$.
	(As $v \in \cS$, $v$ is anti-symmetric around index $1$, that is:
		$v_1 = 0$, $v_{i} = -v_{2N-i}$ for $i=2,3,\dots, N$ and so $(Lv)_1 = 0$.
		Similarly 
		$v_{N-i} = -v_{N+i}$ for $i=0, 1,\dots, N-2$ and so $(Lv)_{N} = 0$.
	)
	So we have shown it for $k=1$.
	Suppose the inductive hypothesis $\Delta^{k-1}  u = (L^{k-1} v)_{1:N}$.
	We have for $i=2,\ldots, N-1$,
	\[
	(\Delta^{k}  u)_i 
	= 2(\Delta^{k-1} u)_i - (\Delta^{k-1} u)_{i-1} - (\Delta^{k-1} u)_{i+1} 
	= 2 (L^{k-1} v)_i - (L^{k-1} v)_{i-1} - (L^{k-1} v)_{i+1}  
	= (L^k v)_i.
	\]
	Furthermore, $(\Delta^k u)_1 = (\Delta^k u)_N = 0$ by construction
	and $(L^k v)_1 = (L^k v)_N = 0$ because of anti-symmetry of $v$ around 
	indices $1$ and $N$.
	Thus, $(L^k v)_{1:N} = \Delta^k u$.
\end{proof}
\section{Proofs for lower bounds}

\subsection{Proof of \autoref{prop:lowerbd_homosked}}
\label{sec:lowerbd_homosked_proof}
Denote the $\ell_p$ balls
\begin{equation}
	\label{eq:lp_ball}
	B_p(r; \R^n) = \{ x \in \R^n : \| x \|_p \leq r\}
\end{equation}
for $p \ge 1, r \ge 0, n \ge 1$. We simply refer to this $B_p(r)$ when the
dimension $n$ is clear from the context.
Consider the set 
\begin{equation}
\label{eq:B_r_m}
B(r,m) = \big\{ \beta \in \R^n : \| \beta \|_\infty \leq r, \|\beta\|_0 \leq m 
\big\}
\end{equation}
which consists of signals with at most $m$ non-zero components
and with all entries at most $r$ in magnitude.

For $\beta\in \R$ and $\sigma>0,$ 
	let $\mathrm{Lap}(\beta,\sigma)$ denote the Laplace distribution 
centered at $\beta$ with scale $\sigma.$
	For $\beta \in \R^n,$ let $\mathrm{Lap}(\beta,\sigma)$ denote the 
product distribution of $\mathrm{Lap}(\beta_1,\sigma), \dots,  
\mathrm{Lap}(\beta_n,\sigma).$
	
\begin{proof}[Proof of \autoref{prop:lowerbd_homosked}] 
The null space of $D$ has a dimension of $\kappa$. Using Fano's lemma, similar 
to the way it is applied in Example 15.8 in \citet{wainwright2019high}, we can 
show that 
\begin{equation}
\label{eq:lb_null_space}
n \cdot R_M \big( \ktvset_{n,d}^k (C_n)  \big) \geq \frac{\kappa \sigma^2}{128}
\end{equation}
The main difference is in upper bounding for KL divergence, but from 
\autoref{lem:kl_lap_upper_bd} we can show that
\begin{equation}
\label{eq:kl_bd_lap}
\kl \left(\lap(a, \sigma), \lap(b, \sigma) \right) 
\leq \| a - b \|_2^2 / 2\sigma^2
\end{equation}
for $a,b \in \R^n.$
This is sufficient to apply the argument in Example 15.8 in 
\cite{wainwright2019high}.

Now we show the second lower bound.
Note that 
$$
	B_1(C_n/c_k) \subseteq \ktvset_{n,d}^k (C_n)
$$
where $c_k$ is the maximum $\ell_1$ norm of columns of $D.$ 
$c_k$ depends only on $k,d.$ 
Denote $r_1 = C_n / c_k.$
For $q\in Q := \{1\} \cup \{2m : 2m \leq n/3\} $, set $r=C_n/(q c_k)$ so that 
$B(r,q)$  is contained in 
$B_1(C_n/c_k)$.
From \autoref{lem:lowerbd_infball}, 
\begin{equation}
n \cdot R_M( B(r,q) ) \geq \frac{1}{12} q a^2
\end{equation}
where $a = r \wedge \sigma g^{-1}(\tau/6)$ where $\tau = \log (en/8q).$
Therefore, from the containment $B(r, q) \subset \ktvset_{n,d}^k (C_n),$
\begin{align}
n \cdot R_M \big( \ktvset_{n,d}^k (C_n)  \big) 
	&\geq \frac{1}{12} \sup_{q \in Q} \;q \min\left\{ r^2, 
\frac{\sigma^2}{3} \log \frac{en}{8q} \vee  \frac{\sigma^2}{36} \log^2 
\frac{en}{8q}
						\right\}\\
	&= \frac{1}{12} \sup_{q \in Q} \;q \min\left\{ \frac{r_1^2}{q^2}, 
\frac{\sigma^2}{3} \log \frac{en}{8q} \vee  \frac{\sigma^2}{36} \log^2 
\frac{en}{8q}
						\right\}
\end{align}
Choose $q\in Q$ that maximizes this bound. 
Set $q$ to the closest number in $Q$ to
\begin{equation}
q^* = \frac{r_1}{\sigma} \left( 
	\sqrt{3} \log^{-1/2} \frac{\sigma n}{\sqrt{3} r_1}  
	\vee 
	6 \log^{-1} \frac{\sigma n}{6 r_1}
	\right)
\end{equation}
where $r_1 = C_n / c_k.$ 
This gives a lower bound of
\begin{equation}
\label{eq:lb_part1}
c_0 \sigma r_1  \left( \sqrt{\log \frac{c_1\sigma n}{r_1}} \vee \log \frac{c_2 
\sigma n}{r_1} \right)
\end{equation}
provided $q^*$ is within the range $[1, n/3]$. 
Two alternate bounds can be obtained by plugging in $q=1$ and $q=2\lfloor n/6 
\rfloor $. With $q=1$, the bound is
$c \min\left\{ r_1^2, \sigma^2 \left(\log \frac{en}{8} \vee  \log^2 
\frac{en}{8} \right)
						\right\}$
and with $q=2\lfloor n/6 \rfloor$, the bound is 
$ c\min\left\{ \frac{r_1^2}{n}, \sigma^2
						\right\}.
$

Finally, we derive the third term in the lower bound by embedding a H\"older ball.
We follow the proof of
Theorem 2.5 
in \cite{tsybakov2009introduction}. 
For $k\ge 0$ and $L>0$, let $H(k+1, L; [0,1]^d)$ denote the H\"older class of 
functions on $[0,1]^d$
whose $k$th order partial derivatives 
$\partial^k f / \partial x_1^{\alpha_1} \dots \partial x_d^{\alpha_d}$
with $\alpha_1 + \dots + \alpha_d = k$ are $L$-Lipschitz.
Define the discrete H\"older set using evaluations of H\"older functions on the 
grid:
\begin{equation}
	\label{eq:holder_set_discrete}
\cH_{n,d}^k(L) = \{ 
	\theta \in \R^n : \theta_i = f(i_1/N, \dots, i_d/n) , 
	f\in H(k+1, L; [0,1]^d)\}.
\end{equation}
\cite{SadhanalaWang2017} shows that
$$
\cH_{n,d}^k (c C_n n^{\alpha -1}) \subset \ktvset_{n,d}^k (C_n)
$$
for a constant $c$ depending only $k.$
Therefore, the minimax risk over \smash{$\ktvset_{n,d}^k (C_n)$} is at least the 
minimax risk over \smash{$\cH_{n,d}^k(C_n)$}. 
\autoref{lem:lower_bd_holder} gives a lower bound on this risk:
\begin{equation}
R_M( \ktvset_{n,d}^k (C_n)) = \Omega\bigg(  \left( \frac{\sigma^2}{n} 
\right)^{\frac{2\alpha}{2\alpha+1}} 
							(C_n n^{\alpha-1}) 
^{\frac{2}{2\alpha+1}}
				\bigg).
\end{equation}
This equation, together with \eqref{eq:lb_null_space}, \eqref{eq:lb_part1} 
gives the desired lower bound.
\end{proof}

\begin{lemma}
\label{lem:lower_bd_holder}
On the $d$-dimensional grid, consider the observation model $y_i = f(x_i) + 
\epsilon_i$ for $i\in [N]^d$ where $f \in H(k+1, L; [0,1]^d)$ and $\epsilon_i$ 
are i.i.d. $\mathrm{Lap}(0,\sigma)$. Then 
\begin{equation}
\label{eq:holder_lower_bd_gauss_L2}
\inf_{\hat f} \sup_{f_0 \in H(k+1, L; [0,1]^d)} \E \| \hat f - f_0 \|_2^2 = 
	\Omega \left(  \left( \frac{\sigma^2}{n} 
\right)^{\frac{2\alpha}{2\alpha+1}} 
							L^{\frac{2}{2\alpha+1}}
				\right).
\end{equation}
Suppose there exists an $h_0 \ge 0$ such that, for any $h\ge h_0,$
any ball of radius $ch/2$ in $[0, 1]^d$ contains at least $c_1 n (ch/2)^d$ grid 
points, where
$c=\sqrt{\log_{2e}{2}}$ and $c_1 > 0$ is a constant  may depend on $d.$ Then
the following lower bound in terms of the empirical norm holds:
\begin{equation}
	\label{eq:holder_lower_bd_gauss_emp}
	\inf_{\hat f} \sup_{f_0 \in H(k+1, L; [0,1]^d)} \E \| \hat f - f_0 
\|_n^2 = 
	\Omega \left(  \left( \frac{\sigma^2}{n} 
\right)^{\frac{2\alpha}{2\alpha+1}} 
	L^{\frac{2}{2\alpha+1}}
	\right).
\end{equation}
\end{lemma}
\begin{proof}[Proof of \autoref{lem:lower_bd_holder}]
We adapt the proof of the univariate case in Section 2.6 of 
\cite{tsybakov2009introduction}.
Partition $[0,1]^d$ into $r = \lceil c_0 n^{1/(2\alpha+1)} \rceil$ hypercubes 
of equal size, where $c_0$ is to be determined later.
The side length of each hypercube $h=(1/r)^{1/d}.$
Let $z_i, i\in [r]$ be the centers of these hypercubes. 
Define the bump function
\begin{equation*}
\varphi(x) = L  h^{k+1} K \bigg( \frac{\|x\|_2}{h}\bigg)
\text{ for } x \in [0,1]^d
\quad
\text{ where }
K(u) = a e^{\frac{-1}{1-4u^2}} 1 \big\{ |u| < \frac 12 \big\}
\end{equation*}
for a constant $a$ such that $\varphi \in H(k+1, 1).$ 
Note that $\varphi(x) = 0$ if $\|x\|_2 \ge h/2.$
Define the bump functions  $\varphi_i(x) = \varphi(x - z_i),$ centered around 
$z_i$ for $i\in[r].$
These functions have disjoint support and so, they are orthogonal to each other 
with respect to the $L_2$ inner product and also the empirical inner product.
Note that 
\begin{equation}
\label{eq:bump_l2_norm}
\| \varphi \|_2^2 = L^2 h^{2k+2+d} \| K \|_2^2
\end{equation}

By Varshamov-Gilbert lemma \citep[see Lemma 2.9 in][]{tsybakov2009introduction}, 
we can get  $
\omega^{(0)}, \dots, \omega^{(M)}  \in \{0,1\}^r$
such that
$\omega^{(0)} = \mathbb{0}_r$,
$M \geq 2^{r/8}$
and
for $i \neq  j \in \{ 0, \dots, M \},$
$d_H( \omega^{(i)}, \omega^{(j)} ) \geq r/8$
where $d_H$ calculates the Hamming distance between two binary vectors of same 
size.
Let 
	$$ f_i = \sum_{j=1}^r \omega^{(i)}_j \varphi_j $$
for $i=0,\dots,M.$
For $i \neq j$, 
\begin{align}
\| f_i - f_j \|_2^2
&= \sum_{\ell=1}^r 1\{ \omega^{(i)}_\ell \neq \omega^{(j)}_\ell\} \| 
\varphi_\ell \|_2^2 \\
&= d_H( \omega^{(i)}, \omega^{(j)})  \| \varphi\|_2^2 \\
\label{eq:fi_fj_lower_bd}
&\ge \frac{r}{8} \cdot L^2 h^{2k+2+d} \| K \|_2^2
\end{align}
The last line is true because  (a) $d_H( \omega^{(i)}, \omega^{(j)}) \ge r/8$ 
by construction of the 
bump functions and
(b)  \eqref{eq:bump_l2_norm}.

distribution $\Pi_{i=1}^n \lap(\mu_i, \sigma).$
Let $x_1, \dots, x_n \in [0,1]^d$ denote the grid locations.
For $j \in \{ 0, \dots, M\},$ let $P_j$ denote the joint distribution of $y_1, 
\dots, y_n$ given by $y_i = f_j(x_i) + \epsilon_i$ with $\epsilon_i$ i.i.d. 
$\lap(0,\sigma).$ 
Then
\begin{align}
\kl (P_j, P_0)
&= \sum_{i = 1}^n \kl \big(\lap(f_j(x_i), \sigma), \lap(0, \sigma)\big)\\
&\leq 
\sum_{i = 1}^n \frac{1}{2\sigma^2} f_j^2(x_i)\\
&\leq 
\sum_{i = 1}^n \frac{1}{2\sigma^2} L^2 a^2 h^{2k+2}\\
&=\frac{n}{2\sigma^2} L^2 a^2 h^{2k+2}\\
&= \frac{n}{2\sigma^2} L^2 a^2 r^{-2\alpha}\\
\label{eq:kl_Pj_P0_bd1}
&= \frac{1}{2\sigma^2} L^2 a^2 r c_0^{-(2\alpha+1)}
\end{align}
The second line is from Lemma~\ref{lem:kl_lap_upper_bd} and the third line is 
from the fact that $f_j$ is a summation of  bump functions with (a) disjoint 
supports and (b) a maximum value of 
$aL h^{k+1}.$
The last two lines follow from the relations $h = r^{-1/d}, r = \lceil c_0 
n^{1/(2\alpha+1)}\rceil.$

Now we choose a $c_0$ (recall  $r = \lceil c_0 n^{1/(2\alpha+1)}\rceil$) such 
that
\begin{equation}
\frac 1M \sum_{j=1}^r \kl(P_j, P_0) \leq \frac{1}{8 \log 4} \log M.
\end{equation}
From \eqref{eq:kl_Pj_P0_bd1} and the fact that $M\ge 2^{r/8}$, it is sufficient 
to choose $c_0$ such that
$\frac{1}{2\sigma^2} L^2 a^2 r c_0^{-(2\alpha+1)} \leq \frac{r}{64}.$ So we 
choose
\begin{equation}
c_0 = \big( 32 a^2 L^2 \sigma^{-2} \big)^{1/(2\alpha+1)}.
\end{equation}
With this choice of $c_0$, and the lower bound in \eqref{eq:fi_fj_lower_bd} we 
can apply Theorem 
2.5 in \cite{tsybakov2009introduction} to get the bound in 
\eqref{eq:holder_lower_bd_gauss_L2}.

\paragraph{Lower bound in empirical norm.}
We follow the same approach to show the lower bound in 
\eqref{eq:holder_lower_bd_gauss_emp} in 
terms of the empirical norm.
It is sufficient to show a bound analogous to \eqref{eq:fi_fj_lower_bd} in 
terms of the empirical 
norm.
Let $B(z, s)$ denote an $\ell_2$ ball of radius $s$ centered at $z.$ 

For any $\ell \in [r]$, by hypothesis, there are at least $c_1 n (ch/2)^d$ grid 
points in $B(z_\ell, 
ch/2)$. For $x\in B(z_\ell, ch/2)$, $\varphi(x) = L h^{k+1} K(\|x-z_\ell\|_2/h) 
\ge L h^{k+1} K(c/2).$
For our choice $c = \sqrt{\log_{2e}{2}}$, $K(c/2) \ge K(0)/2e = a/2e.$ 
Therefore, for all $x 
\in B(z_\ell, ch/2)$, $\varphi_\ell(x) \ge a/2e \cdot L h^{k+1}$. Consequently,
\begin{equation}
	\| \varphi_\ell\|_n^2 \ge \frac{1}{n} \cdot c_1 n (ch/2)^d \cdot (a/2e 
L h^{k+1})^2 = c_2 L^2 
	h^{2k+2+d}.
\end{equation}
Recall that 
\[
\| \varphi_\ell \|_2^2 = L^2 h^{2k+2+d} \| K \|_2^2
\]
and therefore
\begin{equation}
	\label{eq:bump_emp_norm_bd}
		\| \varphi_\ell\|_n^2 \geq c_3 	\| \varphi_\ell\|_2^2
\end{equation}
for a constant $c_3$ that may depend on $d$.
\begin{align}
\| f_i - f_j \|_n^2 
&= \sum_{\ell = 1}^r 1 \big\{ \omega_\ell^{(i)} \neq \omega_\ell^{(j)}  \big\} 
\| \varphi_\ell\|_n^2\\
&\ge \sum_{\ell = 1}^r 1 \big\{ \omega_\ell^{(i)} \neq \omega_\ell^{(j)}  
\big\} c_3 \| \varphi_\ell\|_2^2\\
&= \sum_{\ell = 1}^r 1 \big\{ \omega_\ell^{(i)} \neq \omega_\ell^{(j)}  \big\} 
c_3 \| \varphi\|_2^2\\
&= d_H( \omega^{(i)}, \omega^{(j)})  c_3 \| \varphi\|_2^2\\
&= c_3  \frac{r}{8} \cdot L^2 h^{2k+2+d} \| K \|_2^2
\end{align}
Second line follows from \eqref{eq:bump_emp_norm_bd}.
Now \eqref{eq:holder_lower_bd_gauss_emp} can be derived similar to
\eqref{eq:holder_lower_bd_gauss_L2}, by applying Theorem 2.5 in 
\cite{tsybakov2009introduction}.
\end{proof}

\begin{lemma}
\label{lem:kl_lap_upper_bd}
For $\mu_1, \mu_2 \in \R$, and $\sigma > 0,$
$$
\kl( \lap(\mu_1, \sigma), \lap(\mu_2, \sigma)) =  e^{-\delta} + \delta - 1 \leq 
\frac 12 \delta^2
$$
where $\delta = |\mu_1 - \mu_2|/\sigma.$
Let $g(x) = e^{-x} + x - 1$ for $x\ge 0.$ Then for $y\geq 0$, 
$$
g^{-1}(y) \geq \max\{ \sqrt{2y}, y\}.
$$

\end{lemma}
\begin{proof}[Proof of Lemma~\ref{lem:kl_lap_upper_bd}]
From a direction integration, as shown in Appendix A in 
\cite{meyer2021alternative}, 
$$
\kl( \lap(\mu_1, \sigma), \lap(\mu_2, \sigma)) = e^{-\delta} + \delta - 1 = 
g(\delta)
$$
where $\delta = |\mu_1-\mu_2|/\sigma$.
We can verify with elementary calculus that, for all $y\ge 0$, 
$$g(y) < y \text{ and } g(y) \leq \frac{y^2}{2}.$$ 
Therefore for all $y\ge 0$, 
$$
g(y) < y \text{ and } g(\sqrt{2y}) \leq y.
$$
$g$ is a strictly increasing function on $[0,\infty).$
Therefore,
\[
y < g^{-1}(y), \sqrt{2y} \leq g^{-1}(y) \text{ for all } y\ge 0. \qedhere
\]
\end{proof}

\begin{lemma}
\label{lem:lowerbd_infball}
Suppose $n\ge 6.$ Suppose  $q=1$ or  $q$ is even with $q \leq n/3$. Then for 
$r>0$, the minimax risk of $B(r, q)$ defined in \eqref{eq:B_r_m} satisfies
\begin{equation}
n \cdot R_M \big( B(r, q) \big) \geq
 \frac{1}{12} q \min\left\{ r^2, \frac{\sigma^2}{3} \log \frac{en}{8q} \vee  
\frac{\sigma^2}{36} \log^2 \frac{en}{8q}
						\right\}
\end{equation}
\end{lemma}

\begin{proof}[Proof of \autoref{lem:lowerbd_infball}]
We will show a slightly stronger bound:
\begin{equation}
		n \cdot R_M \big( B(r, q) \big) \geq
		\frac{1}{12} q \left( r\wedge \sigma g^{-1} \left( \frac{1}{6} \log 
		\frac{en}{8q}  \right) 
		\right)^2
	\end{equation}
	where \smash{$g(x) = e^{-x} + x - 1$} for \smash{$x\ge 0.$}
From this and \autoref{lem:kl_lap_upper_bd}, we get the  bound in 
\autoref{lem:lowerbd_infball}.

	The proof is adapted from that of 
	Theorem 5 in \cite{birge2001gaussian} for Gaussian error model.
	We use Fano's lemma from information theory.
	
	Abbreviate \smash{$\tau = \log \frac{en}{8q}.$}
	\begin{itemize}
		\item Let $$\cM_{q} = \{ S \subseteq [n] : |S| = q\}$$ 
		Here $|S|$ denotes the cardinality of a set $S.$
		Consider signals $\beta_S \in \R^n$
		$$ (\beta_S)_i = \one \{i\in S\} a $$
		where $a = r \wedge \sigma g^{-1}(\tau /6).$
		As $q\leq n/3$, $\tau = \log \frac{en}{8q}$ should be positive. $g$ is 
		strictly 
		increasing over $x\geq 0,$ $\lim_{x\rightarrow \infty} g(x) = \infty$ and 
		so 
		$g^{-1}(\tau /6)$ is well-defined.
		
		We will pick sufficiently separated elements from $\cM_q$ to construct 
		signals 
		for  Fano's lemma.
		
		\item Suppose $q$ is even with $q\leq n/3$. From Lemma 4 
		\cite{birge2001gaussian} we can find a subset $\cS$ of $\cM_q$ such that  
		\begin{itemize}
			\item for any distinct $S,S' \in \cS$, $|S\cap S'| < q/2$
			\item  \begin{equation}
				\label{eq:packing_set_size}
				\log | \cS | >  \frac{q\tau }{2}
			\end{equation}
		\end{itemize}
		Note that when $q=1$, $\cS = \cM_q$ satisfies these two requirements.
		
		Denote $\delta(S,S') = |S\cup S'| - |S\cap S'| = |S|+|S'| - 2 |S\cap S'|.$ 
		For 
		$S,S' \in \cS$ we have $\delta(S,S')= 2q - 2 |S\cap S'|.$
		Therefore for distinct $S,S' \in \cS$, as $|S\cap S'| < q/2,$
		\begin{align} 
			\label{eq:delta_bd}
			q < \delta(S,S') \leq 2q.
		\end{align}
		\item Consider the signals $\{ \beta_S : S \in \cS \}.$ 
		For any distinct $S,S'\in \cS$
		\begin{itemize}
			\item 
			
			
			From \autoref{lem:kl_lap_upper_bd},
			\begin{align}
				\kl (\lap(\beta_S, \sigma), \lap(\beta_{S'},\sigma)) 
				&= \delta(S, S') \kl (\lap(0, \sigma), \lap(a,\sigma))\\ 
				\label{eq:fano_kl_bd}
				&\leq 2q \cdot g(a/\sigma)
			\end{align}
			where \smash{$g(x) = e^{-x} + x - 1$} for \smash{$x\ge 0$}.

			\item $\| \beta_S - \beta_{S'} \|_2^2 = \delta(S,S') r^2 > q a^2$
			
		\end{itemize}
		
		\item 
		From Proposition 9 of \cite{birge2001gaussian} and the KL divergence bound 
		in 
		\eqref{eq:fano_kl_bd},
		\begin{align}
			n \cdot R_M \big( B(r, q) \big) \geq \frac{1}{4} q a^2 
			\left[ 1 - \left( \frac{2}{3} \vee \frac{2q g(a/\sigma)}{\log |\cS|} 
			\right)
			\right].
		\end{align}
		Applying the bound on \smash{$\log |\cS|$} from \eqref{eq:packing_set_size},
		\begin{equation}
			n \cdot R_M \big( B(r, q) \big) \geq \frac{1}{4} q a^2 
			\left[ 1 - \left( \frac{2}{3} \vee \frac{4 g(a/\sigma)}{\tau} \right)
			\right]
		\end{equation}
		By definition of $a$, $\frac{4 g(a/\sigma)}{\tau} \leq \frac{2}{3}.$ 
		Therefore
		\begin{equation}
			n \cdot R_M \big( B(r, q) \big) 
			\geq \frac{1}{12} q a^2
		\end{equation}
		Plugin the expression for $a$ and then for $\tau$ to arrive at the desired 
		bound.$\hfill\qedhere$
	\end{itemize}
\end{proof}

\subsection{Proof of \autoref{prop:lower_bd_2dgrids}}
\label{sec:lower_bd_2dgrids_proof}

\begin{proof}[Proof of \autoref{prop:lower_bd_2dgrids}]
We apply Le Cam's method to derive the lower bound.
Define $\beta^{(1)}, \beta^{(2)} \in \R^n$ as follows. 
$\beta^{(1)}_i = \beta^{(2)}_i = 1$ for all $i\in[n-1]$ and 
$\beta^{(1)}_n = 1+C_n/4, \beta^{(2)}_n = 1+C_n/2.$
Observe that 
\[
\frac{1}{n}\| \beta^{(1)} - \beta^{(2)} \|_2^2 = \frac{C_n^2}{16n}.
\]
Verify that $\beta^{(1)}, \beta^{(2)} \in \Theta(C_n).$ From equation (15.14) 
in 
\cite{wainwright2019high}, we can write
\begin{equation}
\label{eq:wainwright_15.14}
\inf_{\hat\beta} \sup_{\beta \in \Theta(C_n) } \E \| \hat\beta - \beta \|_n^2 
\;\geq 
\frac{C_n^2}{64n} \left( 1 - \| \P_1 - \P_2 \|_{\mathrm{TV}} \right)
\end{equation}
where $\P_j$ is the product distribution of $y_1,\dots,y_n$ with $y_i \sim 
\mathrm{Exp}(\mathrm{mean} = \beta^{(j)}_i)$ for $i\in[n]$.
We can calculate $\| \P_1 - \P_2 \|_{\mathrm{TV}}$ as follows.
\begin{align*}
\| \P_1 - \P_2 \|_{\mathrm{TV}} 
&= \frac{1}{2} \int \bigg| 
	p^{(1)}_1(x_1) p^{(1)}_2(x_2) \dots p^{(1)}_n (x_n) - 
	p^{(2)}_1(x_1) p^{(2)}_2(x_2) \dots p^{(2)}_n (x_n)  \bigg| \; dx \\
&= \frac{1}{2} \int 
	p^{(1)}_1(x_1) p^{(1)}_2(x_2) \dots p^{(1)}_{n-1} (x_{n-1})
	\big| p^{(1)}_n (x_n) - p^{(2)}_n (x_n) \big| \; dx_1 \dots dx_n \\
&= \frac{1}{2} \int \big| p^{(1)}_n (x_n) - p^{(2)}_n (x_n) \big| \; dx_n \\
&= \frac{1}{4}
\end{align*}
Here $p^{(j)}_i$ is the density of the exponential distribution with mean 
$\beta^{(j)}_i$ for $i\in[n].$
The second line above is true because $p^{(1)}_i = p^{(2)}_i$ for $i\in 
[n-1]$. 
The calculation for the last line is given in \autoref{lem:exp_dist_tv}. 
Plugging this back into \eqref{eq:wainwright_15.14}, we get the lower bound
\[
\inf_{\hat\beta} \sup_{\beta \in \Theta(C_n) } \E \| \hat\beta - \beta \|_n^2 
\;\geq 
\frac{3C_n^2}{256n}. \qedhere
\]
\end{proof}

\begin{lemma}
\label{lem:exp_dist_tv}
The total variation distance between two exponential distributions with means 
\smash{$\beta$} and \smash{$2\beta$} is \smash{$\frac{1}{4}$}, for any 
\smash{$\beta >0.$}
\end{lemma}
\begin{proof}[Proof of \autoref{lem:exp_dist_tv}]
The stated total variation distance is 
\begin{align*}
	\frac{1}{2}\int_0^\infty \big| \frac{1}{\beta} e^{-x/\beta} - 
\frac{1}{2\beta} e^{-x/2\beta} \big| \; dx 
	&= \frac{1}{2} \int_0^\infty | 2 e^{-2y} - e^{-y} | \; dy\\
	&= \frac{1}{2}\int_0^{\log 2} (2e^{-2y} - e^{-y}) \; dy 
			+ \frac{1}{2}\int_{\log 2}^\infty (e^{-y} - 2e^{-2y} ) 
\; dy\\
	&= \frac{1}{4}.
\end{align*}
In the first line, the variable is changed (\smash{$x\rightarrow 2\beta y$}).
\end{proof}

\section{Algorithmic details}
\label{sec:app-algor-deets}

This section expands on the algorithmic implementation for the MLE trend filter
described in \autoref{sec:algor-impl}. First, rewrite Equation \eqref{eq:mle1} 
(substituting $x$ for $\theta$) as
\begin{equation}
	\min_{Dx=z} \frac{1}{n}\sum \psi(x_i) - y_i x_i + \lambda \left\lVert z 
	\right\rVert_1.
\end{equation}
This is equivalent to \eqref{eq:mle1} but with additional variables.
The Lagrangian for this constrained minimization is given by
\begin{equation}
	\label{eq:app-admm-lagrangian}
	L(x,z,w) = \frac{1}{n}\sum \psi(x_i) - y_i x_i + \lambda \left\lVert z 
	\right\rVert_1
	+ w^{\top}(Dx-z),
\end{equation}
and the augmented Lagrangian is
\begin{equation}
	L_{\rho}(x,z,w) = \frac{1}{n}\sum \psi(x_i) - y_i x_i + \lambda \left\lVert z 
	\right\rVert_1
	+ w^{\top}(Dx-z) + \frac{\rho}{2}\left\lVert Dx-z\right\rVert_2^2.
\end{equation}
The augmented Lagrangian effectively adds a quadratic term that penalizes
infeasibility. So for any feasible solution with $Dx=z$, the augmented
Lagrangian will be equal to \eqref{eq:app-admm-lagrangian}.
Rather than this form, we instead use the ``scaled'' form for the augmented
Lagrangian, as it makes the update steps a little simpler. Defining $u=w/\rho$, 
then
the augmented Lagrangian becomes
\begin{equation}
	L_{\rho}(x,z,u) = \frac{1}{n}\sum \psi(x_i) - y_i x_i +
	\lambda \left\lVert z \right\rVert_1 + \frac{\rho}{2}\left\lVert Dx - z + 
	u\right\rVert_2^2
	- \frac{\rho}{2}\left\lVert u \right\rVert_2^2.
\end{equation}
\noindent
 The scaled ADMM algorithm iteratively solves this problem by minimizing over 
 $x$
then $z$ then a dual ascent update on $u$: 
\begin{align}
	x &\leftarrow \argmin_x \frac{1}{n}\sum \psi(x_i) - y_i x_i +
	\frac{\rho}{2}\left\lVert Dx-z+u\right\rVert_2^2,
	\label{eq:app-admm-x-update}\\ 
	z &\leftarrow \argmin_z \lambda \left\lVert z \right\rVert_1 +
	\frac{\rho}{2}\left\lVert Dx-z+u\right\rVert_2^2,\\ 
	u &\leftarrow u + Dx - z.
\end{align}
The $x$ update involves a matrix inversion which is best avoided when
$n$ is large. So we linearize that 
problem (the $x$ update 
only) around the current value $x^o$  
\begin{equation}
	\label{eq:app-linearized-x-update}
	x \leftarrow \argmin_x \frac{1}{n}\sum \psi(x_i) - y_i x_i +
	\rho \left(D^\top D x^o - D^\top z + D^\top u\right)^\top x +
	\frac{\mu}{2}\left\lVert x-x^o \right\rVert_2^2.
\end{equation}
To include the null space penalty,
the changes only impact the $x$ update. Therefore, \eqref{eq:app-admm-x-update} 
becomes
\begin{equation}
	x \leftarrow \argmin_x \frac{1}{n}\sum \psi(x_i) - y_i x_i +
	\frac{\rho}{2}\left\lVert Dx-z+u\right\rVert_2^2 + \lambda_2 \|
	P_\cN x \|_2, 
\end{equation}
and \eqref{eq:app-linearized-x-update} 
becomes
$$
x \leftarrow \argmin_x\frac{1}{n}\sum \psi(x_i) - y_i x_i + \rho \left(D^\top D 
x^o -
D^\top z + D^\top u\right)^\top x +\lambda_2 (g(x^o))^\top x +
\frac{\mu}{2}\left\lVert x-x^o \right\rVert_2^2. 
$$
where $g(v)$ is a subgradient of the function $v \mapsto \|P_\cN v\|_2$ given by
$ g(v) = \frac{P_\cN v}{\| P_\cN v\|_2}$ when $P_\cN v \neq 0$ and 
$g(v) = 0$ when $P_\cN v = 0$.

The $z$-update is easily shown to be given by elementwise soft-thresholding,
$$z_i\leftarrow \textrm{sign}(z_i)\left(|z_i| - (Dx-u)_i\right)_+;$$ 
and the $u$-update is 
simply vector addition. The $x$-update is potentially more challenging.
Note first that the $x$-update is the same for each $i$, so we can solve $n$
1-dimensional problems.  
The KKT stationarity condition requires
\begin{align}
	0 & =\left(\psi'(x_i) - y_i\right)  + \rho \left( D^\top \left(D x^o -
	z+u\right)\right)_i + \mu( x_i-x_i^o).\\
  \Longrightarrow  &\quad\psi'(x_i) + \mu x_i =  y_i  - \rho\left( D^\top D x^o - D^\top
  z+u\right)_i +  \mu x_i^o. 
\end{align}
Therefore, for any loss function as given by $\psi$, we want to solve
$
\psi'(x_i) + \mu x_i = b_i,
$
for each $i\in[n]$. For many functions $\psi$, the solution has a closed form. The
Binomial distribution with $\psi(x) = \log(1+e^x)$ is a family without a
simple solution, though standard root finding methods implemented in low-level
languages have no difficulties.
To include the nullspace penalty, the $x$ update changes slightly, but the
logic is the same.

\section{Degrees of freedom and tuning parameter selection}
\label{sec:app-comp-deets}

Here, we provide further details of the tuning parameter selection procedure
described in \autoref{sec:tuning-param-select}. 
If \smash{$Y \sim \mbox{N}(\theta^*, \sigma^2)$}, a now common method of risk
estimation makes use of Stein's Lemma.
\begin{lemma}[Stein's Lemma]
	  Assume $f(Y)$ is weakly differentiable with essentially
	  bounded weak partial derivatives on $\R^n$, then
	  \begin{equation}
		    \label{eq:8}
		    \trace \Cov(Y,f(Y)) = \Expect{\left\langle Y,\ f(Y)\right\rangle} = 
		\sigma^2\Expect{\trace Df(Y) \bigg\vert_y }.
		  \end{equation}
	\end{lemma}
The utility of this result comes from examining the decomposition of
the mean squared error of \smash{$\hat\theta(Y)$} as an estimator of 
\smash{$\theta^*$}.
\begin{align}
	\Expect{\snorm{\theta^*-\hat\theta(Y)}_2^2}
	&= \Expect{\snorm{Y-\hat\theta(Y)}_2^2} -n\sigma^2 + 2
	\trace\Cov(Y,\hat\theta(Y))\\
	&= \Expect{\snorm{Y-\hat\theta(Y)}_2^2} -n\sigma^2 + 2\sigma^2
	\Expect{\trace J\hat\theta(z) \big\vert_Y} .
\end{align}
This characterization motivates the definition of degrees-of-freedom
for linear predictors ($\textrm{df} :=\frac{1}{\sigma^2} \trace
J\hat\theta(z)\big\vert_y$)~\citep{Efron1986}, where 
\smash{$\hat\theta(y)=Hy$}. Using
Stein's Lemma, assuming \smash{$\sigma^2$} is known, we have Stein's Unbiased
Risk Estimator
\begin{equation}
	\mathrm{SURE}(\hat\theta) = \snorm{y-\hat\theta}_2^2 -n\sigma^2 + 
	2\sigma^2\trace\left( J\hat\theta(z) \big\vert_y\right), 
\end{equation}
which satisfies \smash{$\Expect{\textrm{SURE}(\hat\theta)} =
\Expect{\snorm{\theta^*-\hat\theta(Y)}_2^2}$}. 
Note that this is the risk for estimating the \smash{$n$}-dimensional parameter
\smash{$\theta^*$}.  
The following result generalizes this idea to certain continuous
exponential families.
\begin{lemma}[Generalized Stein Lemma; \citealp{Eldar2009}]
	\label{lem:g-stein-lemma}
	Assume \smash{$\hat\theta(y)$} is weakly differentiable in \smash{$y$} with 
	essentially
	bounded weak partial derivatives on \smash{$\R^n$}. Let \smash{$Y$} be 
	distributed
	according to a natural exponential family
	and assume that the base measure \smash{$h$} is weakly differentiable. Then,
	\begin{equation}
		\label{eq:12}
		\Expect{\theta^{*\top} \hat\theta(Y)} = - \Expect{
			\left\langle\frac{\nabla h(Y)}{h(Y)},\ \hat\theta(Y)\right\rangle
			+ \trace J\hat\theta(y) \big\vert_Y}.
	\end{equation}
	Note that \smash{$\nabla h(Y)$} here means the vector \smash{$[d/dy\ h(y)
	\vert_{y_i}]$} and \smash{$h(Y)$} means the vector \smash{$[h(y_i)]$}.
\end{lemma}

Therefore we define the Generalized SURE~\citep{Eldar2009} along the lines of
the multivariate Gaussian case. 
\begin{lemma}
	  \label{lem:gsure}
	  Assume $h$ is weakly differentiable, 
	  $\hat\theta(y)$ is weakly differentiable with essentially bounded partial
	  derivatives. Then
	  \begin{equation}
		    \label{eq:13}
		    \mathrm{SURE}(\hat\theta) = \norm{\hat\theta(y)}^2_2 +
        2 \left\langle\frac{\nabla
            h(y)}{h(y)},\ \hat\theta(y)\right\rangle + 2\trace\left( 
          J\hat\theta(z) \bigg\vert_y\right) + 
		    \frac{1}{h(y)} \trace \frac{\partial^2 h(z)}{\partial z^2}\bigg\vert_y
		  \end{equation}
	  is an unbiased estimator for the MSE of an estimator $\hat\theta(Y)$ of
	  $\theta$: $\Expect{\norm{\hat\theta(Y) - \theta}_2^2}$.
	\end{lemma}

\begin{proof}
	  We have
	  \begin{align}
		    \label{eq:14}
		    \Expect{\norm{f(Y) - \theta(\beta)}_2^2}
		    &= \Expect{\norm{f(Y)}_2^2} + \Expect{\norm{\theta}_2^2} -
		      2\Expect{\left\langle \theta(\beta),\ f(Y)\right\rangle}.
		  \end{align}
      Now, the first term is a function of the data only, and to the last term, 
      we simply apply
	  \autoref{lem:g-stein-lemma}. For the second term,
	  \begin{align}
		    \Expect{\norm{\theta}_2^2}
		    &=\Expect{\langle \theta,\ \theta \rangle}
		     = -\Expect{\left\langle\frac{\nabla h(Y)}{h(Y)},\ 
		\theta\right\rangle}\\
		    &= \Expect{\left\langle\frac{\nabla h(Y)}{h(Y)},\
			      \frac{\nabla h(Y)}{h(Y)}\right\rangle} + \Expect{\trace
			      \frac{\partial}{\partial y} \frac{\nabla h(y)}{h(y)}
			      \bigg\vert_Y}\\
		    &= \Expect{\frac{\norm{\nabla h(Y)}_2^2}{h(Y)^2}} + \Expect{\trace
			      \frac{\norm{\nabla h(Y)}_2^2 +h(Y) \partial^2/\partial y^2 
				h(y)\big\vert_Y}
			      {h(Y)^2}}\\
		    &=\Expect{ \frac{1}{h(Y)} \trace \frac{\partial^2 h(y)}{\partial 
				y^2}\bigg\vert_Y},
		  \end{align}
	  by applying \autoref{lem:g-stein-lemma} twice along with the
	  quotient rule.
	\end{proof}

However, we would prefer to estimate the
Kullback-Leibler Divergence between the density under
$\theta=\hat\theta(y)$ and that under $\theta=\theta^*$.
For
exponential families,
\begin{equation}
	\Expect{\KL{\hat\theta(Y)}{\theta^*}}
	= \Expect{\left\langle \hat\theta(Y)-\theta^*,\ \hat\beta(Y)\right\rangle +
		\psi(\theta^*) -\psi\left( \hat\theta(Y) \right) },
\end{equation}
and, an application of \autoref{lem:g-stein-lemma} provides an
unbiased estimator of this quantity. The result is given in \autoref{lem:sukls}
in the main body.

Finally, we conclude this section with the proof of \autoref{thm:simple-divergence}.

\begin{proof}[Proof of \autoref{thm:simple-divergence}]
  The proof follows from \citet[Theorem 2]{VaiterDeledalle2017}. We have
  \begin{align}
    X_T &= \Pnd\\
    \nabla^2 F_0(\hat\mu(y),y) &= \diag\left( \psi''(\hat\theta) \right)\\
    \mathfrak{A}_\beta &= 0\\
    \nabla^2_{\mathcal{M}} J\left( \hat\beta(y) \right) &= \lambda_2P_\cN\\
    D(\nabla F_0)(\hat\mu(y),y) &= \diag\left( \psi''(\hat\theta) 
    \right).\qedhere
  \end{align}
\end{proof}

%% file: stef-arxiv.bbl
\begin{thebibliography}{50}
\expandafter\ifx\csname natexlab\endcsname\relax\def\natexlab#1{#1}\fi
\expandafter\ifx\csname url\endcsname\relax
  \def\url#1{\texttt{#1}}\fi
\expandafter\ifx\csname urlprefix\endcsname\relax\def\urlprefix{URL: }\fi

\bibitem[{Baby and Wang(2021)}]{baby2021optimal}
Baby, D. and Wang, Y.-X. (2021) Optimal dynamic regret in exp-concave online
  learning.
\newblock In \textit{Proceedings of Thirty Fourth Conference on Learning
  Theory} (eds. M.~Belkin and S.~Kpotufe), vol. 134 of \textit{Proceedings of
  Machine Learning Research}, 359--409.

\bibitem[{Barbero and Sra(2018)}]{barbero2014modular}
Barbero, A. and Sra, S. (2018) Modular proximal optimization for
  multidimensional total-variation regularization.
\newblock \textit{Journal of Machine Learning Research}, \textbf{19},
  2232--2313.

\bibitem[{Bassett and Sharpnack(2019)}]{Bassett2019fused}
Bassett, R. and Sharpnack, J. (2019) Fused density estimation: Theory and
  methods.
\newblock \textit{Journal of Royal Statistical Society, Series B}, \textbf{81},
  839--860.

\bibitem[{Birge and Massart(2001)}]{birge2001gaussian}
Birge, L. and Massart, P. (2001) Gaussian model selection.
\newblock \textit{Journal of the European Mathematical Society}, \textbf{3},
  203--268.

\bibitem[{Brown(1986)}]{Brown1986}
Brown, L.~D. (1986) \textit{Fundamentals of statistical exponential families
  with applications in statistical decision theory}, vol.~9 of \textit{Lecture
  Notes-Monograph Series}.
\newblock Institute of Mathematical Statistics.

\bibitem[{Brown et~al.(2010)Brown, Cai and Zhou}]{brown2010nonparametric}
Brown, L.~D., Cai, T.~T. and Zhou, H.~H. (2010) {Nonparametric regression in
  exponential families}.
\newblock \textit{The Annals of Statistics}, \textbf{38}, 2005--2046.

\bibitem[{Chatterjee and Goswami(2021)}]{chatterjee2021newrisk}
Chatterjee, S. and Goswami, S. (2021) New risk bounds for {2D} total variation
  denoising.
\newblock \textit{IEEE Transactions on Information Theory}, \textbf{67},
  4060--4091.

\bibitem[{Condat(2013)}]{condat2012direct}
Condat, L. (2013) A direct algorithm for {1-D} total variation denoising.
\newblock \textit{IEEE Signal Processing Letters}, \textbf{20}, 1054--1057.

\bibitem[{Deledalle(2017)}]{Deledalle2017}
Deledalle, C.-A. (2017) Estimation of {K}ullback-{L}eibler losses for noisy
  recovery problems within the exponential family.
\newblock \textit{Electronic Journal of Statistics}, \textbf{11}, 3141---3164.

\bibitem[{Efron(1986)}]{Efron1986}
Efron, B. (1986) How biased is the apparent error rate of a prediction rule?
\newblock \textit{Journal of the American Statistical Association},
  \textbf{81}, 461--470.

\bibitem[{Eldar(2009)}]{Eldar2009}
Eldar, Y.~C. (2009) Generalized {SURE} for exponential families: Applications
  to regularization.
\newblock \textit{IEEE Transactions on Signal Processing}, \textbf{57},
  471--481.

\bibitem[{van~de Geer(2020)}]{vandeGeer2020logistic}
van~de Geer, S. (2020) Logistic regression with total variation regularization.
\newblock \textit{Transactions of A.\ Razmadze Mathematical Institute},
  \textbf{174}, 217 -- 233.

\bibitem[{Guntuboyina et~al.(2020)Guntuboyina, Lieu, Chatterjee and
  Sen}]{guntuboyina2017adaptive}
Guntuboyina, A., Lieu, D., Chatterjee, S. and Sen, B. (2020) Adaptive risk
  bounds in univariate total variation denoising and trend filtering.
\newblock \textit{Annals of Statistics}, \textbf{48}, 205--229.

\bibitem[{Hansen et~al.(2012)Hansen, Sato and Ruedy}]{hansen_perception_2012}
Hansen, J., Sato, M. and Ruedy, R. (2012) Perception of climate change.
\newblock \textit{Proceedings of the National Academy of Sciences},
  \textbf{109}, E2415--E2423.

\bibitem[{Harchaoui and Levy-Leduc(2010)}]{harchaoui2010multiple}
Harchaoui, Z. and Levy-Leduc, C. (2010) Multiple change-point estimation with a
  total variation penalty.
\newblock \textit{Journal of the American Statistical Association},
  \textbf{105}, 1480--1493.

\bibitem[{Huntingford et~al.(2013)Huntingford, Jones, Livina, Lenton and
  Cox}]{huntingford_no_2013}
Huntingford, C., Jones, P.~D., Livina, V.~N., Lenton, T.~M. and Cox, P.~M.
  (2013) No increase in global temperature variability despite changing
  regional patterns.
\newblock \textit{Nature}, \textbf{500}, 327--330.

\bibitem[{H{\"u}tter and Rigollet(2016)}]{hutter2016optimal}
H{\"u}tter, J.-C. and Rigollet, P. (2016) Optimal rates for total variation
  denoising.
\newblock In \textit{29th Annual Conference on Learning Theory} (eds.
  V.~Feldman, A.~Rakhlin and O.~Shamir), vol.~49 of \textit{Proceedings of
  Machine Learning Research}, 1115--1146.

\bibitem[{Johnson(2013)}]{johnson2013dynamic}
Johnson, N. (2013) A dynamic programming algorithm for the fused lasso and
  {$L_0$}-segmentation.
\newblock \textit{Journal of Computational and Graphical Statistics},
  \textbf{22}, 246--260.

\bibitem[{Kakade et~al.(2010)Kakade, Shamir, Sridharan and
  Tewari}]{KakadeShamir2010}
Kakade, S., Shamir, O., Sridharan, K. and Tewari, A. (2010) Learning
  exponential families in high-dimensions: Strong convexity and sparsity.
\newblock In \textit{Proceedings of the Thirteenth International Conference on
  Artificial Intelligence and Statistics} (eds. Y.~W. Teh and M.~Titterington),
  vol.~9 of \textit{Proceedings of Machine Learning Research}, 381--388.

\bibitem[{Khodadadi and McDonald(2019)}]{KhodadadiMcDonald2018}
Khodadadi, A. and McDonald, D.~J. (2019) Algorithms for estimating trends in
  global temperature volatility.
\newblock In \textit{Proceedings of the 33rd {AAAI} Conference on Artificial
  Intelligence} (eds. P.~V. Hentenryck and Z.-H. Zhou), vol.~33 of
  \textit{Association for the Advancement of Artificial Intelligence},
  614--621.

\bibitem[{Kim et~al.(2009)Kim, Koh, Boyd and Gorinevsky}]{KimKoh2009}
Kim, S.-J., Koh, K., Boyd, S. and Gorinevsky, D. (2009) $\ell_1$ trend
  filtering.
\newblock \textit{SIAM Review}, \textbf{51}, 339--360.

\bibitem[{Lin et~al.(2017)Lin, Sharpnack, Rinaldo and
  Tibshirani}]{lin2017sharp}
Lin, K., Sharpnack, J.~L., Rinaldo, A. and Tibshirani, R.~J. (2017) A sharp
  error analysis for the fused lasso, with application to approximate
  changepoint screening.
\newblock In \textit{Advances in Neural Information Processing Systems} (eds.
  I.~Guyon, U.~V. Luxburg, S.~Bengio, H.~Wallach, R.~Fergus, S.~Vishwanathan
  and R.~Garnett), vol.~30. Curran Associates, Inc.

\bibitem[{Madrid~Padilla and Chatterjee(2021)}]{padilla2021risk}
Madrid~Padilla, O.~H. and Chatterjee, S. (2021) {Risk Bounds for Quantile Trend
  Filtering}.
\newblock \textit{Biometrika}, forthcoming.

\bibitem[{Madrid~Padilla et~al.(2020)Madrid~Padilla, Sharpnack, Chen and
  Witten}]{madrid2020adaptive}
Madrid~Padilla, O.~H., Sharpnack, J., Chen, Y. and Witten, D.~M. (2020)
  Adaptive nonparametric regression with the k-nearest neighbour fused lasso.
\newblock \textit{Biometrika}, \textbf{107}, 293--310.

\bibitem[{Madrid~Padilla et~al.(2018)Madrid~Padilla, Sharpnack, Scott and
  Tibshirani}]{padilla2018dfs}
Madrid~Padilla, O.~H., Sharpnack, J., Scott, J.~G. and Tibshirani, R.~J. (2018)
  The {DFS} fused lasso: {L}inear-time denoising over general graphs.
\newblock \textit{Journal of Machine Learning Research}, \textbf{18}, 1--36.

\bibitem[{Mammen and van~de Geer(1997)}]{mammen1997locally}
Mammen, E. and van~de Geer, S. (1997) Locally adaptive regression splines.
\newblock \textit{Annals of Statistics}, \textbf{25}, 387--413.

\bibitem[{McCullagh and Nelder(1989)}]{mccullagh1989generalized}
McCullagh, P. and Nelder, J.~A. (1989) \textit{Generalized Linear Models}.
\newblock Boca Raton, FL: Chapman and Hall, 2nd edn.

\bibitem[{Meyer(2021)}]{meyer2021alternative}
Meyer, G.~P. (2021) An alternative probabilistic interpretation of the {Huber}
  loss.
\newblock In \textit{Proceedings of the IEEE/CVF Conference on Computer Vision
  and Pattern Recognition (CVPR)}, 5261--5269.

\bibitem[{{Nu\~{n}o} et~al.(2021){Nu\~{n}o}, Garc\`{i}a, Rajasekar, Pinheiro
  and Schmidt}]{nuno2021covid}
{Nu\~{n}o}, M., Garc\`{i}a, Y., Rajasekar, G., Pinheiro, D. and Schmidt, A.~J.
  (2021) {COVID}-19 hospitalizations in five {C}alifornia hospitals: {A}
  retrospective cohort study.
\newblock \textit{BMC Infectious Diseases}, \textbf{21}, 938.

\bibitem[{Ortelli and van~de Geer(2020)}]{ortelli2020adaptive}
Ortelli, F. and van~de Geer, S. (2020) Adaptive rates for total variation image
  denoising.
\newblock \textit{Journal of Machine Learning Research}, \textbf{247}, 1--38.

\bibitem[{Ortelli and van~de Geer(2021)}]{ortelli2019prediction}
--- (2021) {Prediction bounds for higher order total variation regularized
  least squares}.
\newblock \textit{The Annals of Statistics}, \textbf{49}, 2755--2773.

\bibitem[{Poli et~al.(2016)Poli, Hersbach, Dee, Berrisford, Simmons, Vitart,
  Laloyaux, Tan, Peubey, Th{\'e}paut, Tr{\'e}molet, H{\'o}lm, Bonavita, Isaksen
  and Fisher}]{PoliHersbach2016}
Poli, P., Hersbach, H., Dee, D.~P., Berrisford, P., Simmons, A.~J., Vitart, F.,
  Laloyaux, P., Tan, D. G.~H., Peubey, C., Th{\'e}paut, J.-N., Tr{\'e}molet,
  Y., H{\'o}lm, E.~V., Bonavita, M., Isaksen, L. and Fisher, M. (2016)
  {ERA-20C}: {A}n atmospheric reanalysis of the twentieth century.
\newblock \textit{Journal of Climate}, \textbf{29}, 4083--4097.

\bibitem[{Prasad et~al.(2020)Prasad, Suggala, Balakrishnan and
  Ravikumar}]{prasad2020robust}
Prasad, A., Suggala, A.~S., Balakrishnan, S. and Ravikumar, P. (2020) {Robust
  estimation via robust gradient estimation}.
\newblock \textit{Journal of the Royal Statistical Society Series B},
  \textbf{82}, 601--627.

\bibitem[{Ramdas and Tibshirani(2016)}]{RamdasTibshirani2016}
Ramdas, A. and Tibshirani, R.~J. (2016) Fast and flexible {ADMM} algorithms for
  trend filtering.
\newblock \textit{Journal of Computational and Graphical Statistics},
  \textbf{25}, 839--858.

\bibitem[{Rinaldo(2009)}]{rinaldo2009properties}
Rinaldo, A. (2009) Properties and refinements of the fused lasso.
\newblock \textit{Annals of Statistics}, \textbf{37}, 2922--2952.

\bibitem[{Rudin et~al.(1992)Rudin, Osher and Faterni}]{rudin1992nonlinear}
Rudin, L.~I., Osher, S. and Faterni, E. (1992) Nonlinear total variation based
  noise removal algorithms.
\newblock \textit{Physica {D}: Nonlinear Phenomena}, \textbf{60}, 259--268.

\bibitem[{Sadhanala et~al.(2021)Sadhanala, Wang, Hu and
  Tibshirani}]{sadhanala2021multivariate}
Sadhanala, V., Wang, Y.-X., Hu, A. and Tibshirani, R. (2021) Multivariate trend
  filtering on lattice data.
\newblock \urlprefix\url{http://arxiv.org/abs/2112.14758}.

\bibitem[{Sadhanala et~al.(2017)Sadhanala, Wang, Sharpnack and
  Tibshirani}]{SadhanalaWang2017}
Sadhanala, V., Wang, Y.-X., Sharpnack, J.~L. and Tibshirani, R.~J. (2017)
  Higher-order total variation classes on grids: Minimax theory and trend
  filtering methods.
\newblock In \textit{Advances in Neural Information Processing Systems},
  vol.~30, 5800--5810.

\bibitem[{Sadhanala et~al.(2016)Sadhanala, Wang and
  Tibshirani}]{sadhanala2016total}
Sadhanala, V., Wang, Y.-X. and Tibshirani, R.~J. (2016) Total variation classes
  beyond 1d: {Minimax} rates, and the limitations of linear smoothers.
\newblock In \textit{Advances in Neural Information Processing Systems} (eds.
  D.~Lee, M.~Sugiyama, U.~Luxburg, I.~Guyon and R.~Garnett), vol.~29. Curran
  Associates, Inc.

\bibitem[{Steidl et~al.(2006)Steidl, Didas and Neumann}]{steidl2006splines}
Steidl, G., Didas, S. and Neumann, J. (2006) Splines in higher order {TV}
  regularization.
\newblock \textit{International Journal of Computer Vision}, \textbf{70},
  214--255.

\bibitem[{Tibshirani et~al.(2005)Tibshirani, Saunders, Rosset, Zhu and
  Knight}]{tibshirani2005sparsity}
Tibshirani, R., Saunders, M., Rosset, S., Zhu, J. and Knight, K. (2005)
  Sparsity and smoothness via the fused lasso.
\newblock \textit{Journal of the Royal Statistical Society: Series B},
  \textbf{67}, 91--108.

\bibitem[{Tibshirani(2014)}]{tibshirani2014adaptive}
Tibshirani, R.~J. (2014) Adaptive piecewise polynomial estimation via trend
  filtering.
\newblock \textit{Annals of Statistics}, \textbf{42}, 285--323.

\bibitem[{Tibshirani(2022)}]{tibshirani2020divided}
--- (2022) Divided differences, falling factorials, and discrete splines:
  Another look at trend filtering and related problems.
\newblock \textit{Foundations and Trends in Machine Learning}, \textbf{15},
  694--846.

\bibitem[{Tsybakov(2009)}]{tsybakov2009introduction}
Tsybakov, A.~B. (2009) \textit{Introduction to Nonparametric Estimation}.
\newblock Springer.

\bibitem[{Vaiter et~al.(2017)Vaiter, Deledalle, Fadili, Peyr{\'e} and
  Dossal}]{VaiterDeledalle2017}
Vaiter, S., Deledalle, C., Fadili, J., Peyr{\'e}, G. and Dossal, C. (2017) The
  degrees of freedom of partly smooth regularizers.
\newblock \textit{Annals of the Institute of Statistical Mathematics},
  \textbf{69}, 791--832.

\bibitem[{Vasseur et~al.(2014)Vasseur, DeLong, Gilbert, Greig, Harley, McCann,
  Savage, Tunney and O{\textquoteright}Connor}]{VasseurDeLong2014}
Vasseur, D.~A., DeLong, J.~P., Gilbert, B., Greig, H.~S., Harley, C. D.~G.,
  McCann, K.~S., Savage, V., Tunney, T.~D. and O{\textquoteright}Connor, M.~I.
  (2014) Increased temperature variation poses a greater risk to species than
  climate warming.
\newblock \textit{Proceedings of the Royal Society of London B: Biological
  Sciences}, \textbf{281}.

\bibitem[{Vershynin(2018)}]{Vershynin2018}
Vershynin, R. (2018) \textit{High-Dimensional Probability}.
\newblock Cambridge, UK: Cambridge University Press.

\bibitem[{Wainwright(2019)}]{wainwright2019high}
Wainwright, M.~J. (2019) \textit{High-Dimensional Statistics: A Non-Asymptotic
  Viewpoint}.
\newblock Cambridge Series in Statistical and Probabilistic Mathematics.
  Cambridge University Press.

\bibitem[{Wainwright and Jordan(2008)}]{wainwright2008graphical}
Wainwright, M.~J. and Jordan, M.~I. (2008) Graphical models, exponential
  families, and variational inference.
\newblock \textit{Foundations and Trends in Machine Learning}, \textbf{1},
  1--305.

\bibitem[{Wang et~al.(2016)Wang, Sharpnack, Smola and
  Tibshirani}]{WangSharpnack2016}
Wang, Y.-X., Sharpnack, J., Smola, A.~J. and Tibshirani, R.~J. (2016) Trend
  filtering on graphs.
\newblock \textit{Journal of Machine Learning Research}, \textbf{17}, 1--41.

\end{thebibliography}
